\documentclass[11pt]{amsart}
\usepackage{macro}
\usepackage{booktabs}

\newcommand{\Iwalg}{\Z_p[[\tilde{\Gamma}]]}
\newcommand{\unIw}{\mathcal{O}^{\mathrm ur}[[\tilde{\Gamma}]]}
\newcommand{\Sstar}{{(\star)}}
\usepackage{soul}
\newcommand{\Kinf}{\widetilde{K}_\infty}

\newcommand{\One}{{\mathrm{\uppercase\expandafter{\romannumeral 1}}}}
\newcommand{\Two}{{\mathrm{\uppercase\expandafter{\romannumeral 2}}}}
\newcommand{\rhodn}{{\rho_{{d},n}}}

\newcommand{\rhodnstar}{{\rho^\star_{{d},n}}}
\newcommand{\rhoDN}[2]{{\rho_{{#1},#2}}}
\newcommand{\rhoDNstar}[2]{{\rho^\star_{{#1},#2}}}

\newcommand{\XXX}{\mathfrak{X}}
\newcommand{\Qp}{\Q_p}
\newcommand{\rank}{\mathrm{rank}}
\newcommand{\Ann}{\mathrm{Ann}}

\newcommand{\Gr}{\mathrm{Gr}}

\newcommand{\Drho}{D_{\rhoDN{4}{2}}}
\newcommand{\tDN}[3]{\theta_{{#1},#2}^{#3}}

\newcommand{\fr}{\mathfrak{r}}
\newcommand{\ur}{\mathrm{ur}}

\newcommand{\cmmnt}[1]{}

\newcommand{\HIw}{H^1_{\rm Iw}}

\title{Codimension two cycles in Iwasawa theory and elliptic curves with supersingular reduction}



\author{Antonio Lei}
\address{(Lei) D\'epartement de math\'ematiques et de statistique, Universit\'e Laval, Pavillon Alexandre-Vachon, 1045 avenue de la M\'edecine, Qu\'ebec, Canada, G1V 0A6}
\email{antonio.lei@mat.ulaval.ca}

\author{Bharathwaj Palvannan}
\address{(Palvannan) Department of Mathematics, University of Pennsylvania, 4W1 DRL, 209 South 33rd Street,  Philadelphia, USA  19104-6395}
\email{pbharath@math.upenn.edu}

\thanks{The first named author's research is supported by the NSERC Discovery Grants Program 05710.}

\begin{document}
\maketitle

\begin{abstract}
A  result of Bleher, Chinburg, Greenberg, Kakde, Pappas, Sharifi and Taylor has initiated the topic of higher codimension Iwasawa theory. As a generalization of the classical Iwasawa main conjecture, they prove a relationship between analytic objects (a pair of Katz's $2$-variable $p$-adic $L$-functions) and algebraic objects (two ``everywhere unramified'' Iwasawa modules) involving codimension two cycles in a $2$-variable Iwasawa algebra. We prove a result by considering the restriction to an imaginary quadratic field $K$ (where an odd prime $p$ splits) of an elliptic curve $E$, defined over $\Q$, with good supersingular reduction at $p$. On the analytic side, we consider eight  pairs of  $2$-variable $p$-adic $L$-functions in this setup (four of the  $2$-variable $p$-adic $L$-functions have been constructed by Loeffler and a fifth $2$-variable $p$-adic $L$-function is due to Hida). On the algebraic side, we consider modifications of fine Selmer groups over the $\Z_p^2$-extension of $K$. We also provide numerical evidence, using algorithms of Pollack, towards a pseudo-nullity conjecture of Coates-Sujatha.
\end{abstract}
\thispagestyle{empty}
\tableofcontents

\section{Introduction}
Fix an odd prime $p$. Let $\RRR$ denote a Noetherian, complete, integrally closed, local domain of characteristic zero with Krull dimension $n+1$ and whose residue field has characteristic $p$. To a continuous Galois representation\footnote{The first subscript $d$ of $\rhodn$ indicates the dimension of the Galois representation, while the second subscript $n$ denotes a number one less than the Krull dimension of the ring $\RRR$. In the settings we are interested in, the number $n$ would denote the number of variables in the corresponding $p$-adic $L$-functions.}
\begin{align*}
  \rhodn: \Gal{\overline{\Q}}{\Q} \rightarrow \Gl_d(\RRR),
\end{align*}
satisfying the ``Panchishkin condition''\footnote{The Panchishkin condition is a type of ``ordinariness'' assumption, introduced by Greenberg, while formulating the Iwasawa main conjecture for Galois deformations. See Section 4 in \cite{greenberg1994iwasawa} for the precise definition.}, Ralph Greenberg \cite{greenberg1994iwasawa} has formulated a main conjecture in Iwasawa theory. The Iwasawa main conjecture provides a relation involving codimension one cycles in the divisor group of the ring $\RRR$, relating a $p$-adic $L$-function, satisfying suitable interpolation properties, to a Selmer group. The divisor group, denoted $Z^1(\RRR)$, is the free abelian group on the set of height $1$ prime ideals of the ring $\RRR$. \\

One could consider $Z^2(\RRR)$, the free abelian group on the set of height $2$ prime ideals of the ring $\RRR$. Many standard conjectures in Iwasawa theory predict that pseudo-null modules are
ubiquitous.  For example, see Conjecture 3.5 in Greenberg's article \cite{MR1846466}  and Conjecture B in the work of Coates-Sujatha \cite{MR2148798}. These pseudo-null $\RRR$-modules are supported in codimension at least two. One desirable extension of the Iwasawa main conjecture is an answer to the following question:
\begin{Question} \label{question-analytic}
Can we use codimension two cycles from $Z^2(\RRR)$ to associate analytic invariants  to pseudo-null modules in Iwasawa theory?
\end{Question}
The crucial insight in the seven-author paper \cite{bleher2015higher}, to obtain such an association, is to study a situation when the Galois representation $\rhodn$ satisfies two distinct Panchishkin conditions. We will use this insight and obtain a result by considering the restriction to an imaginary quadratic field, where an odd prime $p$ splits, of an elliptic curve defined over $\Q$ with good supersingular reduction at $p$. \\

Let $K$ denote an imaginary quadratic field where the prime $p$ splits. Let $\p$ and $\q$ denote the two prime ideals in $K$ containing $p$. We fix an isomorphism $i : \overline{\Q}_p \cong \mathbb{C}$ along with embeddings $K \hookrightarrow \overline{\Q}$, $\overline{\Q} \hookrightarrow \overline{\Q}_p$ and $\overline{\Q} \hookrightarrow \overline{\Q}_p \stackrel{i}{\cong}  \mathbb{C}$. The embedding $i_\p: K \hookrightarrow \overline{\Q}_p$ fixes a prime ideal, say $\p$, in $K$ lying above $p$. Let $G_\Q$ and $G_K$ denote the absolute Galois groups of $\Q$ and $K$ respectively. Let $\widetilde{K}_\infty$ denote the compositum of all the $\Z_p$-extensions of $K$. Let $\Q_{\Cyc}$ and $K_{\Cyc}$ denote the cyclotomic $\Z_p$-extensions of $\Q$ and $K$ respectively. Let $K(\p^\infty)_{\Z_p}$ denote the unique $\Z_p$-extension of $K$ that is unramified outside $\p$. Let $\widetilde{\Gamma}$, $\Gamma_{\Cyc}$ and $\Gamma_\p$  denote the Galois groups $\Gal{\widetilde{K}_\infty}{K}$, $\Gal{\Q_\Cyc}{\Q}$ and $\Gal{K(\p^\infty)_{\Z_p}}{K}$  respectively. To summarize, we have the following isomorphisms of topological groups:
\begin{align*}
  \widetilde{\Gamma} = \Gal{\widetilde{K}_\infty}{K} \cong \Z_p^2, \quad \Gamma_{\Cyc} \cong \Gal{K_\Cyc}{K} \cong \Z_p, \quad \Gamma_\p = \Gal{K(\p^\infty)_{\Z_p}}{K} \cong \Z_p.
\end{align*}

The natural restriction maps $\widetilde{\Gamma} \twoheadrightarrow \Gamma_\Cyc$ and $\widetilde{\Gamma} \twoheadrightarrow \Gamma_\p$  provide us the following isomorphism of topological groups: $$\widetilde{\Gamma} \cong \Gamma_\Cyc \times \Gamma_\p.$$
We shall consider the Iwasawa algebras $\Z_p[[\widetilde{\Gamma}]]$, $\Z_p[[\Gamma_\Cyc]]$ and $\Z_p[[\Gamma_\p]]$. Note that we have the following topological ring isomorphisms involving power series rings over $\Z_p$:
\begin{align*}
  \Z_p[[\widetilde{\Gamma}]] \cong \Z_p[[x_1,x_2]], \qquad \Z_p[[\Gamma_\Cyc]] \cong \Z_p[[x_{\Cyc}]], \qquad \Z_p[[\Gamma_\p]] \cong \Z_p[[x_\p]].
\end{align*}
We shall also consider the following tautological characters:
\begin{align*}
  \tilde{\kappa}: G_K \twoheadrightarrow \widetilde{\Gamma} \hookrightarrow \Gl_1\left(\Z_p[[\widetilde{\Gamma}]]\right), \quad \kappa_{\Cyc}: G_\Q \twoheadrightarrow \Gamma_\Cyc \hookrightarrow \Gl_1\left(\Z_p[[\Gamma_{\Cyc}]]\right), \quad \kappa_\p: G_K \twoheadrightarrow \Gamma_\p \hookrightarrow \Gl_1\left(\Z_p[[\Gamma_\p]]\right).
\end{align*}

Let $E$ denote an elliptic curve defined over $\Q$ with good supersingular reduction at $p$ with $a_p(E)=0$. Let $f_E$ denote the weight two cuspidal newform associated to the elliptic curve $E$. The $p$-adic Tate module, denoted $T_p(E)$, has a natural action of the Galois group $G_\Q$.  We let $\Res_{G_K}T_p(E)$ denote the  restriction of $T_p(E)$ to the Galois group $G_K$. We have a four dimensional Galois representation:
$$\rhoDN{4}{2}: G_\Q \rightarrow \Gl_4\left(\Z_p[[\tilde{\Gamma}]]\right)$$
given by the action of $G_\Q$ on the following free $\Z_p[[\tilde{\Gamma}]]$-module of rank four:
\begin{align*}
  T_{\rhoDN{4}{2}}:=T_p(E) \ \hotimes_{\Z_p} \ \Ind^{G_\Q}_{G_K} \left(\Z_p[[\Gamma_\p]] (\kappa_\p^{-1}) \right) \ \hotimes_{\Z_p} \ \Z_p[[\Gamma_\Cyc]](\kappa_{\Cyc}^{-1}).
\end{align*}
Here, $\hotimes_{\Z_p}$ denotes the completed tensor product over $\Z_p$. We will also consider the following discrete $\Z_p[[\widetilde{\Gamma}]]$-module:
\begin{align*}
  D_{\rhoDN{4}{2}}:= T_{\rhoDN{4}{2}} \otimes_{\Z_p[[\widetilde{\Gamma}]]} \Hom_{\cont}\left(\Z_p[[\widetilde{\Gamma}]],\frac{\Q_p}{\Z_p}\right).
\end{align*}
Let $\rho_{K}$ denote the two-dimensional Galois representation\footnote{Our convention is non-standard in that the kernel of a homomorphism $\varphi$ in $\Hom_{\cont}\left(\Z_p[[\Gamma_\p]],\overline{\Q}_p\right) = \Hom_{\cont}\left(\Gamma_\p , \overline{\Q}_p^\times\right)$ corresponds to a weight $k$ specialization if $\varphi$ is the $p$-adic character associated to an algebraic Hecke character of $K$, with conductor equal to a power of $\p$ and of infinity type $(1-k)i \circ i_\p$}, defined over the ring $\Z_p[[\Gamma_\p]]$, given by the action of $G_\Q$ on  $\Ind^{G_\Q}_{G_K} \left(\Z_p[[\Gamma_\p]] (\kappa_\p^{-1}) \right)$. \\

The Galois representation $\rhoDN{4}{2}$ satisfies Greenberg's Panchishkin condition. Following Greenberg's approach in \cite{greenberg1994iwasawa}, it is possible to define a discrete Selmer group, denoted $\Sel^{\Gr}(\Q,D_{\rhoDN{4}{2}})$, attached to the Galois representation $\rhoDN{4}{2}$. The Pontryagin dual of $\Sel^{\Gr}(\Q,D_{\rhoDN{4}{2}})$, denoted $\Sel^{\Gr}(\Q,D_{\rhoDN{4}{2}})^\vee$, is a finitely generated $\Z_p[[\widetilde{\Gamma}]]$-module. \\

Corresponding to the tensor product of the Galois representation associated to the newform $f_E$ and $\rho_K$, Hida has constructed a two-variable Rankin-Selberg $p$-adic $L$-function, denoted $\theta^{\Gr}_{\cmmnt{\cmmnt{\pmb}}{4},2}$, in the fraction field of the Iwasawa algebra $\Z_p[[\widetilde{\Gamma}]]$. For the $p$-adic $L$-function $\theta^{\Gr}_{\cmmnt{\pmb}{4},2}$, one can vary the weight variable and the cyclotomic variable. See Hida's works in \cite{Hidafourier} and \cite{hida1996search}. One can consider the Iwasawa main conjecture for $\rhoDN{4}{2}$ formulated by Greenberg.

\begin{restatable}[Conjecture 4.1 in \cite{greenberg1994iwasawa}]{conjecture}{greenbergMC}
  \label{greenberg-main-conjecture}
  The $\Z_p[[\widetilde{\Gamma}]]$-module $\Sel^{\Gr}(\Q,D_{\rhoDN{4}{2}})^\vee$ is torsion. Furthermore, we have the following equality in $Z^1\left(\Z_p[[\widetilde{\Gamma}]]\right)$:
  \begin{align*}
    \Div\left(\Sel^{\Gr}(\Q,D_{\rhoDN{4}{2}})^\vee\right) \stackrel{?}{=} \Div\left(\theta^{\Gr}_{\cmmnt{\pmb}{4},2}\right).
  \end{align*}
\end{restatable}

Though the Galois representation $\rhoDN{4}{2}$ satisfies the Panchishkin condition, it is also possible to consider  Selmer groups and $p$-adic $L$-functions, that are truly artifacts of working in the supersingular case. On the analytic side,  there are four ``2-variable $p$-adic $L$-functions'', which we denote $\theta_{\cmmnt{\pmb}{4},2}^{++}$, $\theta_{\cmmnt{\pmb}{4},2}^{+-}$, $\theta_{\cmmnt{\pmb}{4},{2}}^{-+}$ and $\theta_{\cmmnt{\pmb}{4},2}^{--}$, that are elements of the fraction field of the ring $\Z_p[[\tilde{\Gamma}]]$. The construction of these $p$-adic $L$-functions is essentially\footnote{As we explain in Remark \ref{remark:periodnormalization}, we have chosen to work with the normalizations of the $p$-adic $L$-functions adopted by Xin Wan instead of David Loeffler. The difference in normalizations is due to different choices of complex periods.} due to David Loeffler \cite{loefflerray}. On the algebraic side, Byoung du Kim \cite{MR3224266} has constructed four discrete Selmer groups\footnote{Note that our choice of signs $+$ and $-$ is the same as Loeffler and Pollack. This choice is opposite to that of Kim and Kobayashi.} $\Sel^{++}(\Q,D_{\rhoDN42})$, $\Sel^{+-}(\Q,D_{\rhoDN42})$, $\Sel^{-+}(\Q,D_{\rhoDN42})$ and $\Sel^{--}(\Q,D_{\rhoDN42})$. One can view Loeffler's construction of the $p$-adic $L$-functions $\theta_{\cmmnt{\pmb}{4},2}^{\pm,\pm}$  (and Kim's construction of the Selmer groups $\Sel^{\pm\pm}(\Q,D_{\rhoDN{4}{2}})$ respectively) as a generalization of the construction of the one variable $\pm$ $p$-adic $L$-functions in \cite{pollack} by Robert Pollack (and the Selmer groups in \cite{kobayashi} by Shinichi Kobayashi respectively).  We have the following conjecture of Wan, which is a modification\footnote{Kim's conjecture relates the $\pm,\pm$ Selmer groups to David Loeffler's $2$-variable $p$-adic $L$-functions.} of a conjecture of Kim (Conjecture 3.1 in \cite{MR3224266}):

\begin{restatable}[Conjecture 6.7 in \cite{wan}]{conjecture}{kimMC}
  \label{kim-main-conjecture} Let $\bullet,\circ\in\{+,-\}$. The $\Z_p[[\widetilde{\Gamma}]]$-module  $\Sel^{\bullet\circ}(\Q,D_{\rhoDN{4}{2}})^\vee$ is torsion. We have the following equality in $Z^1\left(\Z_p[[\widetilde{\Gamma}]]\right)$:
  \begin{align*}
    \Div\left(\Sel^{\bullet\circ}(\Q,D_{\rhoDN{4}{2}})^\vee\right) \stackrel{?}{=} \Div\left(\theta_{{4},2}^{\bullet\circ}\right).
  \end{align*}
\end{restatable}

Our main theorem is the following:

\begin{Theorem} \label{maintheorem}
  Let $\{\theta_\One, \theta_\Two\}$ denote one of the following unordered pairs:
  \begin{align*}
      & \{\tDN42{++}, \tDN42{+-}\},\quad   \{\tDN42{++}, \tDN42{-+}\},   \quad \{\tDN42{--},  \tDN42{+-}\}, \quad \{\tDN42{--}, \tDN42{-+}\}, \\ &\{\tDN42{++}, \tDN42{\Gr}\}, \quad \{\tDN42{+-}, \tDN42{\Gr}\} , \quad \{\tDN42{-+}, \tDN42{\Gr}\}, \quad \{\tDN42{--}, \tDN42{\Gr}\}.
  \end{align*}
  Suppose that the following two conditions hold:
  \begin{enumerate}
    \item Conjecture \ref{greenberg-main-conjecture} and Conjecture \ref{kim-main-conjecture} hold.
          \item\label{pn-hyp-main-theorem} The elements $\theta_{\One}$ and $\theta_{\Two}$ of the UFD $\Z_p[[\widetilde{\Gamma}]]$ have no common irreducible factor.
  \end{enumerate}
  Then, we have the following equality in $Z^2\left(\Z_p[[\tilde{\Gamma}]]\right)$:
  \begin{align*}
    c_2\left(\frac{\Z_p[[\tilde{\Gamma}]]}{\left(\theta_\One, \theta_\Two\right)}\right) = c_2\bigg(\mathcal{Z}(\Q,D_{\rhoDN{4}{2}})\bigg) + c_2\bigg(\mathcal{Z}^\Sstar(\Q,D_{\rhoDNstar{4}{2}})\bigg) + \sum_{l \in \Sigma \setminus \{p\}} c_2\left(\left(H^0\left(\Q_l, D_{\rhoDN{4}{2}}\right)^\vee\right)_{\pn}\right).
  \end{align*}
\end{Theorem}

Let us briefly review the notations used in Theorem \ref{maintheorem}. One can consider the Selmer groups $\Sel_{\One}(\Q,D_{\rhoDN42})$ and $\Sel_{\Two}(\Q,D_{\rhoDN42})$ appearing in Conjecture \ref{greenberg-main-conjecture} and Conjecture \ref{kim-main-conjecture} corresponding to the $p$-adic $L$-functions $\theta_\One$ and $\theta_\Two$ respectively.  Here, $\mathcal{Z}(\Q,D_{\rhoDN{4}{2}})$ denotes\footnote{To make the notation for $\mathcal{Z}(\Q,D_{\rhoDN{4}{2}})$ simpler, we have not included the indices $\One$ and $\Two$.} the Pontryagin dual of the intersection
$$ \Sel_{\One}(\Q,D_{\rhoDN{4}{2}}) \cap \Sel_{\Two}(\Q,D_{\rhoDN{4}{2}})$$
inside the appropriate first discrete global Galois cohomology group. One can also define modules $D_{\rhoDNstar{4}{2}}$ and $\mathcal{Z}^\Sstar(\Q,D_{\rhoDNstar{4}{2}})$ for the Tate dual $\rhoDNstar{4}{2}$ of the Galois representation $\rhoDN{4}{2}$. See Section \ref{S:pseudonulldesc} for the description of $\mathcal{Z}(\Q,D_{\rhoDN{4}{2}})$ and $\mathcal{Z}^\Sstar(\Q,D_{\rhoDNstar{4}{2}})$ respectively. The invariant $c_2$ associates an element in $Z^2\left(\Z_p[[\widetilde{\Gamma}]]\right)$ to a pseudo-null $\Z_p[[\widetilde{\Gamma}]]$-module. See Section \ref{sectiondefinition} for its definition. This invariant is the generalization of the ``characteristic divisor'' appearing in the Iwasawa main conjecture formulated by Greenberg \cite{greenberg1994iwasawa}.

For each finite prime $l$ in $\Sigma \setminus \{p\}$, the $\Z_p[[\widetilde{\Gamma}]]$-module $\left(H^0\left(\Q_l, D_{\rhoDN{4}{2}}\right)^\vee\right)_{\pn}$ denotes the maximal pseudo-null submodule of $H^0\left(\Q_l, D_{\rhoDN{4}{2}}\right)^\vee$. See section \ref{section:fudge} for a discussion of these fudge factors $ c_2\left(\left(H^0\left(\Q_l, D_{\rhoDN{4}{2}}\right)^\vee\right)_{\pn}\right)$ at primes $l \neq p$. The discussion is based on the criterion of N\'eron-Ogg-Shafarevich and the circle of ideas in Tate's algorithms \cite{MR0393039}.

\tocless \subsection{The pseudo-nullity conjecture of Coates and Sujatha}

Our main motivation in considering condition (\ref{pn-hyp-main-theorem}) of Theorem \ref{maintheorem}, which we later label ``\ref{gcd}'', is a conjecture of Coates and Sujatha involving the fine Selmer group. Let $\Sigma$ denote a finite set of primes in $\Q$ containing $p$, $\infty$, the set of primes dividing the conductor of the elliptic curve $E$ and the primes ramified in the extension $K/\Q$. Let $\Q_\Sigma$ denote the maximal extension of $\Q$ unramified outside $\Sigma$. Let $G_\Sigma$ denote $\Gal{\Q_\Sigma}{\Q}$. The fine Selmer group, denoted $\Sha^1\left(\Q,D_{\rhoDN{4}{2}}\right)$, is defined below:

\begin{align*}
  \Sha^1\left(\Q,D_{\rhoDN{4}{2}}\right) := \ker\left(H^1\left(G_\Sigma, D_{\rhoDN{4}{2}}\right) \rightarrow \bigoplus_{\nu \in \Sigma} H^1\left(\Gal{\overline{\Q}_\nu}{\Q_\nu},D_{\rhoDN{4}{2}}\right) \right).
\end{align*}
We now state the conjecture of Coates and Sujatha.
\begin{conjecture}[Conjecture B in \cite{MR2148798}] \label{coates-sujatha-conj}
  The $\Z_p[[\widetilde{\Gamma}]]$-module $\Sha^1\left(\Q,D_{\rhoDN{4}{2}}\right)^\vee$
  is pseudo-null.
\end{conjecture}

Recall that a finitely generated torsion-module $M$ over the UFD $\Z_p[[\widetilde{\Gamma}]]$ is said to be pseudo-null if $\Div(M)$ equals zero. We have the following natural surjection of $\Z_p[[\widetilde{\Gamma}]]$-modules:
\[
  \Sel^{\pm,\pm}(\Q,D_{\rhoDN{4}{2}})^\vee \twoheadrightarrow \Sha^1\left(\Q,D_{\rhoDN{4}{2}}\right)^\vee, \qquad \Sel^{\Gr}(\Q,D_{\rhoDN{4}{2}})^\vee \twoheadrightarrow \Sha^1\left(\Q,D_{\rhoDN{4}{2}}\right)^\vee.
\]
Assume that Conjecture \ref{greenberg-main-conjecture} and Conjecture \ref{kim-main-conjecture} are valid.  Then\footnote{ $\mathrm{Supp}_{\Ht=1}(M)$ denotes the set of height one prime ideals of $\Iwalg$ in the support of a finitely generated $\Iwalg$-module~$M$.}, $\mathrm{Supp}_{\Ht=1}\left(\Sha^1\left(\Q,D_{\rhoDN{4}{2}}\right)^\vee\right)$ is a subset of
{\small \begin{align*}
  \mathrm{Supp}_{\Ht=1}\left(\frac{\Iwalg}{(\theta^{\Gr}_{\cmmnt{\pmb}{4},2})}\right) \bigcap  \mathrm{Supp}_{\Ht=1}\left(\frac{\Iwalg}{(\theta^{++}_{\cmmnt{\pmb}{4},2})}\right) & \bigcap \mathrm{Supp}_{\Ht=1}\left(\frac{\Iwalg}{(\theta^{+-}_{\cmmnt{\pmb}{4},2})}\right) \bigcap \mathrm{Supp}_{\Ht=1}\left(\frac{\Iwalg}{(\theta^{-+}_{\cmmnt{\pmb}{4},2})}\right) \bigcap \mathrm{Supp}_{\Ht=1}\left(\frac{\Iwalg}{(\theta^{--}_{\cmmnt{\pmb}{4},2})}\right).
  \end{align*}}

To investigate Conjecture \ref{coates-sujatha-conj}, it seems instructive to consider the setup when two of the above $p$-adic $L$-functions have no common irreducible factor in the UFD $\Z_p[[\widetilde{\Gamma}]]$. The motivation behind proving Theorem \ref{maintheorem} is to provide a partial answer to Question \ref{question-analytic} in this setup. See Section \ref{section-examples}, where we produce numerical evidence towards the existence of elliptic curves $E$ with good supersingular reduction at $p$, such that $\theta^{++}_{\cmmnt{\pmb}{4},2}$ and $\theta^{+-}_{\cmmnt{\pmb}{4},2}$ have no common irreducible factors in the UFD $\Z_p[[\widetilde{\Gamma}]]$. If $\theta^{++}_{\cmmnt{\pmb}{4},2}$ and $\theta^{+-}_{\cmmnt{\pmb}{4},2}$ have no common irreducible factor in $\Iwalg$, then the pseudo-nullity conjecture of Coates-Sujatha is also valid, assuming the validity of the Iwasawa main conjectures. Hence, these example provide evidence towards Conjecture \ref{coates-sujatha-conj} as well.

The examples in Section \ref{section-examples} are based on computations of Robert Pollack, given on his website \url{http://math.bu.edu/people/rpollack/Data/data.html}. These computations and the examples in Section \ref{section-examples} are based on the theory of overconvergent modular symbols of Stevens \cite{stevens1994rigid}, Pollack-Stevens \cite{MR2760194} and related to Pollack's work with Masato Kurihara \cite{kuriharapollack}.

\begin{remark}
The pseudo-nullity conjecture of Coates and Sujatha (Conjecture B in \cite{MR2148798}) is formulated in greater generality than the setting of Conjecture \ref{coates-sujatha-conj}. They formulate their conjecture for the fine Selmer group of an elliptic curve over certain (admissible) $p$-adic Lie extensions whose Galois group is a $p$-adic Lie group with dimension $\geq 2$.
We refer the reader to works of Ochi \cite{MR2830442},  Lim \cite{MR3440445} and Shekhar \cite{shekhar2018journal} where there are other examples verifying the pseudo-nullity conjecture of Coates-Sujatha.  The setups in their works and their approaches are completely different to ours.
\end{remark}

\tocless \subsection{Method of Proof}
Theorem \ref{maintheorem} is a consequence of Theorem \ref{general-main-theorem}, which is applicable in a fairly general setting.  To explain the method of proof of Theorem \ref{general-main-theorem} in the introduction, we will consider the setup of Theorem \ref{maintheorem}. A key ingredient in proving our results is the construction of an auxillary $\Z_p[[\widetilde{\Gamma}]]$-module\footnote{To make the notation for $\XXX(\Q,D_{\rhoDN{4}{2}})$ also simpler, we have not included the indices $\One$ and $\Two$.}, denoted $\XXX(\Q,D_{\rhoDN{4}{2}})$. Assuming the validity of Conjecture \ref{greenberg-main-conjecture} and Conjecture \ref{kim-main-conjecture}, this $\Z_p[[\widetilde{\Gamma}]]$-module turns out to be torsion-free and to have rank one. This module also fits into the following surjection of $\Z_p[[\widetilde{\Gamma}]]$-modules:
\begin{align*}
  \underbrace{H^1\left(G_\Sigma, D_{\rhoDN{4}{2}} \right)^\vee}_{\substack{\text{Conjecturally} \\ \text{has $\Z_p[[\widetilde{\Gamma}]]$-rank two}}} \twoheadrightarrow \underbrace{\XXX(\Q,D_{\rhoDN{4}{2}})}_{\substack{\text{Conjecturally} \\ \text{has $\Z_p[[\widetilde{\Gamma}]]$-rank one}}}  \twoheadrightarrow \underbrace{\substack{\Sel_{\One}(\Q,D_{\rhoDN{4}{2}})^\vee \\ \Sel_{\Two}(\Q,D_{\rhoDN{4}{2}})^\vee}}_{\substack{\text{Conjecturally} \\ \text{ $\Z_p[[\widetilde{\Gamma}]]$-torsion}}} \twoheadrightarrow \underbrace{\ZZZ(\Q,D_{\rhoDN{4}{2}})}_{\substack{\text{Pseudo-null, assuming}\\ \text{the hypotheses in Theorem \ref{maintheorem}}}} \twoheadrightarrow \underbrace{\Sha^{1}\left(\Q,D_{\rhoDN{4}{2}}\right)^\vee}_{\text{Conjecturally pseudo-null}}.
\end{align*}
To prove Theorem \ref{maintheorem}, we first show that for every height two prime ideal $\QQQ$ in the ring $\Z_p[[\widetilde{\Gamma}]]$, we have the following short exact sequence of $\Z_p[[\widetilde{\Gamma}]]_\QQQ$-modules:
\begin{align} \label{intro-ses}
  0 \rightarrow \ZZZ(\Q,D_{\rhoDN{4}{2}})_\QQQ  \rightarrow \frac{\Z_p[[\widetilde{\Gamma}]]_\QQQ}{(\theta_\One,\theta_\Two)} \rightarrow \coker\left(\XXX(\Q,D_{\rhoDN{4}{2}}) \rightarrow \XXX(\Q,D_{\rhoDN{4}{2}})^{**}\right) \otimes_{\Z_p[[\widetilde{\Gamma}]]} \Z_p[[\widetilde{\Gamma}]]_\QQQ \rightarrow 0.
\end{align}

Here, $\XXX(\Q,D_{\rhoDN{4}{2}})^{**}$ is the reflexive hull of the $\Z_p[[\widetilde{\Gamma}]]$-module $\XXX(\Q,D_{\rhoDN{4}{2}})$. The $\Iwalg$-module $\XXX(\Q,D_{\rhoDN{4}{2}})^{**}$ turns out to be free.  \\

The $\Z_p[[\widetilde{\Gamma}]]$-module $\coker\bigg(\XXX(\Q,D_{\rhoDN{4}{2}}) \rightarrow \XXX(\Q,D_{\rhoDN{4}{2}})^{**}\bigg)$ is pseudo-null. To study the invariant $c_2$ associated to this pseudo-null $\Z_p[[\widetilde{\Gamma}]]$-module and thus complete the proof of Theorem \ref{maintheorem}, we use the duality theorems developed by Jan Nekov{\'a}{\v{r}} in his work on Selmer complexes \cite{nekovar2006selmer}. We show that we have the following equality in $Z^2\left(\Z_p[[\widetilde{\Gamma}]]\right)$:
\begin{align} \label{intro-c2-equality}
  c_2\bigg(\coker\left(\XXX(\Q,D_{\rhoDN{4}{2}}) \rightarrow \XXX(\Q,D_{\rhoDN{4}{2}})^{**}\right)\bigg) =  c_2\bigg(\mathcal{Z}^\Sstar(\Q,D_{\rhoDNstar{4}{2}})\bigg) + \sum_{l \in \Sigma \setminus \{p\}} c_2\left(\left(H^0\left(\Q_l, D_{\rhoDN{4}{2}}\right)^\vee\right)_{\pn}\right).
\end{align}
A key idea in establishing (\ref{intro-c2-equality}) involves a careful study of $\Ext$ groups. This idea is based on the theory of Iwasawa adjoints, which was first conceived by Kenkichi Iwasawa \cite{MR0349627} and later developed in greater generality in works of Uwe Jannsen \cite{MR1097615,MR3271243}. Combining equations (\ref{intro-ses}) and (\ref{intro-c2-equality}) completes the proof of Theorem \ref{maintheorem}. \\

To apply our theorem in the general setup (Theorem \ref{general-main-theorem}) to the setting of Theorem \ref{maintheorem}, we use results of Kim \cite{MR3224266} and Kobayashi \cite{kobayashi}. These results of Kim and Kobayashi are in turn built on earlier works of Perrin-Riou \cite{MR1031902} and Rubin \cite{MR880958}.

\begin{remark} \label{notconsidering}
  Since it is important to our methods to construct the rank one $\Z_p[[\widetilde{\Gamma}]]$-module $\XXX(\Q,D_{\rhoDN{4}{2}})$, we do not  consider the case where $\{\theta_\One, \theta_\Two\}$ equals $\{\theta^{++}_{\cmmnt{\pmb}{4},2}, \theta^{--}_{\cmmnt{\pmb}{4},2}\}$ or $\{\theta^{+-}_{\cmmnt{\pmb}{4},2}, \theta^{-+}_{\cmmnt{\pmb}{4},2}\}$. A similar construction of $\XXX(\Q,D_{\rhoDN{4}{2}})$ in these cases would produce a $\Z_p[[\widetilde{\Gamma}]]$-module that has rank at least two (it is exactly two if the Pontryagin dual of any of the corresponding Selmer groups is torsion over $\Z_p[[\tilde\Gamma]]$). Another interesting point to note is that when the root number of the elliptic curve $E$ over $K$ equals $-1$, it is known that $\ker(\pi_{\mathrm{ac}})$ belongs to the support of the divisors $\Div\left(\tDN42{++}\right)$ and $\Div\left(\tDN42{--}\right)$. The map  $\pi_{\mathrm{ac}}:\Z_p[[\widetilde{\Gamma}]] \rightarrow \Z_p[[\Gamma_{\mathrm{ac}}]]$ denotes the natural projection map (which one can view as the ``anticyclotomic specialization''). Here, $\Gamma_{\mathrm{ac}}$ denotes $\Gal{K_{\mathrm{ac}}}{K}$, the Galois group of the anti-cyclotomic $\Z_p$-extension $K_{\mathrm{ac}}$ of $K$.
  See \cite{BL2,CW,LV}. However, at present, we do not know when the root number of the elliptic curve $E$ over $K$ equals $-1$ whether $\ker(\pi_{\mathrm{ac}})$ belongs to the support of $\Div\left(\tDN42{+-}\right)$ or $\Div\left(\tDN42{-+}\right)$.
\end{remark}

\begin{remark}
As a referee had pointed out to us, the construction of Hida's Rankin-Selberg $p$-adic $L$-function is in fact contingent on a choice of a prime above $p$. On the algebraic side, this choice amounts to defining the filtration $\Fil T_{\rhoDN{4}{2}}$ for the Selmer group in Section 5.1 involving either the character $\kappa_\p$ or $\kappa_\q$. One can thus consider another Rankin-Selberg $p$-adic $L$-function. We do not consider the pair of Rankin-Selberg $p$-adic $L$-functions in Theorem \ref{maintheorem} since, just as in Remark \ref{notconsidering}, the construction of $\XXX(\Q,D_{\rhoDN{4}{2}})$ does not yield a rank one $\Z_p[[\widetilde{\Gamma}]]$-module.  The discussion of Theorem \ref{maintheorem} for the pair of $p$-adic $L$-functions involving this additional Rankin-Selberg $p$-adic $L$-function and one of the $p$-adic $L$-functions $\tDN42{\pm,\pm}$ is completely analogous to the discussion of the pairs $\{\tDN42{\pm,\pm}, \tDN42{\Gr}\}$.
\end{remark}

\begin{remark}
Analogues of the $\pm\pm$ $p$-adic $L$-functions and Selmer groups have been constructed for cuspidal eigenforms with weight $\ge2$ when $p$ is a good non-ordinary prime (see \cite{BL,CCSS,lei14,sprung}).  In the proof of Theorem~\ref{maintheorem}, we rely on results of Kim \cite{MR3224266} to verify that the local Selmer conditions at $p$ satisfy certain freeness conditions (labeled \ref{loc-free} in Section \ref{generalsetup}). At present, we are not sure to what extent these freeness results are available and whether they can be extended to the more general settings.
\end{remark}

\begin{remark}
The purpose of developing Theorem \ref{general-main-theorem} in a general setting is to pursue applications to other arithmetic setups. In \cite{lei2019codimension}, we obtain one such application of Theorem \ref{general-main-theorem} in the setting of cyclotomic twist deformations of Hida families. We refer the interested reader to \cite{lei2019codimension} for the exact details. An added difficulty in this setting is that the (normalizations of) deformation rings, appearing in Hida theory, are not always known to be regular (see section 4 in \cite{palvannan2015homological} for examples when such rings are not even UFDs).
\end{remark}

\tocless\subsection{Organization of the paper}
Section \ref{some-homological-algebra} involves establishing preliminaries in commutative and homological algebra. Section \ref{generalsetup} describes the objects involved in the general setting of our main theorem. Section \ref{generalresults} involves the proof of our main theorem in the general setting.

Section \ref{S:Panch} describes the Iwasawa main conjecture formulated by Greenberg in the setting of Theorem \ref{maintheorem}. Section \ref{S:plusminus} describes the $\pm\pm$ Iwasawa main conjectures for elliptic curves with supersingular reduction at $p$, formulated by Kim. Section \ref{S:proofmaintheorem} describes the proof of Theorem \ref{maintheorem}. Section \ref{section-examples} deals with providing evidence for \ref{gcd}, and in turn the pseudo-nullity conjecture of Coates-Sujatha, in the setup of Theorem \ref{maintheorem}. \\

\tocless\subsection{Notations}

Here is a brief summary of some of the notations used in this paper.

\begin{itemize}
  \item If $\RRR$ is a Noetherian, complete, integrally closed, local domain of characteristic zero with Krull dimension $n+1$ and whose residue field has characteristic $p$, we let $\RRR^\vee$ denote $\Hom_{\cont}\left(R,\frac{\Q_p}{\Z_p}\right)$. For a finitely generated module $M$ over such a ring $\RRR$, we let $M^\vee$ denote the Pontryagin dual of $M$ and $M^*$ denote its reflexive dual $\Hom_\RRR(M,\RRR)$. Similarly, if $D$ is a discrete module over such a ring $\RRR$, we let $D^\vee$ denote the Pontryagin dual of $D$.
If $I$ is an ideal of $\RRR$, we write $D[I]$ for the $\RRR$-submodule of $D$ annihilated by all elements of the ideal $I$.
  \item Suppose $L$ is a field. If $M$ is a discrete module with a continuous action of the absolute Galois group $G_L$, we will let $H^i\left(L,M\right)$ denote the continuous cohomology group $H^i\left(G_L,M\right)$, for each $i \geq 0$.
  \item If $\hat{\mathcal{F}}$ is a formal group defined over the  integral closure of $\Z_p$ in an algebraic extension $L$ of $\Q_p$,  we let $\hat{\mathcal{F}}(L)$ denote the group of $L$-points on the corresponding formal group. That is, if $\m_L$ denotes the maximal ideal in the integral closure of $\Z_p$ in $L$, then $\hat{\mathcal{F}}(L)$ would equal  $\hat{\mathcal{F}}(\m_L)$.

  \item Let $L$ be an algebraic extension of $\Qp$. Let $\Gamma_1$ be a topological group isomorphic to $\Z_p$. Let $\gamma_1$ be a topological generator of $\Gamma_1$ and fix an isomorphism $\Z_p[[\Gamma_1]] \cong \Z_p[[T]]$ of rings by sending $\gamma_1$ to $T+1$. To emphasize the choice of the topological generator, we may write $\Z_p[[\gamma_1-1]]$ and $L[[\gamma_1-1]]$ instead of the power series rings $\Z_p[[T]]$ and $L[[T]]$. If $\mu$ is an $L$-valued distribution on the topological group $\Gamma_1$, we have the Amice transform
        \begin{align*}
          \int_{x\in \Gamma_1}(1+T)^x \mu(x) & \in  L[[T]].
        \end{align*}

        Suppose $\Gamma_2$ is another topological group isomorphic to $\Z_p$. Let $\gamma_2$ be a topological generator of $\Gamma_2$. Fix an isomorphism of the completed group ring $\Z_p[[\Gamma_1 \times \Gamma_2]]$ with $\Z_p[[S,T]]$ by identifying $\gamma_1-1$ and $\gamma_2-1$ with $S$ and $T$ respectively. Once again, to emphasize the choice of the topological generator, we may write $\Z_p[[\gamma_1-1,\gamma_2-1]]$ and $L[[\gamma_1-1,\gamma_2-1]]$ instead of the power series rings $\Z_p[[S,T]]$ and $L[[S,T]]$.  Suppose $\mu$ is an $L$-valued distribution on $\Gamma_1 \times \Gamma_2$, its Amice transform is given by $$\int_{x\in \Gamma_1,y\in \Gamma_2}(1+S)^{x}(1+T)^y\mu(x,y)\in L[[S,T]].$$
\end{itemize}

\section{Some Commutative and Homological Algebra} \label{some-homological-algebra}

Let $\RRR$ denote a Noetherian local ring. Let $\MMM$ be a finitely generated $\RRR$-module. In Section \ref{some-homological-algebra}, we will simply accumulate the general results in commutative and homological algebra needed for our purposes.

\begin{definition} The length of a finitely generated $\RRR$-module $\MMM$, denoted $\len_\RRR \MMM$, is the length of a(ny) composition series of $\MMM$ (if a composition series does not exist, we set the length to equal infinity).
\end{definition}

We recall a few results on lengths of modules.
\begin{lemma} [Proposition 6.9 in \cite{atiyah1969introduction}] \label{length-additive}
  Suppose we have a short exact sequence $$0 \rightarrow \MMM' \rightarrow \MMM \rightarrow \MMM'' \rightarrow 0$$ of finitely generated $\RRR$-modules. Then, we have the following equality of lengths: $$\len_{\RRR}\MMM  = \len_{\RRR}\MMM' + \len_{\RRR} \MMM''.$$
\end{lemma}

The following lemma follows from Propositions 6.3 and  6.8 in \cite{atiyah1969introduction}.
\begin{lemma}\label{artinian-length}
  Suppose the local Noetherian ring $\RRR$ is also Artinian. For every finitely generated $\RRR$-module $\MMM$, we have  $\len_\RRR(\MMM) < \infty$.
\end{lemma}

\subsection{Pseudo-null modules, Reflexive modules and $\Ext$ groups}\label{sectiondefinition}

\begin{definition} A finitely generated $\RRR$-module $\MMM$ is said to be a pseudo-null $\RRR$-module if
  \begin{align*}
    \MMM_\p = 0, \text{ for all }   \text{ prime ideals } \p \text{ in } \RRR \text{ such that } \mathrm{height}(\p) \leq 1.
  \end{align*}
\end{definition}

Let $\Ann_\RRR(\MMM)$ denote the annihilator of the $\RRR$-module $\MMM$. If the $\RRR$-module $\MMM$ is pseudo-null, then the height of the ideal $\Ann_\RRR(\MMM)$ is greater than or equal to two.

\begin{lemma} \label{finite-length}
  Suppose $\ZZZ$ is a pseudo-null $\RRR$-module. For every height two prime ideal $\QQQ$ in $\RRR$, we have $ \len_{\RRR_\QQQ} \ZZZ \otimes_\RRR \RRR_\QQQ < \infty$.
\end{lemma}

\begin{proof}
  There exists a positive integer $m$ and a surjection $$\left(\frac{\RRR_\QQQ}{\Ann_\RRR(\ZZZ) \RRR_\QQQ}\right)^m \twoheadrightarrow \ZZZ \otimes_\RRR \RRR_\QQQ,$$ of $\RRR_\QQQ$-modules. Note that $\RRR_\QQQ$ is a Noetherian local ring with Krull dimension two. The $\RRR_\QQQ$-ideal $\Ann_\RRR(\ZZZ)\RRR_\QQQ$ has height two. As a result, the quotient ring $\frac{\RRR_\QQQ}{\Ann_\RRR(\ZZZ)\RRR_\QQQ}$, being Noetherian and having Krull dimension zero, is Artinian. See Theorem 8.5 in Atiyah-Macdonald \cite{atiyah1969introduction}. The proof of this lemma now follows from Lemma \ref{artinian-length}.
\end{proof}

Lemma \ref{finite-length} allows us to associate, to a pseudo-null $\RRR$-module,  an element in $Z^2(\RRR)$ (the free abelian group on the set of height two prime ideals of $\RRR$). If $\MMM$ is a pseudo-null $\RRR$-module, we define an element $c_2(\MMM)$ in $Z^2(\RRR)$ as the following formal sum:
\begin{align}
  c_2(\MMM) := \sum_{\substack{\QQQ \subset \RRR \\ \mathrm{height}(\QQQ)=2}} \left(\len_{\RRR_\QQQ} \MMM_\QQQ\right) \QQQ.
\end{align}
In the above formula, the summation is taken over all the height two prime ideals $\QQQ$ of $\RRR$. Since length is additive in exact sequences and since localization is exact, we have the following lemma:
\begin{lemma}
  Suppose we have a short exact sequence $$0 \rightarrow \MMM' \rightarrow \MMM \rightarrow \MMM'' \rightarrow 0$$ of finitely generated pseudo-null $\RRR$-modules. Then, we have the following equality in $Z^2(\RRR)$: $$c_2(\MMM)  = c_2(\MMM') + c_2(\MMM'').$$
\end{lemma}

Theorem 4.4.8 in Weibel's book \cite{weibel1995introduction} and Lemma \ref{finite-length} (see also Standard Facts 4.4.7 in \cite{weibel1995introduction}) automatically give us the following lemma:
\begin{lemma}
  Let $\ZZZ$ be a finitely generated pseudo-null $\RRR$-module. Suppose $\QQQ$ is a height two prime ideal in $\RRR$ such that $\ZZZ \otimes_\RRR \RRR_\QQQ  \neq 0$. Then, we have $\Depth_{\RRR_\QQQ} \left(\ZZZ \otimes_\RRR \RRR_\QQQ\right) = 0$.
\end{lemma}

Let us recall a theorem of Serre.

\begin{lemma}[Theorem 23.8 in \cite{matsumura1989commutative}] \label{serre-r1-s2}
  Let $\RRR$ be a domain. The domain $\RRR$ is integrally closed if and only if both the following conditions hold:
  \begin{itemize}
    \item $\RRR_\p$ is a discrete valuation ring, for every height one prime ideal $\p$ of $\RRR$,
    \item $\mathrm{Depth}_{\RRR_\QQQ}(\RRR_\QQQ)=2$, for all height two prime ideals $\QQQ$ of $\RRR$.
  \end{itemize}
\end{lemma}

For the rest of section \ref{some-homological-algebra}, we will let $\RRR$ be an integrally closed, local, Noetherian domain. Let $\XXX$ be a finitely generated $\RRR$-module. The reflexive hull of $\XXX$ (denoted $\XXX^{**}$) is defined below:
\begin{align*}
  \XXX^{**} := \Hom_{\RRR}\left(\Hom_{\RRR}\left(\XXX,\RRR\right),\RRR\right).
\end{align*}

One can define a natural map $i_\XXX:\XXX \rightarrow \XXX^{**}$, as given below:
\begin{align*}
  i_\XXX: \XXX & \rightarrow \XXX^{**}                   \\
  x            & \rightarrow (\phi \rightarrow \phi(x)).
\end{align*}

\begin{definition} The $\RRR$-module $\XXX$ is said to be a \underline{reflexive $\RRR$-module} if the map $i_\XXX$ is an isomorphism.
\end{definition}

The following lemma is proved in Section 1 of Chapter V in the book by Neukirch-Winberg-Schmidt \cite{neukirch2008cohomology}.
\begin{lemma} \label{torfree-pn}
  \mbox{}
  \begin{enumerate}
    \item The $\RRR$-module $\ker(i_\XXX)$ equals $\XXX_{\mathrm{tor}}$, the maximal $\RRR$-torsion submodule of $\XXX$.
    \item The $\RRR$-module $\coker(i_\XXX)$ is a pseudo-null $\RRR$-module.
  \end{enumerate}
\end{lemma}

We also have the following useful result that follows from Propositions 1.2.12 and 1.4.1 in the book by Bruns and Herzog \cite{MR1251956}.
\begin{lemma}
  Suppose $\MMM$ is a finitely generated non-zero reflexive $\RRR$-module. Let $\QQQ$ be a height two prime ideal of $\RRR$. Then, $\Depth_{\RRR_\QQQ} \MMM_\QQQ$ equals $2$.
\end{lemma}

\begin{lemma}[Proposition 3.3.10 in Weibel's book \cite{weibel1995introduction}]
  Let $\p$ be a prime ideal in $\RRR$. Let $\MMM$ be a finitely generated $\RRR$-module. Let $\NNN$ be an $\RRR$-module. We have the following natural isomorphism of $\RRR_\p$-modules, for all $i \geq 0$:
  \begin{align*}
    \Ext^{i}_\RRR \left(\MMM,\NNN \right) \otimes_\RRR \RRR_\p \cong \Ext^{i}_{\RRR_\p} \left(\MMM_\p, \NNN_\p\right).
  \end{align*}
\end{lemma}

In particular, if $\MMM$ is a finitely generated reflexive $\RRR$-module, then $\MMM_\p$ is a finitely generated reflexive $\RRR_\p$-module for every prime ideal $\p$ since
\begin{align*}
  \MMM_\p \cong \Hom_{\RRR}\left(\Hom_{\RRR}\left(\MMM,\RRR\right),\RRR\right) \otimes_\RRR \RRR_\p \cong \Hom_{\RRR_\p}\left(\Hom_{\RRR_\p}\left(\MMM_\p,\RRR_\p\right),\RRR_\p\right).
\end{align*}

\begin{lemma} \label{van-ext-reflexive}
  Suppose $\MMM$ is a finitely generated reflexive $\RRR$-module. Suppose $\QQQ$ is a height two prime ideal in $\RRR$. If the projective dimension of the $\RRR_\QQQ$-module $\MMM_\QQQ$ is finite, then $\MMM_\QQQ$ is a free $\RRR_\QQQ$-module. Furthermore, we have
  \begin{align} \label{van-ext-ref}
    \Ext^{i}_\RRR \left(\MMM,\RRR\right) \otimes_\RRR \RRR_\QQQ \cong \Ext^{i}_{\RRR_\QQQ} \left(\MMM_\QQQ,\RRR_\QQQ \right) = 0, \quad \forall i \geq 1.
  \end{align}
\end{lemma}

\begin{proof}
If $\MMM_\QQQ$ equals zero, the lemma follows automatically. Let us work with the case when $\MMM_\QQQ$ is not zero. Using the Auslander-Buchsbaum equality (see Theorem 4.4.15 in Weibel's book \cite{weibel1995introduction}) over the local ring $\RRR_\QQQ$ gives us the following equality:
  \begin{align*}
    \mathrm{pd}_{\RRR_\QQQ} \MMM_\QQQ = \Depth_{\RRR_\QQQ}\RRR_\QQQ - \Depth_{\RRR_\QQQ}\MMM_\QQQ = 2-2 = 0
  \end{align*}
  Here, $\mathrm{pd}_{\RRR_\QQQ}$ denotes the projective dimension over the ring $\RRR_\QQQ$. Finitely generated projective modules over commutative Noetherian local rings are free.  Equation (\ref{van-ext-ref}) then follows from Lemma 2 in Chapter 19 of Matsumura's book \cite{matsumura1989commutative}.
\end{proof}

Suppose $\MMM$ and $\NNN$ are two finitely generated $\RRR$-modules. Since $\RRR$ is a commutative Noetherian ring, $\Ext^{i}_{\RRR}(\MMM,\NNN)$ is a finitely generated $\RRR$-module, for every non-negative integer $i$. See Lemma 3.3.6 in Weibel's book \cite{weibel1995introduction}. As indicated by Lemma \ref{van-ext-reflexive},  localization commutes with $\Ext$ for finitely generated modules over commutative rings. So, if $\ZZZ$ is a finitely generated pseudo-null $\RRR$-module, then $\Ext^{i}_{\RRR}(\ZZZ,\RRR)$ is also a finitely generated pseudo-null $\RRR$-module, for every non-negative integer $i$.

\begin{lemma} \label{van-ext0-1-pn}
  Let $\ZZZ$ be a finitely generated pseudo-null  $\RRR$-module. Let $\QQQ$ be a height two prime ideal in $\RRR$. We have
  \begin{align}\label{van-ext0-1-pn-eqn}
    \Ext^0_{\RRR_\QQQ}\left(\ZZZ \otimes_\RRR \RRR_\QQQ, \RRR_\QQQ \right) = \Ext^1_{\RRR_\QQQ}\left(\ZZZ \otimes_\RRR \RRR_\QQQ, \RRR_\QQQ \right)  = 0.
  \end{align}
\end{lemma}

\begin{proof}
  By Lemma \ref{finite-length}, one can consider a composition series for $\ZZZ \otimes_{\RRR} \RRR_\QQQ$ of finite length. Suppose the following chain of $\RRR_\QQQ$-modules is such a composition series for $\ZZZ \otimes_{\RRR} \RRR_\QQQ$:
  \begin{align*}
    \ZZZ \otimes_{\RRR} \RRR_\QQQ = \MMM_0 \supsetneq \MMM_1 \supsetneq \cdots \supsetneq \MMM_n =0.
  \end{align*}
  Each quotient $\frac{\MMM_i}{\MMM_{i+1}}$ in the composition series is isomorphic, as an $\RRR_\QQQ$-module, to $k_\QQQ$, the residue field of $\RRR_\QQQ$. One can then use a d\'evissage argument to reduce to the case when $\ZZZ \otimes_\RRR \RRR_\QQQ$ equals $k_\QQQ$. As Lemma \ref{serre-r1-s2} indicates, $\Depth_{\RRR_\QQQ} \RRR_\QQQ$ equals $2$. So, in this case when $\ZZZ \otimes_\RRR \RRR_\QQQ$ equals $k_\QQQ$, equation (\ref{van-ext0-1-pn-eqn}) follows from Theorem 4.4.8 in Weibel's book \cite{weibel1995introduction}.
\end{proof}

Let $\QQQ$ be a height two prime ideal in $\RRR$ such that the localization $\RRR_\QQQ$ is a Gorenstein local ring. In this case, the injective dimension of $\RRR_\QQQ$, as an $\RRR_\QQQ$-module, equals two. See Corollary 4.4.10 in Weibels' book \cite{weibel1995introduction}. As a result, we have
\begin{align*}
  \Ext^{3}_{\RRR_\QQQ}\left(\MMM,\RRR_\QQQ \right) =0,  \text{ for all $\RRR_\QQQ$-modules } \MMM.
\end{align*}
See Lemma 2 in Chapter 19 of Matsumura's book \cite{matsumura1989commutative}. These observations lets us deduce the following corollary to Lemma \ref{van-ext0-1-pn}.

\begin{corollary} \label{pseudo-null-ext2}
  Let $\QQQ$ be a height two prime ideal in $\RRR$ such that the localization $\RRR_\QQQ$ is a Gorenstein local ring. Suppose we have a short exact sequence $0 \rightarrow \ZZZ' \rightarrow \ZZZ \rightarrow \ZZZ'' \rightarrow 0$ of finitely generated pseudo-null $\RRR$-modules. Then, we have the following short exact sequence of finitely generated $\RRR_\QQQ$-modules:
  \begin{align*}
    0 \rightarrow \Ext^2_{\RRR_\QQQ}\left(\ZZZ'' \otimes_\RRR \RRR_\QQQ,\RRR_\QQQ \right) \rightarrow \Ext^2_{\RRR_\QQQ}\left(\ZZZ \otimes_\RRR \RRR_\QQQ,\RRR_\QQQ \right) \rightarrow \Ext^2_{\RRR_\QQQ}\left(\ZZZ' \otimes_\RRR \RRR_\QQQ,\RRR_\QQQ \right)  \rightarrow 0.
  \end{align*}
\end{corollary}

\begin{remark}
  A regular local ring is Gorenstein. See Corollary 4.4.17 in Weibel's book \cite{weibel1995introduction}.
\end{remark}

\subsection{Cokernels of maps defining reflexive hulls}\label{cokernelmap}

The ring $\RRR$ is assumed to be an integrally closed, Noetherian, local domain in Section \ref{cokernelmap}.

\begin{proposition} \label{Ext-second-Chern}
 Suppose the finitely generated $\RRR$-module $\XXX$ is torsion-free. Suppose also that for every height two prime ideal $\QQQ$ of $\RRR$, the $\RRR_\QQQ$-modules $\XXX_\QQQ$ and $\coker(i_\XXX)_\QQQ$ have finite projective dimension.
  Then, we have the following equality in $Z^2(\RRR)$:
  \begin{align*}
    c_2\big(\coker(i_\XXX)\big) = c_2\left(\Ext_\RRR^{1}\big(\XXX,\RRR\right)\big).
  \end{align*}
\end{proposition}

\begin{proof}

  We will divide the proof into two parts:
  \begin{enumerate}
    \item[Part One:] For every height two prime ideal $\QQQ$ in $\RRR$, we have the following isomorphism of $\RRR_\QQQ$-modules:
          \begin{align} \label{firstpart-outline}
            \Ext^{1}_\RRR\big(\XXX,\RRR\big) \otimes_\RRR \RRR_\QQQ \cong \Ext^{2}_\RRR\big(\coker(i_\XXX),\RRR\big) \otimes_\RRR \RRR_\QQQ.
          \end{align}
    \item[Part Two:]  For every height two prime ideal $\QQQ$ in $\RRR$, we have the following equality of lengths of $\RRR_\QQQ$-modules:
          \begin{align} \label{secondpart-outline}
            \mathrm{Len}_{\RRR_\QQQ} \bigg(\Ext^{2}_\RRR\big(\coker(i_\XXX),\RRR\big) \otimes_\RRR \RRR_\QQQ \bigg) = \mathrm{Len}_{\RRR_\QQQ} \bigg(\coker(i_\XXX) \otimes_\RRR \RRR_\QQQ\bigg).
          \end{align}
  \end{enumerate}

  The first part of the proof follows from Lemma \ref{part-one}. As Lemma \ref{torfree-pn} indicates,  $\coker(i_\XXX)$ is a finitely generated pseudo-null $\RRR$-module. The second part of the proof would then follow from Lemma \ref{part-two}. It is clear that the proposition would then follow from equations (\ref{firstpart-outline}) and (\ref{secondpart-outline}).

  \begin{lemma} \label{part-one}
    Follow all the notations and hypotheses of Proposition \ref{Ext-second-Chern}. For every height two prime ideal $\QQQ$ in $\RRR$, we have the following isomorphism of $\RRR_\QQQ$-modules:
    \begin{align} \label{firstpart}
      \Ext^{1}_\RRR\big(\XXX,\RRR\big) \otimes_\RRR \RRR_\QQQ \cong \Ext^{2}_\RRR\big(\coker(i_\XXX),\RRR\big) \otimes_\RRR \RRR_\QQQ.
    \end{align}
  \end{lemma}

  \begin{proof}
     Let $\QQQ$ be a height two prime ideal in $\RRR$. Note that since $\XXX$ is a torsion-free $\RRR$-module, $\ker(i_\XXX)$ equals zero. Consider the localization of the short exact sequence $0 \rightarrow \XXX \rightarrow  \XXX^{**} \rightarrow \coker(i_\XXX) \rightarrow 0$ at the prime ideal $\QQQ$ of $\RRR$. We obtain the following short exact sequence of $\RRR_\QQQ$-modules:
    \begin{align} \label{xxx-dual}
      0 \rightarrow \XXX_\QQQ \rightarrow  \XXX^{**} \otimes_\RRR \RRR_\QQQ \rightarrow \coker(i_\XXX)\otimes_\RRR \RRR_\QQQ \rightarrow 0.
    \end{align}

  The hypotheses of the lemma tell us that the $\RRR_\QQQ$-module $\XXX^{**} \otimes_\RRR \RRR_\QQQ$ has finite projective dimension and hence is free (by Lemma \ref{van-ext-reflexive}). We have the following equalities, for all $i \geq 1$:
    \begin{align*}
      \Ext^{i}_{\RRR_\QQQ}\big(\XXX^{**} \otimes_\RRR \RRR_\QQQ,\RRR_\QQQ\big)=0.
    \end{align*}

    Applying the functor $\Hom_{\RRR_\QQQ}\left(-,\RRR_\QQQ\right)$ to the short exact sequence (\ref{xxx-dual}) given above, we get the following isomorphism of $\RRR_\QQQ$-modules, for all $i \geq 1$:
    \begin{align}
      \Ext^{i}_\RRR\big(\XXX,\RRR\big) \otimes_\RRR \RRR_\QQQ \cong \Ext^{i+1}_\RRR\big(\coker(i_\XXX),\RRR\big) \otimes_\RRR \RRR_\QQQ.
    \end{align}
    In particular, we have the following isomorphism of $\RRR_\QQQ$-modules:
    \begin{align*}
      \Ext^{1}_\RRR\big(\XXX,\RRR\big) \otimes_\RRR \RRR_\QQQ \cong \Ext^{2}_\RRR\big(\coker(i_\XXX),\RRR\big) \otimes_\RRR \RRR_\QQQ.
    \end{align*}
  \end{proof}

  \begin{lemma}\label{part-two}
    Let $\ZZZ$ be a finitely generated pseudo-null $\RRR$-module. Let $\QQQ$ be a height two prime ideal in $\RRR$ such that $\RRR_\QQQ$ is Gorenstein. Then, we have the following equality of lengths of $\RRR_\QQQ$-modules:
    \begin{align} \label{secondpart}
      \mathrm{Len}_{\RRR_\QQQ} \bigg(\Ext^{2}_{\RRR_\QQQ}\big(\ZZZ \otimes_\RRR \RRR_\QQQ,\RRR_\QQQ\big)  \bigg) = \mathrm{Len}_{\RRR_\QQQ} \bigg(\ZZZ \otimes_\RRR \RRR_\QQQ\bigg).
    \end{align}
  \end{lemma}

  \begin{proof}
Without loss of generality, assume $\ZZZ \otimes_{\RRR} \RRR_\QQQ$ is not zero.  As Lemma \ref{finite-length} indicates, we have
    \begin{align*}
      \len_{\RRR_\QQQ} \left(\ZZZ \otimes_{\RRR} \RRR_\QQQ \right)< \infty, \qquad \len_{\RRR_\QQQ} \left(\Ext^2_{\RRR_\QQQ}\left(\ZZZ \otimes_{\RRR} \RRR_\QQQ,\RRR_\QQQ \right) \right) < \infty.
    \end{align*}
    Suppose the following chain of $\RRR_\QQQ$-modules is  a composition series for $\ZZZ \otimes_{\RRR} \RRR_\QQQ$:
    \begin{align*}
      \ZZZ \otimes_{\RRR} \RRR_\QQQ = \MMM_0 \supsetneq \MMM_1 \supsetneq \cdots \supsetneq \MMM_n =0.
    \end{align*}

    Consider the following short exact sequence of finitely generated pseudo-null $\RRR_\QQQ$-modules:
    \begin{align} \label{pn-mod-secondpart}
      0 \rightarrow \MMM_1 \rightarrow \ZZZ \otimes_{\RRR} \RRR_\QQQ  \rightarrow \frac{\MMM_0}{\MMM_1}  \rightarrow 0.
    \end{align}
    By Corollary \ref{pseudo-null-ext2}, we have the following short exact sequence of finitely generated pseudo-null $\RRR_\QQQ$-modules:
    \begin{align} \label{ext2-pn-secondpart}
      0 \rightarrow \Ext^{2}_{\RRR_\QQQ}\left(\frac{\MMM_0}{\MMM_1} ,\RRR_\QQQ\right) \rightarrow \Ext^{2}_{\RRR_\QQQ}\left(\ZZZ \otimes_{\RRR} \RRR_\QQQ,\RRR_\QQQ\right)   \rightarrow \Ext^{2}_{\RRR_\QQQ}\left(\MMM_1,\RRR_\QQQ\right)  \rightarrow 0.
    \end{align}
    If $\MMM_1$ equals zero, then the $\RRR_\QQQ$-module $\ZZZ \otimes_{\RRR} \RRR_\QQQ$ is isomorphic to $k_\QQQ$. Otherwise, note that
    \begin{align*}
      \len_{\RRR_\QQQ} \left(\MMM_1\right) < \len_{\RRR_\QQQ} \ZZZ \otimes_\RRR \RRR_\QQQ, \qquad \len_{\RRR_\QQQ} \left(\frac{\MMM_0}{\MMM_1}\right) < \len_{\RRR_\QQQ} \ZZZ \otimes_\RRR \RRR_\QQQ.
    \end{align*}
    Suppose we could establish the following equality of lengths:
    \begin{align} \label{devissage-secondpart}
      \len_{\RRR_\QQQ} \left(\MMM_1\right) \stackrel{?}{=} \len_{\RRR_\QQQ} \left(\Ext^{2}_{\RRR_\QQQ}\left(\MMM_1,\RRR_\QQQ\right)\right), \qquad \len_{\RRR_\QQQ}\left(\frac{\MMM_0}{\MMM_1} \right) \stackrel{?}{=} \len_{\RRR_\QQQ} \left(\Ext^{2}_{\RRR_\QQQ}\left(\frac{\MMM_0}{\MMM_1} ,\RRR_\QQQ\right)\right).
    \end{align}
    As Lemma \ref{length-additive} indicates, length is additive in  exact sequences. Equality in (\ref{devissage-secondpart}), along with the short exact sequences (\ref{pn-mod-secondpart}) and (\ref{ext2-pn-secondpart}), would then let us obtain equation (\ref{secondpart}).

    These observations let us use a d\'evissage argument to reduce to the case when $\ZZZ \otimes_\RRR \RRR_\QQQ$ equals $k_\QQQ$. The $\RRR_\QQQ$-module $\RRR_\QQQ$ is a canonical module (sometimes also called the dualizing module) for the Gorenstein local ring $\RRR_\QQQ$. See Theorem 3.3.7 in the book by Bruns and Herzog \cite{MR1251956}. Consequently, we have the following isomorphism of $\RRR_\QQQ$-modules:

    \begin{align}\label{ext2-dualizing}
      \Ext^2_{\RRR_\QQQ} \left(k_\QQQ,\RRR_\QQQ\right) \cong k_\QQQ.
    \end{align}

    As a result, we have the following equality of lengths of $\RRR_\QQQ$-modules:
    \begin{align*}
      \len_{\RRR_\QQQ} \left(k_\QQQ\right) = \len_{\RRR_\QQQ} \left(\Ext^2_{\RRR_\QQQ}\left(k_\QQQ,\RRR_\QQQ \right)\right) = 1.
    \end{align*}
    This lets us deduce equation (\ref{secondpart}) when $\ZZZ \otimes_\RRR \RRR_\QQQ$ equals $k_\QQQ$. Hence, Lemma \ref{part-two} follows.
  \end{proof}
  Proposition \ref{Ext-second-Chern} follows.
\end{proof}

We would like to state one more application of Lemma \ref{part-two}. For each finitely generated $\RRR$-module $\MMM$, we let $\MMM_{\pn}$ denote the maximal pseudo-null $\RRR$-submodule of $\MMM$.

\begin{lemma} \label{ext2-is-PN}
For every finitely generated $\RRR$-module $\MMM$,  the $\RRR$-module $\Ext^2_{\RRR}\left(\MMM,\RRR\right)$ is pseudo-null. Furthermore, if we suppose that for every height two prime ideal $\QQQ$ in $\RRR$,
  \begin{itemize}\item  $\RRR_\QQQ$ is a Gorenstein local ring, and
    \item the $\RRR_\QQQ$-module $\frac{\MMM}{\MMM_{\pn}} \otimes_\RRR \RRR_\QQQ$ has finite projective dimension.
  \end{itemize}
  Then, we have the following equality in $Z^2\left(\RRR\right):$
  \begin{align}\label{c2equality-pn}
    c_2\left(\MMM_{\pn}\right) = c_2\left(\Ext^2_{\RRR}\left(\MMM,\RRR\right)\right).
  \end{align}
\end{lemma}

\begin{proof}
  Let $\p$ be a height one prime ideal in the ring $\RRR$. Since $\RRR$ is a Noetherian integrally closed local domain, the localization $\RRR_\p$ is a DVR. As a result the localization $\Ext^i_{\RRR}\left(\MMM,\RRR\right) \otimes_\RRR \RRR_\p$, which is isomorphic to the $\RRR_\p$-module $\Ext^i_{\RRR_\p}\left(\MMM_\p,\RRR_\p\right)$, vanishes for all $i \geq 2$. In particular, $\Ext^2_{\RRR_\p}\left(\MMM_\p,\RRR_\p\right)$ vanishes. This lets us conclude that $\RRR$-module $\Ext^2_{\RRR}\left(\MMM,\RRR\right)$ is pseudo-null. \\

  Let $\QQQ$ be a height two prime ideal in $\RRR$. Since $\RRR_\QQQ$ is a Gorenstein ring, the injective dimension of the $\RRR_\QQQ$-module $\RRR_\QQQ$ must equal two. By Lemma 2 in Chapter 19 of Matsumura's book \cite{matsumura1989commutative} $$\Ext^3_{\RRR_\QQQ}\left(\frac{\MMM}{\MMM_{\pn}} \otimes_\RRR \RRR_\QQQ,\RRR_\QQQ\right)=0.$$

We will argue that $\pd_{\RRR_\QQQ} \frac{\MMM}{\MMM_{\pn}} \otimes_\RRR \RRR_\QQQ \leq 1$. This is straightforward if $\frac{\MMM}{\MMM_{\pn}} \otimes_\RRR \RRR_\QQQ$ equals zero. Let us consider the case when $\frac{\MMM}{\MMM_{\pn}} \otimes_\RRR \RRR_\QQQ$ is not zero.  Since $\MMM_{\pn}$ is the maximal pseudo-null $\RRR$-submodule of $\MMM$, the $\RRR$-module $\frac{\MMM}{\MMM_\pn}$ has no non-zero pseudo-null submodules. As a result the $\RRR_\QQQ$-module $\frac{\MMM}{\MMM_{\pn}} \otimes_\RRR \RRR_\QQQ$ has no non-zero pseudo-null submodules. Hence, we have $\Depth_{\RRR_\QQQ} \frac{\MMM}{\MMM_{\pn}} \otimes_\RRR \RRR_\QQQ \geq 1$. The hypotheses of the lemma tell us that the $\RRR_\QQQ$-module $\frac{\MMM}{\MMM_{\pn}} \otimes_\RRR \RRR_\QQQ$ has finite projective dimension. One can apply the Auslander-Buchsbaum equality over the two-dimensional local ring $\RRR_\QQQ$ to conclude that $\pd_{\RRR_\QQQ} \frac{\MMM}{\MMM_{\pn}} \otimes_\RRR \RRR_\QQQ \leq 1$.

  By Lemma 2 in Chapter 19 of Matsumura's book \cite{matsumura1989commutative}, we have
  $$\Ext^2_{\RRR_\QQQ}\left(\frac{\MMM}{\MMM_{\pn}} \otimes_\RRR \RRR_\QQQ,\RRR_\QQQ\right)=0.$$

  Applying the functor $\Hom_{\RRR_\QQQ}\left(-,\RRR_\QQQ\right)$ to the short exact sequence $0 \rightarrow \MMM_{\pn}  \otimes_\RRR \RRR_\QQQ \rightarrow \MMM \otimes_\RRR \RRR_\QQQ \rightarrow \frac{\MMM}{\MMM_{\pn}} \otimes_\RRR \RRR_\QQQ \rightarrow 0$ of $\RRR_\QQQ$-modules lets us obtain the following isomorphism of $\RRR_\QQQ$-modules:
  \begin{align} \label{ext-c2-pnsubmodule}
  \Ext^2_{\RRR_\QQQ}\left(\MMM \otimes_\RRR \RRR_\QQQ,\RRR_\QQQ\right) \cong   \Ext^2_{\RRR_\QQQ}\left(\MMM_{\pn} \otimes_\RRR \RRR_\QQQ,\RRR_\QQQ\right).
  \end{align}
  Using Lemma \ref{part-two}, we obtain the following equality in $Z^2\left(\RRR\right)$:
  \begin{align*}
    c_2\bigg(M_{\pn}\bigg) = c_2\bigg(\Ext^2_{\RRR}\left(\MMM_{\pn},\RRR\right)\bigg) = c_2\bigg(\Ext^2_{\RRR}\left(\MMM,\RRR\right)\bigg).
  \end{align*}
  This completes the proof of the Lemma.
\end{proof}

\section{The General setup} \label{generalsetup}

We will obtain our results from a general perspective in Section \ref{generalresults}. In this section, we will simply outline the various objects involved in describing the general results. Let $\RRR$ denote a Noetherian, complete, integrally closed, local domain,   with Krull dimension $n+1$, characteristic zero and whose residue field is finite with characteristic $p$. Cohen's structure theorems tell us that there exists a subring $\Lambda_n$ of $\RRR$ that is isomorphic to the power series $\Z_p[[x_1,\ldots,x_n]]$. Let $\m_\RRR$ denote the maximal ideal of $\RRR$. Let $\Sigma$ denote a finite set of primes of $\Q$, containing $p$, $\infty$ and a finite prime $l_0 \neq p$. Let $\Q_\Sigma$ denote the maximal extension of $\Q$ unramified outside $\Sigma$. Let $G_\Sigma$ denote the Galois group $\Gal{\Q_\Sigma}{\Q}$. \\

Consider a continuous Galois representation
\begin{align*}
  \rhodn: G_\Sigma \rightarrow \Gl_d(\RRR).
\end{align*}
Let $T_{\rhoDN{d}{n}}$ denote the underlying free $\RRR$-module of rank $n$ on which $G_\Sigma$-module acts to let us obtain $\rhodn$. We will let $d^{+}(\rhodn)$ be the rank of $\RRR$-submodule of $T_{\rhoDN{d}{n}}$ fixed by complex conjugation. We let $d^-(\rhodn)$ equal $d - d^+(\rhodn)$. Note that $\RRR^\vee$ denotes $\Hom_{\cont}\left(\RRR,\frac{\Q_p}{\Z_p}\right)$.

We shall also consider the Galois representation $\rhodnstar: G_\Sigma \rightarrow \Gl_d(\RRR)$, given by the action of $G_\Sigma$ on $T_{\rhoDNstar{d}{n}}:=\Hom_{\RRR}\left(T_{\rhoDN{d}{n}},\RRR(\chi_p)\right)$.  Here, $\RRR(\chi_p)$ is the free $\RRR$-module of rank one on which $G_\Sigma$ acts via the $p$-adic cyclotomic character $\chi_p:G_\Sigma \rightarrow \Z_p^\times$. To each of these Galois representation, one can attach the following discrete modules:

\begin{align*}
  D_{\rhodn} := T_{\rhodn} \otimes_\RRR \RRR^\vee, \qquad D_{\rhodnstar} := T_{\rhodnstar} \otimes_\RRR \RRR^\vee.
\end{align*}

The various modules appearing in the general results are described in Section \ref{variousmodules}. To define these modules, we will need to consider various subgroups of the first discrete global Galois cohomology group $H^1\left(G_\Sigma, D_{\rhodn}\right)$. In turn, to define the subgroups of the first global Galois cohomology group,  we will need to consider local Selmer conditions, which are subgroups of the first discrete local cohomology group at the prime $p$. We will simply suppose that we have two discrete $\RRR$-submodules, denoted $\Loc_{\One}\left(\Q_p,D_{\rhodn}\right)$ and $\Loc_{\Two}\left(\Q_p,D_{\rhodn}\right)$, inside the first local Galois cohomology group $H^1\left(\Q_p,D_{\rhodn}\right)$. That is, we have the following inclusions of $\RRR$-modules:
\begin{align*} \Loc_{\One}\left(\Q_p,D_{\rhodn}\right) \subset H^1\left(\Q_p,D_{\rhodn}\right), \qquad  \Loc_{\Two}\left(\Q_p,D_{\rhodn}\right) \subset H^1\left(\Q_p,D_{\rhodn}\right).
\end{align*}

The properties, that these local factors at $p$ need to satisfy, will be described in Section \ref{varioushypotheses}.

\subsection{Various modules, assumptions and hypotheses}

\subsubsection{Descriptions of the various modules} \label{variousmodules}

We let $\XXX(\Q,D_\rhodn)$ denote the Pontryagin dual of
\begin{align*}
  \ker\bigg(H^1\left(G_\Sigma, D_{\rhodn}\right) \xrightarrow {\phi_\XXX}  \frac{H^1\left(\Q_p,D_{\rhodn}\right)}{\Loc_{\One}\left(\Q_p,D_{\rhodn}\right)+\Loc_{\Two}\left(\Q_p,D_{\rhodn}\right)} \oplus \bigoplus_{l \in \Sigma \setminus \{p\}} H^1\left(\Q_l,D_{\rhodn}\right) \bigg).
\end{align*}

We let $\XXX(\Q,D_\rhodn)_{\tor}$ denote the $\RRR$-torsion submodule of $\XXX(\Q,D_\rhodn)$. \\

We shall define two discrete Selmer groups $\Sel_{\One}\left(\Q,D_{\rhodn}\right)$ and $\Sel_{\Two}\left(\Q,D_{\rhodn}\right)$ as follows:

\begin{align*}
  \Sel_{\One}\left(\Q,D_{\rhodn}\right):=\ker\bigg(H^1\left(G_\Sigma, D_{\rhodn}\right) & \xrightarrow {\phi_\One}  \frac{H^1\left(\Q_p,D_{\rhodn}\right)}{\Loc_{\One}\left(\Q_p,D_{\rhodn}\right)} \oplus \bigoplus_{l \in \Sigma \setminus \{p\}} H^1\left(\Q_l,D_{\rhodn}\right)  \bigg), \qquad \\
  \Sel_{\Two}(\Q,D_{\rhodn}):=\ker\bigg(H^1\left(G_\Sigma, D_{\rhodn}\right)            & \xrightarrow {\phi_\Two}  \frac{H^1\left(\Q_p,D_{\rhodn}\right)}{\Loc_{\Two}\left(\Q_p,D_{\rhodn}\right)} \oplus \bigoplus_{l \in \Sigma \setminus \{p\}} H^1\left(\Q_l,D_{\rhodn}\right)  \bigg).
\end{align*}

For each $i \in \{1,2\}$, we will define the following discrete subgroup, labeled $\Sha^i\left(\Q,D_{\rhodn}\right)$, of the  global Galois cohomology group $H^i\left(G_\Sigma, D_{\rhodn}\right)$:
\begin{align*}
  \Sha^i\left(\Q,D_{\rhodn}\right):=\ker\bigg(H^i\left(G_\Sigma, D_{\rhodn}\right) & \rightarrow    \bigoplus_{l \in \Sigma} H^i\left(\Q_l,D_{\rhodn}\right) \bigg).
\end{align*}

For each $i \in \{1,2\}$, one can similarly define $\Sha^i\left(\Q,D_{\rhodnstar}\right)$ inside the  global Galois cohomology group $H^i\left(G_\Sigma, D_{\rhodnstar}\right)$. \\

We let $\ZZZ(\Q,D_{\rhodn})$ denote the Pontryagin dual of
\begin{align} \label{ZD}
  \ker\bigg(H^1\left(G_\Sigma, D_{\rhodn}\right) \rightarrow   \frac{H^1\left(\Q_p,D_{\rhodn}\right)}{\Loc_{\One}\left(\Q_p,D_{\rhodn}\right) \bigcap \Loc_{\Two}\left(\Q_p,D_{\rhodn}\right)} \oplus \bigoplus_{l \in \Sigma \setminus \{p\}} H^1\left(\Q_l,D_{\rhodn}\right) \bigg).
\end{align}

It will be helpful to keep following surjections of $\RRR$-modules in mind:
\begin{align}
  H^1\left(G_\Sigma, D_{\rhodn} \right)^\vee \twoheadrightarrow\XXX(\Q,D_{\rhodn})  \twoheadrightarrow \substack{\Sel_{\One}\left(\Q,D_{\rhodn}\right)^\vee \\ \Sel_{\Two}\left(\Q,D_{\rhodn}\right)^\vee} \twoheadrightarrow \ZZZ(\Q,D_{\rhodn}) \twoheadrightarrow \Sha^{1}\left(\Q,D_{\rhodn}\right)^\vee
\end{align}

\begin{remark}
  If one follows Greenberg's definition of Selmer groups in \cite{greenberg1994iwasawa}, one requires the global cocycles to be unramified at primes $l \in \Sigma \setminus \{p\}$.   This point will not matter to us since for our applications, we will be considering Galois representations $\rhodn$ that are related to cyclotomic deformations (following the notations in Section 3 of \cite{greenberg1994iwasawa}). For such representations, the natural restriction map  $H^1\left(\Q_l,D_{\rhodn} \right) \rightarrow H^1\left(I_l,D_{\rhodn} \right)$ turns out to be injective, whenever $l \neq p$. Here, $I_l$ is the inertia subgroup inside the decomposition group $\Gal{\overline{\Q}_l}{\Q_l}$.
\end{remark}

\subsubsection{Statements of the various assumptions and hypotheses} \label{varioushypotheses}

We will deduce our results under various conditions. Some of these conditions will have the prefix ``Assumption''. Some of these conditions will have the prefix ``Hypothesis''. The conditions with the prefix ``Hypothesis'' are those conditions which we will be able to establish in the setting of Theorem \ref{maintheorem}.  In the setting of Theorem \ref{maintheorem}, the conditions labeled ``Assumptions'' are currently not known to always hold unconditionally.

\begin{enumerate}[leftmargin=4cm, style=sameline, align=left, label=\underline{$\mathrm{Assumption \ MC}$}, ref=$\mathrm{Assumption \ MC}$]
  \item\label{mc} We have the following equality of ranks
  \begin{align}\label{mc-rank-equality}
    \mathrm{Rank}_\RRR \ {\Sel_{\One}\left(\Q,D_{\rhodn}\right)^\vee}  = 0, \qquad \mathrm{Rank}_\RRR \ {\Sel_{\Two}\left(\Q,D_{\rhodn}\right)^\vee} = 0.
  \end{align}In addition,  there exist two elements $\theta_\One$ and $\theta_\Two$ in $\RRR$ such that we have the following equalities in $Z^1(\RRR)$:
  \begin{align}\label{mc-div-equality}
    \Div\left(\Sel_{\One}(\Q,D_{\rhodn})^\vee\right) = \Div(\theta_\One), \qquad
    \Div\left(\Sel_{\Two}(\Q,D_{\rhodn})^\vee\right) = \Div(\theta_\Two).
  \end{align}
\end{enumerate}

\begin{enumerate}[leftmargin=4cm, style=sameline, align=left, label=\underline{$\mathrm{Assumption\ GCD}$}, ref=$\mathrm{Assumption \ GCD }$]
  \item\label{gcd} The height of the ideal $(\theta_{\One},\theta_{\Two})$ in $\RRR$ is greater than or equal to two.
\end{enumerate}
In the statement of \ref{gcd}, $\theta_\One$ and $\theta_\Two$ are the elements appearing in the statement of \ref{mc}.

\begin{enumerate}[leftmargin=4cm, style=sameline, align=left, label=\underline{$\mathrm{Hypothesis \ Rank}$}, ref=$\mathrm{Hypothesis \ Rank}$]
  \item\label{rank} We have the following equality of ranks:
  \begin{align*}\mathrm{Rank}_\RRR \ {\Loc_{\One}\left(\Q_p,D_{\rhodn}\right)^\vee}  = \mathrm{Rank}_\RRR \ {\Loc_{\Two}\left(\Q_p,D_{\rhodn}\right)^\vee} = d^+, \\
    \mathrm{Rank}_{\RRR} \ \bigg( \Loc_{\One}\left(\Q_p,D_{\rhodn}\right) + \Loc_{\Two}\left(\Q_p,D_{\rhodn}\right) \bigg)^\vee = d^+ + 1.
  \end{align*}
\end{enumerate}

\begin{enumerate}[leftmargin=3cm, style=sameline, align=left, label=\underline{$\mathrm{Hypothesis \ LF}$}, ref=$\mathrm{Hypothesis \ LF}$]
  \item\label{loc-free} The following $\RRR$-modules are free:
  \begin{align*}
      & \Loc_{\One}\left(\Q_p,D_{\rhodn}\right)^\vee, && \Loc_{\Two}\left(\Q_p,D_{\rhodn}\right)^\vee, \\ & \bigg(\frac{\Loc_{\Two}\left(\Q_p,D_{\rhodn}\right)}{{\Loc_{\One}\left(\Q_p,D_{\rhodn}\right)\bigcap \Loc_{\Two}\left(\Q_p,D_{\rhodn}\right)}}\bigg)^\vee, \quad &&\bigg(\frac{\Loc_{\One}\left(\Q_p,D_{\rhodn}\right)}{{\Loc_{\One}\left(\Q_p,D_{\rhodn}\right)\bigcap \Loc_{\Two}\left(\Q_p,D_{\rhodn}\right)}}\bigg)^\vee.
  \end{align*}
\end{enumerate}

\begin{enumerate}[leftmargin=4cm, style=sameline, align=left, label=\underline{$\mathrm{Hypothesis \ Loc_p(0)}$}, ref=$\mathrm{Hypothesis \ Loc_p(0)}$]
  \item\label{locp0} The $\Gal{\overline{\Q}_p}{\Q_p}$-modules $D_{\rhodn}[\m_\RRR]$ and $D_{\rhodnstar}[\m_\RRR]$ have no quotient isomorphic to the trivial representation.
\end{enumerate}

\begin{enumerate}[leftmargin=4cm, style=sameline, align=left, label=\underline{$\mathrm{Hypothesis \ Loc(0)}$}, ref=$\mathrm{Hypothesis \ Loc(0)}$]
  \item\label{tor-loc-sigma} For every $l$ in $\Sigma$, we have $$\mathrm{Rank}_\RRR H^0\left(\Q_l,D_\rhodn\right)^\vee=\mathrm{Rank}_\RRR H^0\left(\Q_l,D_\rhodnstar\right)^\vee = 0.$$
\end{enumerate}

\begin{enumerate}[leftmargin=4cm, style=sameline, align=left, label=\underline{$\mathrm{Hypothesis \ Reg(0)}$}, ref=$\mathrm{Hypothesis \ Reg(0)}$]
  \item\label{Sigma0-pd}  For every prime $l \in \Sigma \setminus \{p\}$ and every height two prime ideal $\QQQ$ in $\RRR$, the $\RRR_\QQQ$-module $H^0\left(\Q_l,D_{\rhodn}\right)^\vee \otimes_\RRR \RRR_\QQQ$ has finite projective dimension.
\end{enumerate}

\begin{enumerate}[leftmargin=4cm, style=sameline, align=left, label=\underline{$\mathrm{Hypothesis \ Gor}$}, ref=$\mathrm{Hypothesis \ Gor}$]
  \item\label{Gor}  $\RRR$ is a Gorenstein local ring.
\end{enumerate}

\begin{remark}
  In the setup of Theorem~\ref{maintheorem}, the ring $\RRR$ (which is isomorphic to the power series ring $\Z_p[[x_1,x_2]]$) is a regular local ring. \ref{Gor} and \ref{Sigma0-pd} are automatically valid.
\end{remark}

\begin{remark}
  When the ring $\RRR$ is a UFD, \ref{gcd} is equivalent to the statement that the elements $\theta_\One$ and $\theta_\Two$ have no common irreducible factor.
\end{remark}

\begin{remark}
  We would like to make a few remarks concerning \ref{mc}. Since our approach towards proving Theorem \ref{maintheorem} only involves studying the module theory of Galois cohomology groups, we have not defined $p$-adic $L$-functions in the general setup. One must view the statements in \ref{mc} simply as abstract formulations of Iwasawa main conjectures (just as in Section 3 of the seven author paper \cite{bleher2015higher}). The content of equation (\ref{mc-div-equality}) in \ref{mc} is significant only when the ring $\RRR$ is not a UFD.

  When the Galois representation $\rhodn$ satisfies the Panchishkin condition, the Iwasawa main conjecture (formulated in \cite{greenberg1994iwasawa}) predicts an equality between the divisor associated to the $p$-adic $L$-function (say $\theta_\One$) and the divisor
  \begin{align*}
    \Div\left(\Sel_\One(\Q,D_\rhodn)^\vee\right) - \Div\left(H^0\left(G_\Sigma, D_{\rhodn}\right)^\vee\right) - \Div\left(H^0\left(G_\Sigma, D_{\rhodnstar}\right)^\vee\right)
  \end{align*}
  in $Z^1\left(\RRR\right)$. In our situation, \ref{locp0} lets us deduce that
  \begin{align*}
    H^0\left(G_\Sigma, D_{\rhodn}\right)=H^0\left(G_\Sigma, D_{\rhodnstar}\right)= 0.
  \end{align*}
\end{remark}

\begin{remark}
  While it seems reasonable to expect \ref{mc} to always hold in the setup of Theorem \ref{maintheorem}, we do not have any reason to believe that \ref{gcd} would always hold in this setup. However, we do produce some evidence towards the validity of \ref{gcd} in the setup of Theorem~\ref{maintheorem} in Section \ref{section-examples}.
\end{remark}

\subsection{Consequences of the various hypotheses and assumptions} \label{consequences-hyp-assumptions}

\begin{enumerate}[(a)]
  \item \label{obs-a} Note that we have a natural isomorphism :
        \begin{align*}
          \left(\frac{\Loc_{\One}\left(\Q_p,D_{\rhodn}\right) +  \Loc_{\Two}\left(\Q_p,D_{\rhodn}\right)}{\Loc_{\One}\left(\Q_p,D_{\rhodn}\right)}\right)^\vee \cong \bigg(\frac{\Loc_{\Two}\left(\Q_p,D_{\rhodn}\right)}{{\Loc_{\One}\left(\Q_p,D_{\rhodn}\right)\bigcap \Loc_{\Two}\left(\Q_p,D_{\rhodn}\right)}}\bigg)^\vee.
        \end{align*}
        This observation, \ref{loc-free} and the (split) short exact sequence
        {\small \begin{align*}
          & 0  \rightarrow \left(\frac{\Loc_{\One}\left(\Q_p,D_{\rhodn}\right) +  \Loc_{\Two}\left(\Q_p,D_{\rhodn}\right)}{\Loc_{\One}\left(\Q_p,D_{\rhodn}\right)}\right)^\vee \rightarrow \left(\Loc_{\One}\left(\Q_p,D_{\rhodn}\right) +  \Loc_{\Two}\left(\Q_p,D_{\rhodn}\right)\right)^\vee  \rightarrow \Loc_{\One}\left(\Q_p,D_{\rhodn}\right)^\vee \rightarrow 0,
          \end{align*}}
        let us conclude that the $\RRR$-module $\left(\Loc_{\One}\left(\Q_p,D_{\rhodn}\right) +  \Loc_{\Two}\left(\Q_p,D_{\rhodn}\right)\right)^\vee$ is free.

  \item As a result of local duality theorems (for example, see Section 0.3 in Nekov{\'a}{\v{r}}'s work on Selmer complexes \cite{nekovar2006selmer}), \ref{tor-loc-sigma} and Proposition 3.10 in \cite{MR2290593}, for all primes $l \in \Sigma$, we have
        \begin{align} \label{primes-consequence}
          H^0_\ct\left(\Q_l,T_{\rhodnstar}\right)= H^0_\ct \left(\Q_l,T_\rhodn\right) = 0, \qquad   H^2\left(\Q_l,D_\rhodn\right) = H^2\left(\Q_l,D_\rhodnstar\right) = 0.
        \end{align}
        As a result, for the zeroth global Galois cohomology groups, we also have
        \begin{align}\label{h0-global-compact-consequence}
          H^0_\ct\left(G_\Sigma,T_\rhodn\right) = H^0_\ct\left(G_\Sigma,T_\rhodnstar\right) = 0.
        \end{align}
  \item \ref{locp0} lets us conclude that
        \begin{align} \label{local-p-h0-consequence}
            & H^0\left(\Q_p, D_{\rhodn} \right) = H^0\left(\Q_p, D_{\rhodnstar} \right) = 0, \qquad
            & H^2\left(\Q_p, D_{\rhodn} \right) = H^2\left(\Q_p, D_{\rhodnstar} \right) = 0.
        \end{align}
        As a result, for the zeroth global Galois cohomology groups, we also have
        \begin{align} \label{global-p-h0-consequence}
          H^0\left(G_\Sigma, D_{\rhodn} \right) = H^0\left(G_\Sigma, D_{\rhodnstar} \right) = 0.
        \end{align}
  \item By studying the local and global Euler Poincar\'e characteristics and using equation (\ref{mc-rank-equality}) and \ref{rank}, one can conclude that the Weak Leopoldt conjecture for $\rhodn$ holds. That is, the $\RRR$-module $\Sha^2\left(\Q,D_{\rhodn}\right)^\vee$ is torsion (in fact, equal to zero). See Proposition 4.6 in \cite{palvannan2016algebraic}. In fact, in our situation, Proposition 6.1 in \cite{MR2290593}  lets us conclude that
        \begin{align} \label{weak-leo-consequence}
          H^2\left(G_\Sigma,D_{\rhodn}\right)=0.
        \end{align}
  \item The global to local maps $\Phi_\One$ and $\Phi_\Two$, defining the Selmer groups $\Sel_\One(\Q,D_\rhodn)$ and $\Sel_\Two(\Q,D_\rhodn)$, are surjective. This follows by applying Proposition 3.2.1 in \cite{greenberg2010surjectivity} and using \ref{locp0}.
  \item All the hypotheses in Proposition 4.1.1 listed in \cite{greenberg2014pseudonull} can be verified to let us conclude that $\RRR$-module $\XXX(\Q,D_{\rhodn})$ has no non-zero pseudo-null submodules. The fact, that the $\RRR$-module  $ \left(\Loc_{\One}\left(\Q_p,D_{\rhodn}\right) +  \Loc_{\Two}\left(\Q_p,D_{\rhodn}\right)\right)^\vee $ is free, comes into play.
\end{enumerate}

\subsubsection{Various commutative diagrams}

Let $J \in \{\One,\Two\}$. To relate the various modules in this general setup, it will be helpful to keep the following commutative diagrams in mind:

\begin{align} \tag{\underline{Commutative\ diagram\ A}}
  \xymatrix{
  H^1\left(G_\Sigma, D_\rhodn\right) \ar[d]^{\phi_J}\ar[r]^{\cong}                                                                                                            & H^1\left(G_\Sigma, D_\rhodn\right) \ar[d]^{\phi_\XXX }                                                                                                                                                           \\
  \frac{H^1\left(\Q_p,D_{\rhodn}\right)}{\Loc_{J}\left(\Q_p,D_{\rhodn}\right)} \oplus \bigoplus \limits_{l \in \Sigma \setminus \{p\}} H^1\left(\Q_l,D_{\rhodn}\right) \ar[r] & \frac{H^1\left(\Q_p,D_{\rhodn}\right)}{\Loc_{\One}\left(\Q_p,D_{\rhodn}\right)+\Loc_{\Two}\left(\Q_p,D_{\rhodn}\right)}  \oplus \bigoplus \limits_{l \in \Sigma \setminus \{p\}} H^1\left(\Q_l,D_{\rhodn}\right)
  }
\end{align}

{\small \begin{align}\tag{\underline{Commutative\ diagram\ B}}
  \xymatrix{
    H^1\left(G_\Sigma, D_\rhodn\right) \ar[d]^{\phi_\ZZZ}\ar[r]^{\cong} &   H^1\left(G_\Sigma, D_\rhodn\right) \ar[d]^{\phi_\XXX } \\
    \frac{H^1\left(\Q_p,D_{\rhodn}\right)}{\Loc_{\One}\left(\Q_p,D_{\rhodn}\right)\bigcap \Loc_{\Two}\left(\Q_p,D_{\rhodn}\right)} \oplus \bigoplus \limits_{l \in \Sigma \setminus \{p\}} H^1\left(\Q_l,D_{\rhodn}\right) \ar[r]& \frac{H^1\left(\Q_p,D_{\rhodn}\right)}{\Loc_{\One}\left(\Q_p,D_{\rhodn}\right)+\Loc_{\Two}\left(\Q_p,D_{\rhodn}\right)}  \oplus \bigoplus \limits_{l \in \Sigma \setminus \{p\}} H^1\left(\Q_l,D_{\rhodn}\right)
  }
  \end{align}}

{\small \begin{align}\tag{\underline{Commutative\ diagram\ C}}
  \xymatrix{
    H^1\left(G_\Sigma, D_\rhodn\right) \ar[d]^{\phi_\ZZZ}\ar[r]^{\cong} &   H^1\left(G_\Sigma, D_\rhodn\right) \ar[d]^{\phi_J } \\
    \frac{H^1\left(\Q_p,D_{\rhodn}\right)}{\Loc_{\One}\left(\Q_p,D_{\rhodn}\right)\bigcap \Loc_{\Two}\left(\Q_p,D_{\rhodn}\right)} \oplus \bigoplus \limits_{l \in \Sigma \setminus \{p\}} H^1\left(\Q_l,D_{\rhodn}\right) \ar[r]& \frac{H^1\left(\Q_p,D_{\rhodn}\right)}{\Loc_{J}\left(\Q_p,D_{\rhodn}\right)}  \oplus \bigoplus \limits_{l \in \Sigma \setminus \{p\}} H^1\left(\Q_l,D_{\rhodn}\right)
  }
  \end{align}}

\begin{lemma} \label{ses-Sel-X}
  Suppose all the hypotheses and assumptions in Section \ref{varioushypotheses} hold.

  Then, the map $\phi_\XXX$ is surjective. We also have the following short exact sequences of $\RRR$-modules:
  \begin{align}
    0 \rightarrow \bigg(\frac{\Loc_{\Two}\left(\Q_p,D_{\rhodn}\right)}{{\Loc_{\One}\left(\Q_p,D_{\rhodn}\right)\bigcap \Loc_{\Two}\left(\Q_p,D_{\rhodn}\right)}}\bigg)^\vee \rightarrow \X(D_\rhodn,\Q) \rightarrow \Sel_\One(\Q,D_\rhodn)^\vee \rightarrow 0, \label{ses-selmer-one}                             \\
    0 \rightarrow\bigg(\frac{\Loc_{\One}\left(\Q_p,D_{\rhodn}\right)}{{\Loc_{\One}\left(\Q_p,D_{\rhodn}\right)\bigcap \Loc_{\Two}\left(\Q_p,D_{\rhodn}\right)}}\bigg)^\vee \rightarrow \X(D_\rhodn,\Q) \rightarrow \Sel_\Two(\Q,D_\rhodn)^\vee \rightarrow 0, \label{ses-selmer-two}                              \\
    \bigg(\frac{\Loc_{\One}\left(\Q_p,D_{\rhodn}\right) + \Loc_{\Two}\left(\Q_p,D_{\rhodn}\right)}{{\Loc_{\One}\left(\Q_p,D_{\rhodn}\right) \bigcap \Loc_{\Two}\left(\Q_p,D_{\rhodn}\right)}}\bigg)^\vee \rightarrow \XXX (\Q,D_{\rhodn}) \rightarrow \ZZZ(\Q,D_{\rhodn}) \rightarrow 0, \label{ses-selmer-three} \\
    \bigg(\frac{\Loc_{\One}\left(\Q_p,D_{\rhodn}\right) }{{\Loc_{\One}\left(\Q_p,D_{\rhodn}\right) \bigcap \Loc_{\Two}\left(\Q_p,D_{\rhodn}\right)}}\bigg)^\vee \rightarrow  \Sel_\One(\Q,D_\rhodn)^\vee \rightarrow \ZZZ(\Q,D_{\rhodn}) \rightarrow 0, \label{ses-selmer-four}                                   \\
    \bigg(\frac{\Loc_{\Two}\left(\Q_p,D_{\rhodn}\right) }{{\Loc_{\One}\left(\Q_p,D_{\rhodn}\right) \bigcap \Loc_{\Two}\left(\Q_p,D_{\rhodn}\right)}}\bigg)^\vee \rightarrow  \Sel_\Two(\Q,D_\rhodn)^\vee \rightarrow \ZZZ(\Q,D_{\rhodn}) \rightarrow 0. \label{ses-selmer-five}
  \end{align}
\end{lemma}

\begin{proof}
  The global to local maps $\Phi_\One$ and $\Phi_\Two$, defining the Selmer groups $\Sel_\One(\Q,D_\rhodn)$ and $\Sel_\Two(\Q,D_\rhodn)$, are surjective. The lemma follows by applying the Snake Lemma to the commutative diagrams given above.
\end{proof}

\subsection{Consequences of duality theorems} \label{dualitysection}

Recall that the ring $\RRR$ is a Gorenstein local ring whose residue field is finite with characteristic $p$. The dualizing module, often denoted $\omega_\RRR$, is isomorphic to $\RRR$.

Let $\G$ denote a profinite group satisfying the following conditions:
\begin{enumerate}[label=(F)]
  \item\label{F-cond} {$H^i\left(\G,M\right)$} is finite for all $i \geq 0$ and for every finite  $\RRR$-module $M$ (that is, the cardinality of $M$ is finite) equipped with a continuous $\RRR$-linear action of the profinite group $\G$.
\end{enumerate}
\begin{enumerate}[label=(CD)]
  \item\label{CD-cond} The $p$-cohomological dimension of $\G$ is less than or equal to $2$.
\end{enumerate}

\begin{remark}
  Note that conditions \ref{F-cond} and \ref{CD-cond} are both valid when the profinite group $\G$ equals the local decomposition $\Gal{\overline{\Q}_l}{\Q_l}$ for any prime $l \in \Sigma$, or the global Galois group $G_\Sigma$.
\end{remark}
Let $\TTT$ denote a finitely generated $\RRR$-module with a continuous $\RRR$-linear $\G$-action. Let $\DDD$ denote the discrete $\RRR$-module $\TTT \otimes_{\RRR} \RRR^\vee$. Observe that $\DDD$ also has a natural continuous $\RRR$-linear action of $\G$. \\

Let $\DDD^+(\RRR-\mathrm{mod})$ denote the subcategory of the derived category of finitely generated $\RRR$-modules, whose objects are chain complexes that are bounded from below. Proposition 4.2.5 in Nekov{\'a}{\v{r}}'s work on Selmer complexes \cite{nekovar2006selmer} shows that we have
\begin{align*}
  D\left(\mathbf{R\Gamma}_\cont(\G,\DDD)\right) \in \mathbf{D}^{+}\left(\RRR-\mathrm{mod}\right), \qquad \mathbf{R\Gamma}_\cont(\G,\TTT) \in \mathbf{D}^{+}\left(\RRR-\mathrm{mod}\right),
\end{align*}
such that we have the following isomorphism of $\RRR$-modules, for all $j \geq 0$:
\begin{align}
  H^{j}\bigg(D\left(\mathbf{R\Gamma}_\cont(\G,\DDD)\right)\bigg) \cong H^j\left(\G,\DDD\right)^\vee, \quad H^{j}\bigg(\mathbf{R\Gamma}_\cont(\G,\TTT)\bigg) \cong H^j_{\ct}\left(\G,\TTT\right).
\end{align}
Here, $H^j\left(\G,\DDD\right)$ (and $H^j_{\ct}\left(\G,\TTT\right)$ respectively) denote the discrete (and compact respectively) Galois cohomology groups for the continuous action of $\G$ on $\DDD$ (and $\TTT$ respectively). \\

To state Nekov{\'a}{\v{r}}'s results, one needs to use the notion of \textit{hyperext} groups (denoted $\mathrm{\mathbb{E}xt}\left(-,-\right)$) in the derived category $\mathbf{D}^{+}\left(\RRR-\mathrm{mod}\right)$. See Section 6 in Chapter I of Hartshorne's book on Residues and Duality \cite{MR0222093} for the definition of hyperext groups. Nekov{\'a}{\v{r}} has deduced the following \textit{hyper-cohomology} spectral sequence:
\begin{align} \label{hyperext}
  \mathrm{\mathbb{E}xt}^{i}\bigg(H^j\left(\G,\DDD\right)^\vee,\RRR\bigg) \implies H^{i+j}_{\ct}\left(\G,\TTT\right).
\end{align}
See equation (4.3.1.2) in Section 4.3 of his work on Selmer complexes \cite{nekovar2006selmer}. Using Corollary 6.1 in Chapter I of Hartshorne's book \cite{MR0222093} (see also Corollary 10.7.5 in Weibel's book \cite{weibel1995introduction}), we have the following isomorphism of $\RRR$-modules:
\begin{align}
  \mathrm{\mathbb{E}xt}^{i}\bigg(H^j\left(\G,\DDD\right)^\vee,\RRR\bigg) \cong \Ext^i_{\RRR}\left(H^j\left(\G,\DDD\right)^\vee,\RRR\right).
\end{align}

We shall suppose the following conditions hold:
\begin{align}
  H^2\left(\G,\DDD\right)=0, \qquad H^0_\ct\left(\G,\TTT\right)=0.
\end{align}

The spectral sequence (\ref{hyperext}) lets us obtain the exact sequences of $\RRR$-modules:
\begin{align}\label{1-ext-spectral}
  0 \rightarrow \Ext^{1}_{\RRR}\left(H^0\left(\G,\DDD\right)^\vee,\RRR\right) & \rightarrow \\ \notag & \rightarrow H^1_\ct\left(\G,\TTT\right) \rightarrow \Ext^{0}_{\RRR}\left(H^1\left(\G,\DDD\right)^\vee,\RRR\right) \rightarrow    \Ext^{2}_{\RRR}\left(H^0\left(\G,\DDD\right)^\vee,\RRR\right),
\end{align}
and
\begin{align} \label{2-ext-spectral}
  \Ext^{0}_{\RRR}\left(H^1\left(\G,\DDD\right)^\vee,\RRR\right)  \rightarrow & \Ext^{2}_{\RRR}\left(H^0\left(\G,\DDD\right)^\vee,\RRR\right)  \rightarrow \\ \notag & \rightarrow H^2_\ct\left(\G,\TTT\right) \rightarrow  \Ext^{1}_{\RRR}\left(H^1\left(\G,\DDD\right)^\vee,\RRR\right) \rightarrow   \Ext^{3}_{\RRR}\left(H^0\left(\G,\DDD\right)^\vee,\RRR\right).
\end{align}

\subsubsection{Consequences for local Galois cohomology groups, $l =p$.} \label{subsub-lp}
By equation (\ref{primes-consequence}), we have
\begin{align*}
  H^2\left(\Q_p,D_{\rhodn} \right) =0, \qquad H^0_\ct\left(\Q_p,T_{\rhodn} \right)^\vee = 0.
\end{align*}

Equation (\ref{local-p-h0-consequence}) lets us deduce that
\begin{align*}
  H^0\left(\Q_p,D_{\rhodn} \right)^\vee = 0.
\end{align*}
\ref{locp0} along with Proposition 5.10 (see also Remark 5.10.1) in Greenberg's work \cite{MR2290593} lets us conclude that the $\RRR$-module $H^1\left(\Q_p,D_{\rhodn}\right)^\vee$ is free. So, we have
\begin{align}
  \Ext^i_\RRR\bigg(H^1\left(\Q_p,D_{\rhodn}\right)^\vee,\RRR\bigg)=0, \quad \forall i \geq 1.
\end{align}

We have the following lemma:
\begin{lemma}\label{ext1-locap}
  Suppose all the hypotheses and assumptions in Section \ref{varioushypotheses} hold. Then, the $\RRR$-module $\left(\frac{H^1\left(\Q_p,D_{\rhodn}\right)}{\Loc_{\One}\left(\Q_p,D_{\rhodn}\right)+\Loc_{\Two}\left(\Q_p,D_{\rhodn}\right)}\right)^\vee$ is free. Consequently, \begin{align*}
  \Ext_\RRR^{i}\left(\left(\frac{H^1\left(\Q_p,D_{\rhodn}\right)}{\Loc_{\One}\left(\Q_p,D_{\rhodn}\right)+\Loc_{\Two}\left(\Q_p,D_{\rhodn}\right)}\right)^\vee,\RRR\right) =0, \quad \forall i \geq 1.
  \end{align*}
\end{lemma}

\begin{proof}
  The lemma follows directly using the following short exact sequence:
  \begin{align*}
    0  \rightarrow \left(\frac{H^1\left(\Q_p,D_{\rhodn}\right)}{\Loc_{\One}\left(\Q_p,D_{\rhodn}\right)+\Loc_{\Two}\left(\Q_p,D_{\rhodn}\right)}\right)^\vee & \rightarrow \underbrace{H^1\left(\Q_p,D_{\rhodn}\right)^\vee}_{\text{free over }\RRR} \rightarrow \\ & \rightarrow \underbrace{\left(\Loc_{\One}\left(\Q_p,D_{\rhodn}\right)+\Loc_{\Two}\left(\Q_p,D_{\rhodn}\right)\right)^\vee}_{\text{free over }\RRR}  \rightarrow  0
  \end{align*}
\end{proof}

The observations in Section \ref{consequences-hyp-assumptions} along with equation (\ref{1-ext-spectral}) lets us obtain the following isomorphisms of $\RRR$-modules:
\begin{align} \label{iso-local-tate-duality}
  \Hom_\RRR\bigg(H^1\left(\Q_p,D_{\rhodn}\right)^\vee,\RRR\bigg) & \cong H^1_\ct\left(\Q_p,T_{\rhodn}\right),                                    \\ \notag
                                                                 & \cong H^1\left(\Q_p,D_{\rhodnstar}\right)^\vee, \quad \text{(local duality)}.
\end{align}

The natural injection of $\RRR$-modules
\begin{align*}
  \Loc_\One(\Q_p,D_\rhodn)+ \Loc_\Two(\Q_p,D_\rhodn) \hookrightarrow H^1\left(\Q_p,\D_\rhodn\right)
\end{align*}
give us the following natural surjections of $\RRR$-modules (by considering Pontryagin duals)
\begin{align*}
  H^1\left(\Q_p,D_\rhodn\right)^\vee \twoheadrightarrow \left(\Loc_\One(\Q_p,D_\rhodn)+ \Loc_\Two(\Q_p,D_\rhodn)\right)^\vee,
\end{align*}
which in turn let us obtain the following natural injections of $\RRR$-modules (by considering reflexive duals and the isomorphism in equation (\ref{iso-local-tate-duality})):
\begin{align*}
  \left(\left(\Loc_\One(\Q_p,D_\rhodn)+ \Loc_\Two(\Q_p,D_\rhodn)\right)^\vee\right)^*  \hookrightarrow H^1_\ct\left(\Q_p,T_\rhodn\right).
\end{align*}

Under the perfect pairing given by local duality
\begin{align*}
  H^1\left(\Q_p, D_\rhodnstar\right) \times H^1_\ct\left(\Q_p,T_\rhodn\right) \rightarrow \frac{\Q_p}{\Z_p},
\end{align*}
we define
\begin{align*}
  \Loc_{\One,\Two}\left(\Q_p,D_\rhodnstar\right)  \subset H^1\left(\Q_p,D_\rhodnstar\right),
\end{align*}
to be the orthogonal complement of $\left(\left(\Loc_{\One}\left(\Q_p,D_\rhodn\right)+\Loc_{\Two}\left(\Q_p,D_\rhodn\right)\right)^\vee\right)^*$ under the pairing given above.

We define $\ZZZ^\Sstar(\Q,D_\rhodnstar)$ to be the Pontyagin dual of
\begin{align} \label{ZDstar}
  \ker\left(H^1\left(G_\Sigma,D_\rhodnstar\right) \rightarrow \frac{H^1\left(\Q_p,D_\rhodnstar\right)}{\Loc_{\One,\Two}\left(\Q_p,D_\rhodnstar\right)}\oplus \bigoplus_{l \in \Sigma \setminus \{p\}} H^1\left(\Q_l,D_\rhodnstar\right) \right).
\end{align}
Note that we have the following natural surjection of $\RRR$-modules:
\begin{align*}
  \ZZZ^\Sstar(\Q,D_\rhodnstar) \twoheadrightarrow \Sha^1\left(\Q,D_\rhodnstar\right)^\vee.
\end{align*}

\begin{remark}
This definition of $\ZZZ^\Sstar(\Q,D_\rhodnstar)$ in equation (\ref{ZDstar}) does not match the description given in equation (\ref{ZD}) for the Galois representation $\rhodnstar$ and, in fact, it need not in general.  See Section \ref{S:pseudonulldesc} for a precise description of $\ZZZ^\Sstar(\Q,D_\rhodnstar)$ in the setting of Theorem \ref{maintheorem} and Remark \ref{synopsis-match} for when these descriptions do match.
\end{remark}

\begin{lemma} \label{lem:insidecompactiso}
  Suppose all the hypotheses and assumptions in Section \ref{varioushypotheses} hold. We have the following natural isomorphism of $\RRR$-modules:
  \begin{align*}
    \frac{\left(H^1\left(\Q_p,D_{\rhodn}\right)^\vee\right)^*}{\left(\left(\Loc_\One(\Q_p,D_\rhodn) + \Loc_\Two(\Q_p,D_\rhodn)\right)^\vee\right)^*}  \cong \left( \left(\frac{H^1\left(\Q_p,D_{\rhodn}\right)}{\Loc_\One(\Q_p,D_\rhodn) + \Loc_\Two(\Q_p,D_\rhodn)} \right)^\vee\right)^*
  \end{align*}
\end{lemma}

\begin{proof}
  The lemma follows directly using the observation \ref{obs-a} in Section \ref{consequences-hyp-assumptions} and the following short exact sequence of free $\RRR$-modules:
  {\small \begin{align*}
    0  \rightarrow \left(\frac{H^1\left(\Q_p,D_{\rhodn}\right)}{\Loc_{\One}\left(\Q_p,D_{\rhodn}\right)+\Loc_{\Two}\left(\Q_p,D_{\rhodn}\right)}\right)^\vee & \rightarrow H^1\left(\Q_p,D_{\rhodn}\right)^\vee \rightarrow \left(\Loc_{\One}\left(\Q_p,D_{\rhodn},\Q_p\right)+\Loc_{\Two}\left(\Q_p,D_{\rhodn}\right)\right)^\vee  \rightarrow  0
    \end{align*}}
\end{proof}

\subsubsection{Consequences for local Galois cohomology groups, $l \neq p$.}
\ref{tor-loc-sigma} lets us conclude that
\begin{align} \label{h1-lneqp}
  \Hom_\RRR \bigg(H^1\left(\Q_l,D_{\rhodn}\right)^\vee, \RRR \bigg) = 0.
\end{align}

Let $\QQQ$ be a height two prime ideal in the ring $\RRR$. Since the ring $\RRR$ is Gorenstein, the localization $\RRR_\QQQ$ is also a Gorenstein ring with Krull dimension two. The injective dimension of the $\RRR_\QQQ$-module $\RRR_\QQQ$ equals two. By Lemma 2 in Chapter 19 of Matsumura's book \cite{matsumura1989commutative}, we have $$\Ext^{3}_{\RRR}\left(H^0\left(\G,\DDD\right)^\vee,\RRR\right) \otimes_\RRR \RRR_\QQQ \cong \Ext^{3}_{\RRR_\QQQ}\left(H^0\left(\G,\DDD\right)^\vee \otimes_\RRR \RRR_\QQQ,\RRR_\QQQ\right) = 0 $$

Using equation (\ref{2-ext-spectral}), we have the following short exact sequence of $\RRR_\QQQ$-modules:
\begin{align} \label{local-group-consequence}
  0 \rightarrow \Ext^{2}_{\RRR}\left(H^0\left(\Q_l,D_{\rhodn}\right)^\vee,\RRR\right)_\QQQ & \rightarrow H^2_\ct\left(\Q_l,T_{\rhodn}\right)_\QQQ  \rightarrow \Ext^{1}_{\RRR}\left(H^1\left(\Q_l,D_{\rhodn}\right)^\vee,\RRR\right)_\QQQ \rightarrow 0.
\end{align}

\subsubsection{Consequences for global Galois cohomology groups}
By equations (\ref{h0-global-compact-consequence}) and (\ref{weak-leo-consequence}), we have
\begin{align*}
  H^2\left(G_\Sigma,D_{\rhodn} \right) =0, \qquad H^0_\ct\left(G_\Sigma,T_{\rhodn} \right)^\vee = 0.
\end{align*}

Equation (\ref{global-p-h0-consequence}) lets us deduce that
\begin{align*}
  H^0\left(G_\Sigma,D_{\rhodn} \right)^\vee = 0.
\end{align*}

Using (\ref{1-ext-spectral}) and (\ref{2-ext-spectral}), we have the following isomorphisms of $\RRR$-modules:
\begin{align} \label{global-h1h2}
  H^1_\ct\left(G_\Sigma,T_{\rhodn} \right) \cong \Hom_{\RRR}\bigg(H^1\left(G_\Sigma,D_{\rhodn}\right)^\vee,\RRR\bigg), \qquad H^2_\ct\left(G_\Sigma,T_{\rhodn}\right) \cong \Ext^{1}_{\RRR}\bigg(H^1\left(G_\Sigma,D_{\rhodn}\right)^\vee,\RRR\bigg).
\end{align}

\subsubsection{Consequences for the module $\XXX(\Q,D_{\rhodn})$}
\begin{proposition} \label{star-ses}
  Suppose all the hypotheses and assumptions in Section \ref{varioushypotheses} hold. Then, the $\RRR$-module $\Ext^1_{\RRR}\left(\XXX(\Q,D_{\rhodn}),\RRR\right)$ is pseudo-null if and only if $\ZZZ^\Sstar(\Q,D_{\rhodnstar})$ is pseudo-null.
  Furthermore, if the $\RRR$-module $\ZZZ^\Sstar(\Q,D_{\rhodnstar})$ is pseudo-null, then we have the following equality in $Z^2\left(\RRR\right)$:
  \begin{align*}
    c_2\bigg(\Ext^1_{\RRR}\left(\XXX(\Q,D_{\rhodn}),\RRR\right) \bigg) = c_2\bigg(\ZZZ^\Sstar(\Q,D_{\rhodnstar}) \bigg)  + \sum_{l \in \Sigma \setminus \{p\}} c_2\bigg(\left(H^0\left(\Q_l,D_{\rhodn}\right)^\vee\right)_{\pn}\bigg).
  \end{align*}
\end{proposition}

\begin{proof}
  To prove that an $\RRR$-module $\MMM$ is pseudo-null, it suffices to show that the $\RRR_\QQQ$-module $\MMM_\QQQ$ is pseudo-null for every height two prime ideal $\QQQ$ in $\RRR$.

  By Lemma \ref{ses-Sel-X}, the map $\Phi_\XXX$ is surjective. So, we have the following short exact sequence of $\RRR$-modules:
  \begin{align*}
    0 \rightarrow   \bigg(\frac{H^1\left(\Q_p,D_{\rhodn}\right)}{\Loc_{\One}\left(\Q_p,D_{\rhodn}\right)+\Loc_{\Two}\left(\Q_p,D_{\rhodn}\right)} \bigg)^\vee \oplus \bigoplus_{l \in \Sigma \setminus \{p\}} H^1\left(\Q_l,D_{\rhodn}\right)^\vee & \rightarrow H^1\left(G_\Sigma, D_{\rhodn} \right)^\vee\rightarrow \\ & \rightarrow \XXX(\Q,D_\rhodn) \rightarrow 0.
  \end{align*}
  Apply the functor $\Hom_{\RRR}\left(-,\RRR\right)$. Use Lemmas \ref{ext1-locap} and \ref{lem:insidecompactiso} along with the isomorphisms given in equations (\ref{h1-lneqp}) and (\ref{global-h1h2}). We obtain the following long exact sequence of $\RRR$-modules:
  \begin{align*}
    \rightarrow H^1_\ct\left(G_\Sigma,T_{\rhodn} \right) & \rightarrow \frac{H^1_\ct\left(\Q_p,T_{\rhodn}\right)}{\left(\left(\Loc_\One(\Q_p,D_\rhodn)+ \Loc_\Two(\Q_p,D_\rhodn)\right)^\vee\right)^*}  \rightarrow \\ \rightarrow &\Ext^1_{\RRR}\left(\XXX(\Q,D_{\rhodn}),\RRR\right) \rightarrow H^2_\ct\left(G_\Sigma,T_{\rhodn} \right) \rightarrow  \bigoplus_{l \in \Sigma \setminus \{p\}}\Ext^{1}_{\RRR}\left(H^1\left(\Q_l,D_{\rhodn}\right)^\vee,\RRR\right) \\ & \rightarrow \Ext^2_{\RRR}\left(\XXX(\Q,D_{\rhodn}),\RRR\right) \rightarrow .
  \end{align*}

  The observations in Section \ref{subsub-lp} and Poitou-Tate duality along with the arguments in Section 3.1 of Greenberg's work on the surjectivity of the global-to-local map defining Selmer groups \cite{greenberg2010surjectivity} let us deduce that the cokernel of the map
  \begin{align*}
    H^1_\ct\left(G_\Sigma,T_{\rhodn} \right) & \rightarrow \frac{H^1_\ct\left(\Q_p,T_{\rhodn}\right)}{\left(\left(\Loc_\One(\Q_p,D_\rhodn)+ \Loc_\Two(\Q_p,D_\rhodn)\right)^\vee\right)^*}
  \end{align*}
  is isomorphic, as an $\RRR$-module to $\ker\left(\ZZZ^\Sstar(\Q,D_{\rhodnstar}) \twoheadrightarrow \Sha^1\left(\Q,D_\rhodnstar\right)^\vee\right)$. Thus, one obtains the following short exact sequence of $\RRR$-modules:
  {\small \begin{align} \label{surjectivity-poitou-tate}
    0 \rightarrow \ker\bigg(\ZZZ^\Sstar(\Q,D_{\rhodnstar}) \twoheadrightarrow \Sha^1\left(\Q,D_\rhodnstar\right)^\vee\bigg) & \rightarrow  \Ext^1_{\RRR}\left(\XXX(\Q,D_{\rhodn}),\RRR\right) \rightarrow \\ & \notag \rightarrow \ker\bigg(H^2_\ct\left(G_\Sigma,T_{\rhodn} \right) \rightarrow \bigoplus_{l \in \Sigma \setminus \{p\}}\Ext^{1}_{\RRR}\left(H^1\left(\Q_l,D_{\rhodn}\right)^\vee,\RRR\right)\bigg) \rightarrow 0.
    \end{align}}

  One can consider the following $\RRR$-module:
  \begin{align*}
    \Sha^2_\ct\left(\Q,T_{\rhodn}\right) :=\ker\bigg(H^2_\ct\left(G_\Sigma,T_{\rhodn}\right) \xrightarrow {\phi_{\Sha^2_\ct}} \underbrace{H^2_\ct\left(\Q_p,T_{\rhodn}\right)}_{=0} \oplus \bigoplus \limits_{l \in \Sigma \setminus \{p\}} H^2_\ct\left(\Q_l,T_{\rhodn}\right) \bigg).
  \end{align*}

  Poitou-Tate duality provides us an isomorphism between the $\RRR$-modules $\Sha^1(\Q,D_{\rhodnstar})^\vee$ and $\Sha^2_\ct\left(\Q,T_{\rhodn}\right)$. Furthermore, Poitou-Tate duality also tells us that the cokernel of the map $\phi_{\Sha^2_\ct}$  is isomorphic to $H^0\left(G_\Sigma,D_\rhodnstar\right)^\vee$ and hence zero.

 Let us fix a height two prime ideal, say $\QQQ$, in the ring $\RRR$.  Using equation (\ref{local-group-consequence}), we have the following commutative diagram of $\RRR_\QQQ$-modules:

  {\tiny
    \begin{align*}
      \xymatrix{
      & & H^2_\ct\left(G_\Sigma,T_{\rhodn}\right)_\QQQ \ar[d]\ar[r]^{\cong} & H^2_\ct\left(G_\Sigma,T_{\rhodn}\right)_\QQQ \ar[d]\\
      0 \ar[r] & \bigoplus \limits_{l \in \Sigma \setminus \{p\}} \Ext^{2}_{\RRR}\left(H^0\left(\Q_l,D_{\rhodn}\right)^\vee,\RRR\right)_\QQQ \ar[r] & \bigoplus \limits_{l \in \Sigma \setminus \{p\}} H^2_\ct\left(\Q_l,T_{\rhodn}\right)_\QQQ  \ar[r] & \bigoplus \limits_{l \in \Sigma \setminus \{p\}} \Ext^{1}_{\RRR}\left(H^1\left(\Q_l,D_{\rhodn}\right)^\vee,\RRR\right)_\QQQ \ar[r] & 0.
      }
    \end{align*}
  }

  Applying the Snake Lemma, we obtain the following short exact sequence of $\RRR_\QQQ$-modules:
  \begin{align} \label{PN-equivalence}
    0 \rightarrow \left(\Sha^1(\Q,D_{\rhodnstar})^\vee\right)_\QQQ & \rightarrow \ker\left(H^2_\ct\left(G_\Sigma,T_{\rhodn}\right)  \rightarrow \Ext^{1}_{\RRR}\left(H^1\left(\Q_l,D_{\rhodn}\right)^\vee,\RRR\right) \right)_\QQQ \rightarrow \\ \notag & \rightarrow \bigoplus \limits_{l \in \Sigma \setminus \{p\}} \Ext^{2}_{\RRR}\left(H^0\left(\Q_l,D_{\rhodn}\right)^\vee,\RRR\right)_\QQQ \rightarrow 0.
  \end{align}

  By Lemma \ref{ext2-is-PN}, for each $l \in \Sigma \setminus \{p\}$, the $\RRR$-module $\Ext^{2}_{\RRR}\left(H^0\left(\Q_l,D_{\rhodn}\right)^\vee,\RRR\right)$ is pseudo-null. As a result, equation (\ref{PN-equivalence}) lets us deduce that the $\RRR_\QQQ$-module $$\ker\left(H^2_\ct\left(G_\Sigma,T_{\rhodn}\right)  \rightarrow \Ext^{1}_{\RRR}\left(H^1\left(\Q_l,D_{\rhodn}\right)^\vee,\RRR\right) \right)_\QQQ $$ is pseudo-null if and only if the $\RRR_\QQQ$-module $\left(\Sha^1(\Q,D_{\rhodnstar})^\vee\right)_\QQQ$ is pseudo-null. This observation along with equation (\ref{surjectivity-poitou-tate}) lets us deduce the following implications:
  \begin{align*}
      & \text{The $\RRR_\QQQ$-module } \left(\Ext^1_{\RRR}\left(\XXX(\Q,D_{\rhodn}),\RRR\right)^\vee\right)_\QQQ \text{ is pseudo-null} \\  \iff & \text{The $\RRR_\QQQ$-modules }  \left(\ker\bigg(\ZZZ^\Sstar(\Q,D_{\rhodnstar}) \twoheadrightarrow \Sha^1\left(\Q,D_\rhodnstar,\Q\right)^\vee\bigg)\right)_Q \text{ and } \left(\Sha^1(\Q,D_{\rhodnstar})^\vee\right)_\QQQ\\& \text{ are pseudo-null}, \\
    \iff & \text{The $\RRR_\QQQ$-module } \ZZZ^\Sstar(\Q,D_{\rhodnstar})_\QQQ \text{ is pseudo-null}.
  \end{align*}

  This lets us conclude that the $\RRR$-module $\Ext^1_{\RRR}\left(\XXX(\Q,D_{\rhodn}),\RRR\right)$ is pseudo-null if and only if the $\RRR$-module $\ZZZ^\Sstar(\Q,D_\rhodnstar)$ is pseudo-null. Now, suppose the $\RRR$-module $\ZZZ^\Sstar(\Q,D_\rhodnstar)$ is pseudo-null.  Equations (\ref{surjectivity-poitou-tate}) and (\ref{PN-equivalence}) let us deduce the following equalities in $Z^2\left(\RRR\right)$:
\begin{align*}
&c_2\bigg(\Ext^1_{\RRR}\left(\XXX(\Q,D_{\rhodn}),\RRR\right) \bigg)
\\ =& c_2\left(\ZZZ^\Sstar(\Q,D_\rhodnstar)\right) -  c_2\left(\Sha^1\left(\Q,D_\rhodnstar\right)^\vee\right) + c_2\left(\Sha^1\left(\Q,D_\rhodnstar\right)^\vee\right) +
\sum_{l \in \Sigma \setminus \{p\}} c_2\bigg(\Ext^{2}_{\RRR}\left(H^0\left(\Q_l,D_{\rhodn}\right)^\vee,\RRR\right)\bigg),
\\ =& c_2\left(\ZZZ^\Sstar(\Q,D_\rhodnstar)\right) +
\sum_{l \in \Sigma \setminus \{p\}} c_2\bigg(\Ext^{2}_{\RRR}\left(H^0\left(\Q_l,D_{\rhodn}\right)^\vee,\RRR\right)\bigg),
\\ =& c_2\left(\ZZZ^\Sstar(\Q,D_\rhodnstar)\right) +
\sum_{l \in \Sigma \setminus \{p\}} c_2\bigg(\left(H^0\left(\Q_l,D_{\rhodn}\right)^\vee\right)_\pn\bigg), \qquad \text{using \ref{Sigma0-pd} and by Lemma \ref{ext2-is-PN}}.
\end{align*}
  This completes the proof of the proposition.
\end{proof}

\section{The main theorem in the general setup} \label{generalresults}

\subsection{Alternative characterizations of \ref{gcd}}

Before proving the main theorem in the general setup, we would like to provide alternative characterizations of \ref{gcd}.

\begin{proposition} \label{gcd-proposition}
  Suppose all the hypotheses and assumptions in Section \ref{varioushypotheses} hold.

  Then, the following statements are equivalent:
  \begin{enumerate}
    \item \label{gcd-one} The height of the ideal $(\theta_{\One},\theta_{\Two})$ in $\RRR$ is greater than or equal to two.
    \item \label{gcd-two} The $\RRR$-modules $\ZZZ(\Q,D_{\rhodn})$ and $\XXX(\Q,D_{\rhodn})_{\tor}$ are pseudo-null.
    \item \label{gcd-three} The $\RRR$-modules $\ZZZ(\Q,D_{\rhodn})$ and $\ZZZ^\Sstar(\Q,D_{\rhodnstar})$ are pseudo-null.
  \end{enumerate}
\end{proposition}

\begin{proof}

  We will first show Condition (\ref{gcd-one}) $\implies$ Condition (\ref{gcd-two}). \\
  Suppose Condition (\ref{gcd-one}) holds. Let $\p$ be a height one prime ideal in $\RRR$. Without loss of generality, assume $\theta_\One \notin \p$. By equation (\ref{mc-div-equality}) in \ref{mc}, we can conclude that the prime ideal $\p$ does not belong to the support of the $\RRR$-module $\Sel_\One(\Q,D_\rhodn)^\vee$. As a result, $\Sel_\One(\Q,D_\rhodn)^\vee \otimes_\RRR \RRR_\p$ would equal zero. By (\ref{ses-selmer-four}), we have $$\ZZZ(\Q,D_{\rhodn}) \otimes_\RRR \RRR_\p=0.$$
  By equation (\ref{ses-selmer-one}) and \ref{loc-free}, we have the following isomorphism of free $\RRR_\p$-modules of rank one:
  \begin{align*}
    \underbrace{\bigg(\frac{\Loc_{\Two}\left(\Q_p,D_{\rhodn}\right)}{{\Loc_{\One}\left(\Q_p,D_{\rhodn}\right)\bigcap \Loc_{\Two}\left(\Q_p,D_{\rhodn}\right)}}\bigg)^\vee \otimes_\RRR \RRR_\p}_{\text{free $\RRR_\p$-module of rank one}} \cong \XXX(\Q,D_{\rhodn})\otimes_\RRR \RRR_\p.
  \end{align*}
  As a result, $\XXX(\Q,D_{\rhodn})_{\tor} \otimes_\RRR \RRR_p$ equals zero. \\

  Secondly, we will  show Condition (\ref{gcd-two}) $\implies$ Condition (\ref{gcd-one}). \\
  Suppose Condition (\ref{gcd-two}) holds. Let $\p$ be a height one prime ideal in $\RRR$. It suffices to show that $\theta_\One \notin \p$ or $\theta_\Two \notin \p$. We have $$\ZZZ(\Q,D_{\rhodn}) \otimes_\RRR \RRR_\p=0, \qquad \XXX(\Q,D_{\rhodn})_{\tor} \otimes_\RRR \RRR_p =0.$$
  By equation (\ref{ses-selmer-three}) and \ref{loc-free}, we have the following surjection of free $\RRR_\p$-modules:
  {\small \begin{align*}
    \underbrace{\bigg(\frac{\Loc_{\One}\left(\Q_p,D_{\rhodn}\right) }{{\Loc_{\One}\left(\Q_p,D_{\rhodn}\right) \bigcap \Loc_{\Two}\left(\Q_p,D_{\rhodn}\right)}}\bigg)^\vee \otimes_\RRR \RRR_\p}_{\text{free $\RRR_\p$-module of rank one}} & \oplus  \underbrace{\bigg(\frac{\Loc_{\Two}\left(\Q_p,D_{\rhodn}\right) }{{\Loc_{\One}\left(\Q_p,D_{\rhodn}\right) \bigcap \Loc_{\Two}\left(\Q_p,D_{\rhodn}\right)}}\bigg)^\vee \otimes_\RRR \RRR_\p}_{\text{free $\RRR_\p$-module of rank one }} \\ \twoheadrightarrow & \underbrace{\XXX(\Q,D_{\rhodn}) \otimes_\RRR \RRR_p}_{\text{free $\RRR_\p$-module of rank one}}.
    \end{align*}}

  As a result, at least one of the two maps
  \begin{align*}
    \bigg(\frac{\Loc_{\One}\left(\Q_p,D_{\rhodn}\right) }{{\Loc_{\One}\left(\Q_p,D_{\rhodn}\right) \bigcap \Loc_{\Two}\left(\Q_p,D_{\rhodn}\right)}}\bigg)^\vee \otimes_\RRR \RRR_\p \rightarrow \XXX(\Q,D_{\rhodn}) \otimes_\RRR \RRR_\p, \\ \bigg(\frac{\Loc_{\Two}\left(\Q_p,D_{\rhodn}\right) }{{\Loc_{\One}\left(\Q_p,D_{\rhodn}\right) \bigcap \Loc_{\Two}\left(\Q_p,D_{\rhodn}\right)}}\bigg)^\vee \otimes_\RRR \RRR_\p \rightarrow \XXX(\Q,D_{\rhodn}) \otimes_\RRR \RRR_\p.
  \end{align*}
  must be an isomorphism of free $\RRR_\p$-modules of rank one. Without loss of generality, assume that the first map is an isomorphism.
  By (\ref{ses-selmer-one}), we can conclude that $\Sel_\One(\Q,D_\rhodn)^\vee \otimes_\RRR \RRR_\p$ equals zero. That is, $\p$ does not belong to the support of the $\RRR$-module $\Sel_\One(\Q,D_\rhodn)^\vee$. By equation (\ref{mc-div-equality}) in \ref{mc}, we have $\theta_\One \notin \p$.  \\

  We will now show that Condition (\ref{gcd-two}) is equivalent to Condition (\ref{gcd-three}). To do so, we will need to show that the following statements are equivalent:
  \begin{enumerate}[(i)]
    \item\label{starpn} The $\RRR$-modules $\ZZZ^\Sstar(\Q,D_{\rhodnstar})$ is pseudo-null.
    \item\label{torpn} The $\RRR$-module $\XXX(\Q,D_{\rhodn})_{\tor}$ is pseudo-null.
  \end{enumerate}

  Let $\p$ be a height one prime ideal in $\RRR$. Note that $\RRR_\p$ is a discrete valuation ring. Let $\pi_\p$ denote a uniformizer in $\RRR_\p$. Every finitely generated $\RRR_\p$-module $M$ is isomorphic to $\RRR_\p^{r} \bigoplus \oplus_i\frac{\RRR_\p}{(\pi_\p^{r_i})}$, for some non-negative integers $r$, and $r_i$. In particular, $M$ is a torsion-free $\RRR_\p$-module if and only if $M$ is a free $\RRR_\p$-module. Also, the $\RRR_\p$-module $\Ext^{1}_{\RRR_\p}\left(\frac{\RRR_\p}{(\pi_\p^{a})},\RRR_\p\right)$ is non-canonically isomorphic to $ \frac{\RRR_\p}{(\pi_\p^{a})}$.

  The fact that Condition \ref{starpn} and Condition \ref{torpn} are equivalent follows from the above observations:
  \begin{align*}
    & \ZZZ^\Sstar(\Q,D_{\rhodnstar})\otimes_\RRR \RRR_\p =0, \\
    \iff & \Ext^1_{\RRR_\p}\left(\XXX(\Q,D_{\rhodn}) \otimes_\RRR \RRR_\p,\RRR_\p\right)=0, \qquad && (\text{by Proposition \ref{star-ses}}) \\
    \iff & \XXX(\Q,D_{\rhodn})_{\tor} \otimes_\RRR \RRR_\p =0.
  \end{align*}

\end{proof}

We have the following important corollary to Proposition \ref{gcd-proposition}.
\begin{corollary} \label{torsion-free-corollary}
  Suppose all the hypotheses and assumptions in Section \ref{varioushypotheses} hold. Then, the $\RRR$-module $\XXX(\Q,D_{\rhodn})_{\tor}$ equals zero.
\end{corollary}

\begin{proof}
  Proposition \ref{gcd-proposition} tells us that $\RRR$-module $\XXX(\Q,D_{\rhodn})_{\tor}$ is pseudo-null. We have already established that $\RRR$-module $\XXX(\Q,D_{\rhodn})$ has no non-zero pseudo-null submodules (see Section \ref{consequences-hyp-assumptions}). As a result, $\XXX(\Q,D_{\rhodn})_{\tor}$ equals zero.
\end{proof}

We will need to consider the map
\begin{align*}
  i_\XXX: \XXX(D_{\rhodn},\Q) \rightarrow \XXX(\Q,D_{\rhodn})^{**}.
\end{align*}
By Corollary \ref{torsion-free-corollary}, note that $\ker(i_\XXX)$ (which is equal to $\XXX(\Q,D_{\rhodn})_{\tor}$) equals zero.

\subsection{Proof of the main theorem}

\begin{theorem} \label{general-main-theorem}
  Suppose that the following conditions hold:
  \begin{enumerate}
    \item\label{hyp-one-theorem} All the hypotheses and assumptions in Section \ref{varioushypotheses} hold.
    \item\label{hyp-two-theorem} For every height two prime ideal $\QQQ$ of $\RRR$, the $\RRR_\QQQ$-module $\XXX(\Q,D_{\rhodn}) \otimes_\RRR \RRR_\QQQ$ has finite projective dimension.\end{enumerate}
      Then, we have the following equality in $Z^2\left(\RRR\right)$:
  \begin{align*}
    c_2\left(\frac{\Z_p[[\widetilde{\Gamma}]]}{\left(\theta_\One,\theta_\Two\right)}\right) = c_2\left(\ZZZ(\Q,D_{\rhodn})\right) + c_2\left(\ZZZ^\Sstar(\Q,D_{\rhodnstar})\right) +
    \sum_{l \in \Sigma \setminus \{p\}} c_2\bigg(\left(H^0\left(\Q_l,D_{\rhodn}\right)^\vee\right)_\pn\bigg).
  \end{align*}
\end{theorem}

\begin{proof}

  To prove the theorem, we proceed in three steps.

  \begin{enumerate}
 \item[Step One:]For every height two prime ideal $\QQQ$ in $\RRR$, the $\RRR_\QQQ$-module $\coker(i_\XXX)_\QQQ$ has finite projective dimension.

\item[Step Two:] For every height two prime ideal $\QQQ$ in $\RRR$, we have the following short exact sequence of $\RRR_\QQQ$-modules:
          \begin{align*}
            0 \rightarrow \ZZZ(\Q,D_{\rhodn}) \otimes_\RRR \RRR_\QQQ \rightarrow \frac{ \RRR_\QQQ}{(\theta_\One,\theta_\Two)} \rightarrow \coker(i_\XXX) \otimes_{\RRR}\RRR_\QQQ \rightarrow 0.
          \end{align*}
    \item[Step Three:] We have the following equality in $Z^2\left(\RRR\right)$:
          \begin{align*}
            c_2\left(\coker(i_\XXX) \right) = c_2\left(\ZZZ^\Sstar(\Q,D_{\rhodnstar})\right) +
            \sum_{l \in \Sigma \setminus \{p\}} c_2\bigg(\left(H^0\left(\Q_l,D_{\rhodn}\right)^\vee\right)_\pn\bigg).
          \end{align*}
  \end{enumerate}

Step One follows from Corollary \ref{torsion-free-corollary}, Lemma \ref{freeness-X-reflexivehull} and condition (\ref{hyp-two-theorem}). Step Two follows from Lemma \ref{two-coker}. Step Three follows from Lemma \ref{c2-general}. Theorem \ref{general-main-theorem} would following from these lemmas. While proving these lemmas, we assume all the conditions stated in the theorem.

\begin{lemma} \label{freeness-X-reflexivehull}
The $\RRR$-module $\XXX(\Q,D_{\rhodn})^{**}$ is free.
\end{lemma}

\begin{proof}
It will be enough to show that the $\RRR$-module $\XXX(\Q,D_{\rhodn})^{*}$ is free. By applying the functor $\Hom_\RRR(\mbox{--},\RRR)$ to the first short exact sequence in Lemma \ref{ses-Sel-X}, we obtain the following exact sequence of $R$-modules:
\begin{align} \label{eq:Xdagger}
0 \rightarrow \XXX(\Q,D_{\rhodn})^{*} \rightarrow \RRR \rightarrow  \Ext^1_{\RRR}\left(\Sel_\One(\Q,D_{\rhodn})^\vee,\RRR\right) \rightarrow \Ext^1_{\RRR}\left(\XXX(\Q,D_{\rhodn}),\RRR\right)
\end{align}

To obtain equation  (\ref{eq:Xdagger}), we have identified the free $\RRR$-module $\bigg(\frac{\Loc_{\Two}\left(\Q_p,D_{\rhodn}\right)}{{\Loc_{\One}\left(\Q_p,D_{\rhodn}\right)\bigcap \Loc_{\Two}\left(\Q_p,D_{\rhodn}\right)}}\bigg)^\vee$ of rank one with $\RRR$. This allows us to identify $\XXX(\Q,D_{\rhodn})^{*}$ with an ideal inside $\RRR$. This ideal must be reflexive over $\RRR$. The $\RRR$-module $\XXX(\Q,D_{\rhodn})$ is torsion-free. Since $\RRR$ is integrally closed, for every height one prime ideal $\p$ in $\RRR$, the localization $\RRR_\p$ must be a DVR and consequently the $\RRR_\p$-module $\XXX(\Q,D_{\rhodn})_\p$ must be free. Since localization commutes with $\Ext$, the $\RRR$-module $\Ext^1_{\RRR}\left(\XXX(\Q,D_{\rhodn}),\RRR\right)$ must be pseudo-null. Using this trick of localizing at every height one prime ideal of $\RRR$, we have
$\Div\bigg( \Ext^1_{\RRR}\left(\Sel_\One(\Q,D_{\rhodn})^\vee,\RRR\right)\bigg)=\Div(\theta_\One)$, an equality of divisors in $Z^1(\RRR)$. Combining these observations, for every height one prime ideal $\p$ in $\RRR$, we obtain the following short exact sequence of $\RRR_\p$-modules:
\begin{align}  \label{eq:XdaggerP}
0 \rightarrow \XXX(\Q,D_{\rhodn})^{*} \otimes_\RRR \RRR_\p \rightarrow \RRR_\p \rightarrow \frac{\RRR_\p}{(\theta_\One)} \rightarrow 0.
\end{align}
Since $\RRR$ is integrally closed, equation (\ref{eq:XdaggerP}) lets us deduce that under the inclusion $\XXX(\Q,D_{\rhodn})^{*} \hookrightarrow \RRR$ given in equation (\ref{eq:Xdagger}), we have a natural inclusion map $\XXX(\Q,D_{\rhodn})^{*} \hookrightarrow (\theta_\One)$, whose cokernel is pseudo-null. Since the ideal $\XXX(\Q,D_{\rhodn})^{*}$ is reflexive, this natural inclusion must be an equality. This shows that the $\RRR$-module $\XXX(\Q,D_{\rhodn})^{*}$ is free.
\end{proof}

  \begin{lemma} \label{two-coker}
    For every height two prime ideal $\QQQ$ in $\RRR$, we have the following short exact sequence of $\RRR_\QQQ$-modules:
    \begin{align*}
      0 \rightarrow \ZZZ(\Q,D_{\rhodn}) \otimes_\RRR \RRR_\QQQ \rightarrow \frac{ \RRR_\QQQ}{(\theta_\One,\theta_\Two)} \rightarrow \coker(i_\XXX) \otimes_{\RRR}\RRR_\QQQ \rightarrow 0.
    \end{align*}
  \end{lemma}
  \begin{proof}[Proof of Lemma \ref{two-coker}]
    Let $\QQQ$ be a height two prime ideal in $\RRR$. By Lemma \ref{freeness-X-reflexivehull}, the $\RRR_\QQQ$-module $\XXX(\Q,D_{\rhodn})^{**} \otimes_\RRR \RRR_\QQQ$ is free.     By Lemma \ref{ses-Sel-X}, we can conclude that the $\RRR_\QQQ$-module $\XXX(\Q,D_{\rhodn})$ has rank one. \\

    \ref{loc-free} lets us consider the following two maps of free $\RRR_\QQQ$-modules:
    \begin{align*}
      \underbrace{\bigg(\frac{\Loc_{\One}\left(\Q_p,D_{\rhodn}\right)  }{{\Loc_{\One}\left(\Q_p,D_{\rhodn}\right) \bigcap \Loc_{\Two}\left(\Q_p,D_{\rhodn}\right)}}\bigg)^\vee  \otimes_\RRR \RRR_\QQQ}_{\text{free $\RRR_\QQQ$-module of rank one}} \xrightarrow {A_\One} \underbrace{\XXX(\Q,D_{\rhodn})^{**} \otimes_\RRR \RRR_\QQQ}_{\text{free $\RRR_\QQQ$-module of rank one}}, \\ \underbrace{\bigg(\frac{\Loc_{\Two}\left(\Q_p,D_{\rhodn}\right)  }{{\Loc_{\One}\left(\Q_p,D_{\rhodn}\right) \bigcap \Loc_{\Two}\left(\Q_p,D_{\rhodn}\right)}}\bigg)^\vee  \otimes_\RRR \RRR_\QQQ}_{\text{free $\RRR_\QQQ$-module of rank one}} \xrightarrow {A_\Two} \underbrace{\XXX(\Q,D_{\rhodn})^{**} \otimes_\RRR \RRR_\QQQ}_{\text{free $\RRR_\QQQ$-module of rank one}}
    \end{align*}
    Here, $A_\One$ and $A_\Two$ are elements of $\RRR_\QQQ$. For every height one prime ideal $\p$ in $\RRR$, the natural map $$\XXX(\Q,D_{\rhodn}) \otimes_{\RRR} \RRR_\p \xrightarrow {\cong} \XXX(\Q,D_{\rhodn})^{**} \otimes_\RRR \RRR_\p$$ is an isomorphism. Equation (\ref{mc-div-equality}) and the short exact sequences in (\ref{ses-selmer-one}) and (\ref{ses-selmer-two}) provide us the following equalities of divisors in $Z^1\left(\RRR_\QQQ\right)$:
    \begin{align*}
      \Div(A_\One) = \Div(\theta_\One), \qquad \Div(A_\Two)=\Div(\theta_\Two).
    \end{align*}
    As a result, there exists two units $u_\One$ and $u_\Two$ in the ring $\RRR_\QQQ$, such that we have the following equality of elements in the ring $\RRR_\QQQ$:
    \begin{align*}
      A_\One = u_\One\theta_\One, \qquad A_\Two= u_\Two \theta_\Two .
    \end{align*}

    The cokernel of the map
    \begin{align*}
      \underbrace{\bigg(\frac{\Loc_{\One}\left(\Q_p,D_{\rhodn}\right) + \Loc_{\Two}\left(\Q_p,D_{\rhodn}\right) }{{\Loc_{\One}\left(\Q_p,D_{\rhodn}\right) \bigcap \Loc_{\Two}\left(\Q_p,D_{\rhodn}\right)}}\bigg)^\vee  \otimes_\RRR \RRR_\QQQ}_{\text{free $\RRR_\QQQ$-module of rank two}} \xrightarrow {\left[\begin{array}{cc} A_\One, & A_\Two\end{array}\right]} \underbrace{\XXX(\Q,D_{\rhodn})^{**} \otimes_\RRR \RRR_\QQQ}_{\text{free $\RRR_\QQQ$-module of rank one}}
    \end{align*}
    is, thus, isomorphic to $\frac{ \RRR_\QQQ}{(\theta_\One,\theta_\Two)}$.

    By (\ref{ses-selmer-three}), the cokernel of the map
    \begin{align*}
      \bigg(\frac{\Loc_{\One}\left(\Q_p,D_{\rhodn}\right) + \Loc_{\Two}\left(\Q_p,D_{\rhodn}\right) }{{\Loc_{\One}\left(\Q_p,D_{\rhodn}\right) \bigcap \Loc_{\Two}\left(\Q_p,D_{\rhodn}\right)}}\bigg)^\vee  \otimes_\RRR \RRR_\QQQ \rightarrow \XXX(\Q,D_{\rhodn}) \otimes_\RRR \RRR_\QQQ
    \end{align*}
    is isomorphic to $\ZZZ(\Q,D_{\rhodn})$. \\

    Consider the following commutative diagram:
    \begin{align*}
      \xymatrix{
      & \bigg(\frac{\Loc_{\One}\left(\Q_p,D_{\rhodn}\right) + \Loc_{\Two}\left(\Q_p,D_{\rhodn}\right) }{{\Loc_{\One}\left(\Q_p,D_{\rhodn}\right) \bigcap \Loc_{\Two}\left(\Q_p,D_{\rhodn}\right)}}\bigg)^\vee  \otimes_\RRR \RRR_\QQQ \ar[d] \ar[rd]  \\
      0 \ar[r] & \XXX(\Q,D_{\rhodn}) \otimes_\RRR \RRR_\QQQ \ar[r] & \XXX(\Q,D_{\rhodn})^{**} \otimes_\RRR \RRR_\QQQ  \ar[r] & \coker\left(i_\XXX\right) \otimes_\RRR \RRR_\QQQ  \ar[r] & 0.
      }
    \end{align*}

    Now, applying the Snake lemma to the commutative diagram given above, we have the following short exact sequence of $\RRR_\QQQ$-modules:
    \begin{align*}
      0 \rightarrow \ZZZ(\Q,D_{\rhodn}) \otimes_\RRR \RRR_\QQQ \rightarrow \frac{ \RRR_\QQQ}{(\theta_\One,\theta_\Two)} \rightarrow \coker(i_\XXX) \otimes_{\RRR}\RRR_\QQQ \rightarrow 0.
    \end{align*}
  \end{proof}

  \begin{lemma} \label{c2-general}
    We have the following equality in $Z^2\left(\Z_p[[\widetilde{\Gamma}]]\right)$:
    \begin{align*}
      c_2\left(\coker(i_\XXX) \right) = c_2\left(\ZZZ^\Sstar(\Q,D_{\rhodnstar})\right) +
      \sum_{l \in \Sigma \setminus \{p\}} c_2\bigg(\left(H^0\left(\Q_l,D_{\rhodn}\right)^\vee\right)_\pn\bigg).
    \end{align*}
  \end{lemma}

  \begin{proof}[Proof of Lemma \ref{c2-general}]
    Note that Proposition \ref{star-ses} is applicable. By Corollary \ref{torsion-free-corollary}, the $\RRR$-module $\XXX(\Q,D_{\rhodn})$ is torsion-free.  As a result of Step One and the hypotheses of the theorem, Proposition \ref{Ext-second-Chern} is also applicable. The lemma follows from the following equalities in $Z^2\left(\Z_p[[\widetilde{\Gamma}]]\right)$:
    \begin{align*}
      c_2\left(\coker(i_\XXX) \right) & = c_2\bigg(\Ext^1_{\RRR}\left(\XXX(\Q,D_{\rhodn}),\RRR\right)\bigg), \quad & (\text{by Proposition \ref{Ext-second-Chern}}) \\
      & = c_2\left(\ZZZ^\Sstar(\Q,D_{\rhodnstar})\right) +
      \sum_{l \in \Sigma \setminus \{p\}} c_2\bigg(\left(H^0\left(\Q_l,D_{\rhodn}\right)^\vee\right)_\pn\bigg), & (\text{by Proposition \ref{star-ses}}).
    \end{align*}
  \end{proof}

  Theorem \ref{general-main-theorem} follows.
\end{proof}

\section{The Iwasawa main conjecture and the Panchishkin condition}\label{S:Panch}

\subsection{Greenberg's Selmer groups}

We recall some of the notations from the introduction. We let $K(\p^\infty)_{\Z_p}$ (and $K(\q^\infty)_{\Z_p}$ respectively) denote the unique $\Z_p$-extension of $K$ that is unramified outside $\p$ (and $\q$ respectively). Let $\Gamma_\p$ and $\Gamma_\q$  denote the Galois groups $\Gal{K(\p^\infty)_{\Z_p}}{K}$ and $\Gal{K(\q^\infty)_{\Z_p}}{K}$ respectively. We let  $\kappa_\p: G_K \twoheadrightarrow \Gamma_\p \hookrightarrow \Gl_1(\Z_p[[\Gamma_\p]])$ and $\kappa_{\q}: G_K \twoheadrightarrow \Gamma_{\q} \hookrightarrow \Gl_1({\Z_p[[\Gamma_\q]]})$ denote the associated tautological characters. Consider the two-dimensional Galois representation
\begin{align*}
  \rho_{\p}: G_\Q \rightarrow \Gl_2\left(\Z_p[[\Gamma_\p]]\right)
\end{align*}
given by the action of $G_\Q$ on the following free $\Z_p[[\Gamma_\p]]$-module of rank two:
\begin{align*}
  T_\p:=\Ind^{G_\Q}_{G_K} \left(\Z_p[[\Gamma_\p]](\kappa_\p^{-1})\right).
\end{align*}

Since the prime $p$ splits in the imaginary quadratic field, we have the following decomposition of $\Z_p[[\Gamma_\p]]$-modules which is $\Gal{\overline{\Q}_p}{\Q_p}$-equivariant:
\begin{align} \label{decomposition-p-splits}
  T_{\p} \cong \Z_p[[\Gamma_\p]](\kappa_\p^{-1}) \oplus \Z_p[[\Gamma_\p]](\kappa_{\q}^{-1}).
\end{align}

Let $K_\p$ denote the completion of the imaginary quadratic field $K$ with respect to the prime ideal $\p$. Since the prime $p$ splits in the imaginary quadratic field, the embedding $K \hookrightarrow \overline{\Q}_p$ (fixed in our introduction) then gives us an isomorphism $\Q_p \stackrel{\cong}{\hookrightarrow} K_\p$. This embedding then gives us the following isomorphism of Galois groups:
\begin{align}
  \Gal{\overline{\Q}_p}{K_\p} \stackrel{\cong}{\hookrightarrow} \Gal{\overline{\Q}_p}{\Q_p}.
\end{align}
Throughout this paper, we will use this isomorphism to identify $\Gal{\overline{\Q}_p}{K_\p}$ with $\Gal{\overline{\Q}_p}{\Q_p}$.

{Recall from the introduction that $E$ is an elliptic curve over $\Q$ with good supersingular reduction at $p$ with $a_p(E)=0$. {We refer the reader to the introduction for the precise details on the construction of the $\Z_p[[\tilde{\Gamma}]]$-module $T_{\rhoDN42}$, which is free of rank four, that in turn} gives rise to a $G_\Q$-representation $\rhoDN{4}{2}:G_\Q\rightarrow\GL_4\left(\Z_p[[\tilde\Gamma]]\right)$. \\}
We can consider the following $\Z_p[[\widetilde{\Gamma}]]$-modules that have a continuous action of $\Gal{\overline{\Q}_p}{\Q_p}$:
\begin{align*}
T_{\rhoDN{4}{2}} &=T_p(E) \ \hotimes_{\Z_p} T_{\p} \ \hotimes_{\Z_p} \ \Z_p[[\Gamma_\Cyc]](\kappa_{\Cyc}^{-1}),
                   \\ \Fil T_{\rhoDN{4}{2}} &:=T_p(E) \ \hotimes_{\Z_p} \ \Z_p[[\Gamma_\p]] (\kappa_\p^{-1})  \ \hotimes_{\Z_p} \ \Z_p[[\Gamma_\Cyc]](\kappa_{\Cyc}^{-1}).
\end{align*}
One has the following discrete $\Z_p[[\widetilde{\Gamma}]]$-modules:
\begin{align*}
  D_{\rhoDN{4}{2}} = T_{\rhoDN{4}{2}} \otimes_{\Z_p[[\widetilde{\Gamma}]]} \Z_p[[\widetilde{\Gamma}]]^\vee,  \qquad  \Fil D_{\rhoDN{4}{2}} = \Fil T_{\rhoDN{4}{2}} \otimes_{\Z_p[[\widetilde{\Gamma}]]} \Z_p[[\widetilde{\Gamma}]]^\vee.
\end{align*}
The decomposition given in (\ref{decomposition-p-splits}) lets us deduce the following short exact sequences of $\Z_p[[\widetilde{\Gamma}]]$-modules that is $\Gal{\overline{\Q}_p}{\overline{\Q}_p}$-equivariant:
\begin{align*}
  0 \rightarrow \Fil T_{\rhoDN{4}{2}}  \rightarrow T_{\rhoDN{4}{2}} \rightarrow \frac{T_{\rhoDN{4}{2}}}{\Fil T_{\rhoDN{4}{2}}}  \rightarrow 0, \qquad 0 \rightarrow \Fil D_{\rhoDN{4}{2}}  \rightarrow D_{\rhoDN{4}{2}} \rightarrow \frac{D_{\rhoDN{4}{2}}}{\Fil D_{\rhoDN{4}{2}}}  \rightarrow 0.
\end{align*}

The local condition at $p$, denoted $\Loc_\Gr(\Q_p,\Drho)$, is given below.
$$\Loc_\Gr(\Qp,\Drho):= \ker \left(H^1(\Q_p,\Drho)\rightarrow H^1\left(I_p,\frac{D_{\rhoDN{4}{2}}}{\Fil D_{\rhoDN{4}{2}}}\right)\right).
$$
Here, $I_p$ denotes the inertia subgroup inside the decomposition group $G_{\Qp}$. The discrete Selmer group, following Greenberg's construction in \cite{greenberg1994iwasawa}, is defined below:
\[
  \Sel^{\Gr}(\Q,D_{\rhoDN42}):=\ker\left(H^1(G_\Sigma,\Drho)\rightarrow\frac{H^1(\Q_p,\Drho)}{\Loc_\Gr(\Qp,\Drho)}\oplus\bigoplus_{\nu\in \Sigma\setminus\{p\}}H^1(\Q_\nu,\Drho)\right).
\]

\subsection{Hida's Rankin-Selberg $p$-adic $L$-function and the Iwasawa main conjecture}

The two-variable $p$-adic $L$-function $\theta^{\Gr}_{\cmmnt{\pmb}{4},2}$ in this setup has been constructed by Hida. We avoid stating the precise interpolation property satisfied by $\theta^{\Gr}_{\cmmnt{\pmb}{4},2}$. Instead, we refer the reader to Hida's  works\footnote{One must modify the $p$-adic $L$-function constructed in \cite{Hidafourier} by multiplying it with a one-variable $p$-adic $L$-function associated to the 3-dimensional adjoint representation $\Ad^0(\rho_F)$ (see Conjecture 1.0.1 in \cite{hida1996search}). A discussion surrounding the need to introduce this modification, which is related to the choice of a certain period (N\'eron period versus a period involving the Peterson inner product), is carefully explained in \cite{hida1996search}. See the introduction and Section 6 in Hida's article  \cite{hida1996search}.}  in \cite{Hidafourier} and \cite{hida1996search}. We will content ourselves with describing the critical set of specializations (following Greenberg's terminology in \cite{greenberg1994iwasawa}). As mentioned in the introduction, the natural restriction maps $\widetilde{\Gamma} \twoheadrightarrow \Gamma_\Cyc$ and $\widetilde{\Gamma} \twoheadrightarrow \Gamma_\p$  provide us the following isomorphism of topological groups: $$\widetilde{\Gamma} \cong \Gamma_\Cyc \times \Gamma_\p.$$

Note that $\Z_p[[\widetilde{\Gamma}]]$ is a coproduct in the category of complete semi-local Noetherian $\Z_p$-algebras (with respect to the maps $\Z_p[[\Gamma_\p]] \hookrightarrow \Z_p[[\widetilde{\Gamma}]]$ and  $\Z_p[[\Gamma_\Cyc]] \hookrightarrow \Z_p[[\widetilde{\Gamma}]]$). Note also that we have the following equalities of the group of continuous homomorphisms:
\begin{align*}
  \Hom_{\mathrm{ring}}\left(\Z_p[[\Gamma_\p]] , \overline{\Q}_p\right) = \Hom_{\mathrm{gp}}\left(\Gamma_\p , \overline{\Q}_p^\times\right), \  \  \Hom_{\mathrm{ring}}\left(\Z_p[[\Gamma_\Cyc]] , \overline{\Q}_p\right) = \Hom_{\mathrm{gp}}\left(\Gamma_\Cyc , \overline{\Q}_p^\times\right).
\end{align*}
Taking this point of view, one can consider the set $C \subset \Hom_{\cont}\left(\Z_p[[\widetilde{\Gamma}]] , \overline{\Q}_p\right)$ obtained via the following continuous group homomorphisms:
\begin{align*}
  \varphi_k: \Gamma_\p \rightarrow  \overline{\Q}_p^\times,  \qquad \varphi_{\Cyc}\epsilon_p : \Gamma_{\Cyc} \rightarrow \overline{\Q}_p^\times,
\end{align*}
\begin{itemize}
  \item where we allow the group homomorphism $\phi_k:\Gamma_\p \rightarrow \overline{\Q}_p^\times$ to vary over $p$-adic Galois characters, obtained via algebraic Hecke characters of $K$, with conductor equal to a power of $\p$ and of infinity type $(1-k)i \circ i_\p$ subject to the following restriction on its weight $k$:
        \begin{align*}
          k \geq 3,
        \end{align*}
  \item and where we allow $\varphi_{\Cyc}$ to vary over all the continuous group homomorphisms $\Gamma_{\Cyc} \rightarrow \overline{\Q}_p^\times$ of finite order. (The character $\epsilon_p: \Gamma_\Cyc \stackrel{\cong}{\rightarrow} 1+p\Z_p$ is given by the product $\chi_p\omega^{-1}$, where $\chi_p:G_\Q \rightarrow \Z_p^\times$ is the $p$-adic cyclotomic character and $\omega:G_\Q \rightarrow \Z_p^\times$ is the Teichmuller character giving the action of $G_\Q$ on the $p$-th roots of unity.)
\end{itemize}
This set $C$ is the critical set of specializations, corresponding to the fact that the Galois representation $\rhoDN{4}{2}$ satisfies the Panchishkin condition.\\

We recall Conjecture \ref{greenberg-main-conjecture} stated in the introduction.

\greenbergMC*

Progress towards establishing  Conjecture \ref{greenberg-main-conjecture} has been made by Xin Wan. See his works \cite{wan2,wan}.

\section{$\pm\pm$ Iwasawa main conjectures for elliptic curves with supersingular reduction at $p$}
\label{S:plusminus}

We shall recall the formulation of Kim's conjectures and the corresponding local conditions for the signed Selmer groups\footnote{Note that, when the elliptic curve $E$ has good supersingular reduction at $p$, the classical $p$-Selmer group of $E$ defined over the field $\tilde K_\infty$, is not well suited for the study of Iwasawa theory since its Pontryagin dual is not $\Iwalg$-torsion and will not satisfy a ``control theorem''.} at $p$. Consider the following discrete $\Iwalg$-modules that have a continuous action of $\Gal{\overline{\Q}_p}{\Q_p}$:
\begin{align*}
  D_{E,\kappa_\p^{-1}\kappa_\Cyc^{-1}} & := \bigg(T_p(E) \hotimes \Z_p[[\Gamma_\p]](\kappa_\p^{-1}) \hotimes \Z_p[[\Gamma_\Cyc]](\kappa_\Cyc^{-1}) \bigg) \otimes_{\Z_p[[\widetilde{\Gamma}]]} \Z_p[[\widetilde{\Gamma}]]^\vee, \\
  D_{E,\kappa_\q^{-1}\kappa_\Cyc^{-1}} & := \bigg(T_p(E) \hotimes \Z_p[[\Gamma_\p]](\kappa_\q^{-1}) \hotimes \Z_p[[\Gamma_\Cyc]](\kappa_\Cyc^{-1}) \bigg) \otimes_{\Z_p[[\widetilde{\Gamma}]]} \Z_p[[\widetilde{\Gamma}]]^\vee.
\end{align*}

The decomposition in equation (\ref{decomposition-p-splits}) gives us the following isomorphism of discrete $\Z_p[[\widetilde{\Gamma}]]$-modules that is $\Gal{\overline{\Q}_p}{\Q_p}$-equivariant:
\begin{align*}
  D_{\rhoDN{4}{2}} \cong D_{E,\kappa_\p^{-1}\kappa_\Cyc^{-1}} \oplus D_{E,\kappa_\q^{-1}\kappa_\Cyc^{-1}}.
\end{align*}
We obtain the following isomorphisms of discrete $\Z_p[[\widetilde{\Gamma}]]$-modules:
\begin{align} \label{first-dec}
  H^1\left(\Gal{\overline{\Q}_p}{\Q_p}, D_{\rhoDN{4}{2}} \right) & \cong H^1\left(\Q_p,D_{E,\kappa_\p^{-1}\kappa_\Cyc^{-1}}\right) \oplus H^1\left(\Q_p,D_{E,\kappa_\q^{-1}\kappa_\Cyc^{-1}}\right) \\  \notag &\cong H^1\left(K_\p,D_{E,\kappa_\p^{-1}\kappa_\Cyc^{-1}}\right) \oplus H^1\left(K_\p,D_{E,\kappa_\q^{-1}\kappa_\Cyc^{-1}}\right).
\end{align}

If we let $\delta$ denote the non-trivial element of $\Gal{K}{\Q}$ and $\tilde{\delta}$ some lift in $G_\Sigma$ of $\delta$, we have the following natural isomorphism inside $\Gal{\overline{\Q}}{\Q}$:
$$\tilde{\delta}^{-1}\Gal{\overline{\Q}_p}{K_\q}\tilde{\delta} \cong \Gal{\overline{\Q}_p}{K_\p}.$$
This lets us obtain the following isomorphism of discrete $\Z_p[[\widetilde{\Gamma}]]$-modules:
\begin{align} \label{second-dec}
  H^1(K_\p,D_{E,\kappa_\q^{-1}\kappa_\Cyc^{-1}}) \cong H^1\left(K_\q,D_{E,\kappa_\p^{-1}\kappa_\Cyc^{-1}}\right).
\end{align}

Shapiro's lemma (see the discussion in the introduction of Greenberg's work \cite{MR2290593}) provides us the following isomorphism of discrete $\Z_p[[\widetilde{\Gamma}]]$-modules:
\begin{align} \label{third-dec}
  H^1\left(K_\p,D_{E,\kappa_\p^{-1}\kappa_\Cyc^{-1}}\right) \cong  \bigoplus_{\P|\p}H^1\left((\tilde K_\infty)_\P,E[p^\infty]\right), \qquad H^1\left(K_\q,D_{E,\kappa_\p^{-1}\kappa_\Cyc^{-1}}\right) \cong \bigoplus_{\QQ|\q}H^1\left((\tilde K_\infty)_\QQ,E[p^\infty]\right).
\end{align}

Here, if we let $G_\P$ denote the decomposition group for the prime $\P$ inside $\Gal{\overline{\Q}_p}{\Q_p}$, then $(\widetilde{K}_\infty)_\P$ denotes the fixed field $\overline{\Q}_p^{G_\P}$. Note that the isomorphism in equation (\ref{third-dec}) crucially relies on identifying $\widetilde{\Gamma}$ with $\Gamma_\Cyc \times \Gamma_\p$. Note also that there are only finitely many primes $\P$ (and $\QQ$ respectively) in the field $\widetilde{K}_\infty$ lying above the prime $\p$ (and $\q$ respectively) of the imaginary quadratic field $K$. \\

Combining equations (\ref{first-dec}), (\ref{second-dec}) and (\ref{third-dec}), we have the following isomorphism of discrete $\Z_p[[\widetilde{\Gamma}]]$-modules:
\begin{align}\label{eq:shapirolocal}
  H^1\left(\Q_p,D_{\rhoDN42}\right) \cong \bigoplus_{\P|\p}H^1\left((\tilde K_\infty)_\P,E[p^\infty]\right) \oplus \bigoplus_{\QQ|\q}H^1\left((\tilde K_\infty)_\QQ,E[p^\infty]\right).
\end{align}
Keeping this isomorphism in mind, to define the local conditions inside the local cohomology group $H^1(\Q_p,D_{\rhoDN42})$, it will be sufficient to define local conditions inside the local cohomology groups $H^1\left((\tilde K_\infty)_\P,E[p^\infty]\right)$ and $H^1\left((\tilde K_\infty)_\QQ,E[p^\infty]\right)$ for all the primes $\P$ and $\QQ$ in $\tilde K_\infty$ lying above the primes $\p$ and $\q$ in $K$ respectively.

\subsection{The $\pm\pm$ Selmer groups of Kim}\label{S:Kim}

Let $\hat E$ denote the formal group attached to the minimal Weierstrass model of the elliptic curve $E$ over $\Z_p$. The height of the formal group $\hat E$ equals two.

Given an ideal $\mathfrak{n}$ of the ring of integers of $K$, we write $K(\mathfrak{n})$ for the ray class field of $K$ with modulus $\mathfrak{n}$. Let  $\P$ be a fixed prime of $K(p^\infty)$ lying above $\p$. By an abuse of notation, we denote the prime $L\cap \P$ again by $\P$ whenever $L$ is a sub-extension of $K(p^\infty)/K$.  For non-negative integers $m$ and $n$, Kim \cite{MR3224266} defined
\begin{align*}
  E^+(K(\p^m\q^n)_\P) & =\left\{P\in \hat E(K(\p^m\q^n)_\P):\Tr_{m/l+1,n}P\in \hat E(K(\p^l\q^n)_\P), \text{ for   odd } l<m\right\};  \\
  E^-(K(\p^m\q^n)_\P) & =\left\{P\in \hat E(K(\p^m\q^n)_\P):\Tr_{m/l+1,n}P\in \hat E(K(\p^l\q^n)_\P), \text{ for   even } l<m\right\},
\end{align*}
where $\Tr_{m/l+1,n}:\hat E(K(\p^m\q^n)_\P) \rightarrow \hat E(K(\p^{l+1}\q^n)_\P)$ denotes the trace map. We define $E^\pm(K(p^\infty)_\P)$ to be the union $\bigcup_{m,n\ge 0} E^\pm(K(\p^m\q^n)_\P)$. Let $\Delta$ be the finite Galois group $\mathrm{Gal}\left(K(p^\infty)_\P/(\Kinf)_\P\right)$. Then, we may define the  corresponding plus/minus subgroups by
\[
  E^\pm((\Kinf)_\P):=E^\pm(K(p^\infty)_\P)^\Delta\subset E((\Kinf)_\P)
\]
Given a prime $\QQ$ of $K(p^\infty)$ lying above $\q$, we may define $E^\pm((\Kinf)_\QQ)$ in a similar manner. \\

For each prime $\P$ (and $\QQ$ respectively) in $\widetilde{K}_\infty$, lying above $\p$ (and $\q$ respectively), we consider the following discrete $\Z_p[[\widetilde{\Gamma}]]$-submodules of $H^1((\tilde K_\infty)_\P,E[p^\infty])$ (and $H^1((\tilde K_\infty)_\QQ,E[p^\infty])$ respectively):
\begin{align*}
  \Loc_\pm\left((\tilde K_\infty)_\P,E[p^\infty]\right)   & := \mathrm{Image}\left(E^\pm\left((\Kinf)_\P\right)\otimes \Q_p/\Z_p  \stackrel{\varkappa_\P}{\hookrightarrow} H^1((\tilde K_\infty)_\P,E[p^\infty])\right),   \\
  \Loc_\pm\left((\tilde K_\infty)_\QQ, E[p^\infty]\right) & := \mathrm{Image}\left(E^\pm\left((\Kinf)_\QQ\right)\otimes \Q_p/\Z_p \stackrel{\varkappa_\QQ}{\hookrightarrow} H^1((\tilde K_\infty)_\QQ,E[p^\infty])\right).
\end{align*}

Here, $\varkappa_\P$ and $\varkappa_\QQ$ denote the usual Kummer map. Note that the restriction of the Kummer map to the groups $E^{\pm}((\tilde K_\infty)_{\P}) \otimes \Q_p/\Z_p$ and $E^{\pm}((\tilde K_\infty)_{\QQ}) \otimes \Q_p/\Z_p$  are injections since the natural maps $E^{\pm} ((\tilde K_\infty)_{\P})\otimes \Q_p/\Z_p \rightarrow E((\tilde K_\infty)_{\P}) \otimes \Q_p/\Z_p$ and $E^{\pm} ((\tilde K_\infty)_{\QQ})\otimes \Q_p/\Z_p \rightarrow E((\tilde K_\infty)_{\QQ}) \otimes \Q_p/\Z_p$ are injections. This fact follows from an argument similar to the one given in Lemma 8.17 of Kobayashi's work \cite{kobayashi}.

Let  $\bullet,\circ\in\{+,-\}$. Via \eqref{eq:shapirolocal}, we define the following local condition:
\[
  \Loc_{\bullet\circ}(\Qp,\Drho):= \bigoplus_{\P \mid \p} \Loc_\bullet\left((\tilde K_\infty)_\P,E[p^\infty]\right) \oplus  \bigoplus_{\QQ \mid \q} \Loc_\circ \left((\tilde K_\infty)_\QQ,E[p^\infty]\right)  \subset H^1(\Q_p,\Drho).
\]

This, in turn, lets us define the following discrete Selmer group:
\[
  \Sel^{\bullet\circ}(\Q,D_{\rhoDN42})=\ker\left(H^1\left(G_\Sigma,D_{\rhoDN42}\right)\rightarrow\frac{H^1(\Q_p,\Drho)}{\Loc_{\bullet\circ}(\Qp,\Drho)} \oplus \bigoplus_{v\in \Sigma\setminus\{p\}}H^1(\Q_v,\Drho) \right).
\]

\begin{remark}
  Note that our choice of signs $+$ and $-$ is opposite to that of Kim since we would like to formulate the Iwasawa main conjecture using the convention for the signs $+$ and $-$ chosen by Loeffler in \cite{loefflerray}. In particular, the Selmer groups $\Sel^{++}(\Q,D_{\rhoDN42})$, $\Sel^{+-}(\Q,D_{\rhoDN42})$, $\Sel^{-+}(\Q,D_{\rhoDN42})$ and $\Sel^{--}(\Q,D_{\rhoDN42})$ respectively correspond to the $--$, $-+$, $+-$ and $++$ Selmer groups respectively appearing in Kim's work \cite{MR3224266}.
\end{remark}

\subsection{Two-variable main conjectures and two variable $p$-adic $L$-functions}

Following Loeffler's work \cite{loefflerray}, we recall the construction of the four $\pm,\pm$ $2$-variable $p$-adic $L$-functions attached to $E$. Suppose we are given a choice of $\lambda,\mu\in\left\{\pm\sqrt{-p}\right\}$. There exists an unbounded $\Qp(\sqrt{-p})$-valued  distribution $L_{\lambda,\mu}$ on $\tilde{\Gamma}$, satisfying the following interpolation formula:
\begin{equation}\label{eq:interpolation1}
  L_{\lambda,\mu}(\psi)=\frac{p^{m+n}}{\lambda^m\mu^n\tau_K(\psi^{-1})}\times\frac{L(E/K,\psi^{-1},1)}{\Omega_{f_E}^+\Omega_{f_E}^-},
\end{equation}
for every finite Galois character $\psi:\Gal{\overline{K}}{K}\rightarrow \overline{\Q}^\times$ with conductor $\p^m\q^n$ for some $m,n>0$. Here $\tau_K(\psi^{-1})$ is the Gauss sum of $\psi^{-1}$ (see, for example, the work of Skinner-Urban \cite[\S8.1.3]{SU}). The periods $\Omega_{f_E}^\pm$ are the real and imaginary periods  associated to the newform $f_E$ constructed by Shimura (\cite{MR0463119}). See also Theorem 3.5.4 in the work of Greenberg-Stevens \cite{MR1279610}. These periods are well defined up to units in $\overline{\Z}_p^\times$.

\begin{remark} \label{remark:periodnormalization}
  There are three choices of normalizations (depending on the choice of complex periods) available to us for the construction of these unbounded measures. These choices, denoted  $\Omega_{f_E}^{\mathrm{cong}}$, $\Omega_\Pi$ and $\Omega_{f_E}^+\Omega_{f_E}^-$ respectively, appear in works of Castella-Wan \cite{CW}, Loeffler \cite{loefflerray} and Wan \cite{wan} respectively. We refer the reader to these works for the precise definition of $\Omega_{f_E}^{\mathrm{cong}}$ and $\Omega_\Pi$.

  Lemma 9.5 in the work of Skinner-Zhang \cite{SkinnerZhang} asserts that the $p$-adic valuation of the ratio $\frac{\Omega_{f_E}^{\mathrm{cong}}}{\Omega_{f_E}^+\Omega_{f_E}^-}$ is zero.  Castella-Wan \cite[Remark 2.4]{CW} assert that the ratio of the periods $\frac{\Omega_\Pi}{\Omega_{f_E}^+\Omega_{f_E}^-}$ is an element of $\Q^\times$. However, it is not clear to us whether the ratio $\frac{\Omega_\Pi}{\Omega_{f_E}^+\Omega_{f_E}^-}$ is a $p$-adic unit.

  We have chosen to work with the complex period  given in Wan's work \cite{wan} because, as we indicate in Remark \ref{rem:cyclotomicspecialization}, this choice makes it transparent how the cyclotomic specializations of the $2$-variable $++$ and $--$ $p$-adic $L$-functions are related to certain one-variable $p$-adic $L$-functions of Rob Pollack. See Remark \ref{rem:cyclotomicspecialization} for more details.
\end{remark}

Fix a topological generator $\gamma_\p$ of $\Gamma_\p$. We choose $\gamma_\q$ to be the image of $\gamma_\p$ under the ring automorphism $\Z_p[[\widetilde{\Gamma}]] \rightarrow \Z_p[[\widetilde{\Gamma}]]$ induced by complex conjugation. We identify $L_{\lambda,\mu}$ with its Amice transform in $\Qp(\sqrt{-p})[[\gamma_\p-1,\gamma_\q-1]]$.  As power series, we may decompose these $p$-adic $L$-functions to obtain four bounded $\Q_p$-valued measures $\tDN42{\bullet\circ}$ on $\tilde{\Gamma}$, for $\bullet,\circ\in\{+,-\}$:
\begin{equation}\label{eq:decomp}
  L_{\lambda,\mu}=\log_\p^+\log_\q^+\tDN42{++}+\lambda\log_\p^+\log_\q^-\tDN42{+-}+\mu\log_\p^-\log_\q^+\tDN42{-+}+\lambda\mu\log_\p^-\log_\q^-\tDN42{--},
\end{equation}
where $\log_\p^\pm$ and $\log_\q^\pm$ are defined using  Pollack's $\pm$ logarithms in \cite{pollack}. As power series, we have
\[
  \log_\fr^+={\frac{1}{p}}\prod_{i\ge1}\frac{\Phi_{p^{2i}}(\gamma_{\fr})}{p};\quad
  \log_\fr^-={\frac{1}{p}}\prod_{i\ge1}\frac{\Phi_{p^{2i-1}}(\gamma_{\fr})}{p},
\]
where $\Phi_{p^m}$ denotes the $p^m$-th cyclotomic polynomial.  The corresponding distributions under Amice transform are described in \cite{dionlei}.   \\

If $\psi$ is a finite Hecke character on $\widetilde\Gamma$ of conductor $\mathfrak{f}$ for $\mathfrak{r}\in\{\p,\q\}$, then $\log_\mathfrak{r}^+(\psi)$ (respectively, $\log_\mathfrak{r}^-(\psi)$) vanishes if and only if $\mathrm{ord}_{\mathfrak{r}}(\mathfrak{f})>0$ is odd (respectively, even).
On combining equations \eqref{eq:interpolation1} and \eqref{eq:decomp}, we may deduce the following interpolation formulae for $\tDN42{\bullet\circ}$. Let $\bullet,\circ\in\{+,-\}$. Let $\psi$ be a finite Hecke character on $\tilde\Gamma$ of conductor $\p^m\q^n$ such that $\log_\p^\bullet(\psi)\log_\q^\circ(\psi)\ne0$. Then,
\begin{align}
  \tDN42{++}(\psi) & =\frac{(-1)^{\frac{m+n}{2}}p^{m+n}}{\tau_K(\psi^{-1}){\displaystyle\prod_{\substack{1 \leq k<m-1   \\ k \text{ even}}}}  \Phi_{p^k}(\zeta_\p) {\displaystyle\prod_{\substack{1 \leq l < n-1 \\ l \text{  even}}}}  \Phi_{p^l}(\zeta_\q)}\times \frac{L(E/K,\psi^{-1},1)}{\Omega_{f_E}^+\Omega_{f_E}^-},\label{eq:loeffler--}\\
  \tDN42{+-}(\psi) & =\frac{(-1)^{\frac{m+n+1}{2}}p^{m+n}}{\tau_K(\psi^{-1}){\displaystyle\prod_{\substack{1 \leq k<m-1 \\ k \text{ even}}}}  \Phi_{p^k}(\zeta_\p) {\displaystyle\prod_{\substack{1 \leq l < n-1 \\ l \text{  odd}}}}  \Phi_{p^l}(\zeta_\q)}\times \frac{L(E/K,\psi^{-1},1)}{\Omega_{f_E}^+\Omega_{f_E}^-},\\
  \tDN42{-+}(\psi) & =\frac{(-1)^{\frac{m+n+1}{2}}p^{m+n}}{\tau_K(\psi^{-1}){\displaystyle\prod_{\substack{1 \leq k<m-1 \\ k \text{ odd}}}}  \Phi_{p^k}(\zeta_\p) {\displaystyle\prod_{\substack{1 \leq l < n-1 \\ l \text{  even}}}}  \Phi_{p^l}(\zeta_\q)}\times \frac{L(E/K,\psi^{-1},1)}{\Omega_{f_E}^+\Omega_{f_E}^-},\\
  \tDN42{--}(\psi) & =\frac{(-1)^{\frac{m+n}{2}+1}p^{m+n}}{\tau_K(\psi^{-1}){\displaystyle\prod_{\substack{1 \leq k<m-1 \\ k \text{ odd}}}}  \Phi_{p^k}(\zeta_\p) {\displaystyle\prod_{\substack{1 \leq l < n-1 \\ l \text{  odd}}}}  \Phi_{p^l}(\zeta_\q)}\times \frac{L(E/K,\psi^{-1},1)}{\Omega_{f_E}^+\Omega_{f_E}^-},\label{eq:loeffler++}
\end{align}
where $\zeta_\mathfrak{r}=\psi(\gamma_\mathfrak{r})$ for $\mathfrak{r}=\p,\q$. We now recall Conjecture \ref{kim-main-conjecture} stated in the introduction.

\kimMC*

\begin{remark}\label{rk:choices}
  For $\fr \in \{\p,\q\}$,  if a different topological generator of $\Gamma_\fr$ is chosen, then $\log_\fr^\pm$ would change by a factor $u_\fr^\pm$, for some unit $u_\fr^\pm$ in the ring $\Z_p[[\Gamma_\fr]]$. The $p$-adic $L$-functions $\tDN{4}{2}{\bullet\circ}$ would then change by the factor $\left(u_\p^{\bullet}u_\q^{\circ}\right)^{-1}$. As a result, the divisors of the $p$-adic $L$-functions are independent of the choice of topological generators for $\Gamma_\p$ and $\Gamma_\q$.
\end{remark}

Progress towards Conjecture~\ref{kim-main-conjecture} has been made by various authors under various technical hypotheses. We will cite the relevant works, referring the interested reader instead to these works for the statement of the hypotheses. Wan \cite[Theorem 8.5]{wan} has made progress towards establishing the inequality $\Div\left(\tDN42{++}\right) \leq \Div\left(\Sel^{++}(\Q,D_{\rhoDN42})^\vee\right)$ in $Z^1\left(\Z_p[[\widetilde{\Gamma}]]\right)$. For the reverse inequality, one can employ an argument involving Euler systems of Beilinson-Flach elements constructed by Loeffler-Zerbes in \cite{LZColeman}. Under some technical hypotheses, such an Euler system is constructed and used to prove this direction of the main conjecture in \cite[Theorem 1.3]{BL}. Recently, a proof of Conjecture~\ref{kim-main-conjecture} has been announced in the work of Castella-\c{C}iperiani-Skinner-Sprung \cite[Theorem~A]{CCSS}.

\section{Proof of Theorem~\ref{maintheorem}} \label{S:proofmaintheorem}

Theorem~\ref{maintheorem} follows from Theorem \ref{general-main-theorem}. We will need to verify all the conditions stated in Theorem \ref{general-main-theorem}. Since we are working over the regular local ring $\Z_p[[\widetilde{\Gamma}]]$, condition (\ref{hyp-two-theorem}) is automatically satisfied.

One has to verify all the hypotheses in Section \ref{varioushypotheses}. \ref{tor-loc-sigma} follows from Proposition 4.1 in \cite{palvannan2016algebraic}. Since we are working over the regular local ring $\Z_p[[\widetilde{\Gamma}]]$, \ref{Sigma0-pd} and \ref{Gor} automatically hold.

It remains to verify \ref{rank}, \ref{loc-free} and \ref{locp0}. \ref{locp0} follows from Proposition \ref{prop:residually-irred}. \ref{rank} follows from Corollaries \ref{cor:kimfree}, \ref{cor:locgr} and \ref{cor:union}.  \ref{loc-free} follows from Corollaries \ref{cor:kimfree}, \ref{cor:locgr} and  Proposition \ref{prop:verlf}. Note that
\begin{align*}
  d(\rhoDN{4}{2})=4, \qquad d^+(\rhoDN{4}{2})=2, \qquad d^-(\rhoDN{4}{2})=2
\end{align*}

For the remainder of this section, it will be helpful to keep the following picture in mind:
\begin{center}
  \begin{tikzpicture}[auto]
    \node (H1) {$H^1\left(\Gal{\overline{\Q}_p}{\Q_p}, D_{\rhoDN{4}{2}} \right)$};
    \node (H1p) [right of=H1, node distance = 5cm]{$H^1\left(K_\p,D_{E,\kappa_\p^{-1}\kappa_\Cyc^{-1}}\right)$};

    \node(sumpq) [right of =H1p, node distance = 2.7cm] {$\bigoplus$};

    \node (H1q)[right of =H1p, node distance = 5.5cm] {$H^1\left(K_\q,D_{E,\kappa_\p^{-1}\kappa_\Cyc^{-1}}\right)$};

    \node(H1P)[below of = H1p, node distance = 1.7cm]{$\bigoplus \limits_{\P|\p}H^1\left((\tilde K_\infty)_\P,E[p^\infty]\right)$};

    \node(sumPQ) [right of =H1P, node distance = 2.7cm] {$\bigoplus$};

    \node(H1Q)[below of = H1q, node distance = 1.7cm]{$\bigoplus \limits_{\QQ|\q}H^1\left((\tilde K_\infty)_\QQ,E[p^\infty]\right)$};

    \node (circbulletP)[below of = H1P, node distance = 1.7 cm] {$\bigoplus \limits_{\P \mid \p} \left(E^\bullet\left((\Kinf)_\P\right)\otimes \Q_p/\Z_p\right) $};

    \node(sumPQcircbullet) [right of =circbulletP, node distance = 2.7cm] {$\bigoplus$};

    \node (circbulletQ)[below of = H1Q, node distance = 1.7 cm] {$\bigoplus \limits_{\QQ \mid \q} \left(E^\circ\left((\Kinf)_\QQ\right)\otimes \Q_p/\Z_p \right)$};

    \node (circbullet)[left of = circbulletP, node distance = 5 cm] {$\Loc_{\bullet\circ}(\Qp,\Drho)$};

    \node (LocGr) [above of=H1p, node distance = 1.4cm] {$\Loc_\Gr (\Q_p,D_{\rhoDN42})$};
    \draw[->] (H1) to node  {$\cong$} (H1p);
    \draw[right hook->] (circbulletP) to node  {} (H1P);
    \draw[right hook->] (circbulletQ) to node  {} (H1Q);

    \draw[->] (H1P) to node  {$\cong$} (H1p);
    \draw[->] (H1Q) to node  {$\cong$} (H1q);

    \draw[->] (H1p) to node  [left] {$\cong$} node [right] {(Corollary \ref{cor:locgr})}(LocGr);
    \draw[->] (circbullet) to node  {$\cong$} (circbulletP);
  \end{tikzpicture}
  \captionof{figure}{} \label{fig:1}
\end{center}
The identification of  $H^1\left(K_\p,D_{E,\kappa_\p^{-1}\kappa_\Cyc^{-1}}\right)$  with $\Loc_\Gr (\Qp,D_{\rhoDN42})$ is established in Corollary \ref{cor:locgr}. All the remaining inclusions and isomorphisms given in Figure \ref{fig:1} follow from the discussions in Section~\ref{S:plusminus}.

\subsection{Verifying the various hypotheses}

Let $\m$ denote the maximal ideal of local ring $\Iwalg$. Let $\mathbb{F}_p(\omega)$ denote the one-dimensional $\mathbb{F}_p$ vector on which $G_\Sigma$ acts via the Techm\"uller character. The residual representation $\overline{\rho}_E:\Gal{\overline{\Q}_p}{\Q_p} \rightarrow \Gl_2(\mathbb{F}_p)$ associated to the elliptic curve $E$ for the local decomposition group $\Gal{\overline{\Q}_p}{\Q_p}$ is absolutely irreducible (see \cite[Theorem~2.6]{Edix}). This uses the fact that $E$ is an elliptic curve defined over $\Q$ with good supersingular reduction at the prime $p$. As a result, we obtain the following proposition:

\begin{proposition} \label{prop:residually-irred} \mbox{}
  \begin{enumerate}
    \item The $\Gal{\overline{\Q}_p}{\Q_p}$-modules $D_{\rhoDN{4}{2}}[\m]$ and $D_{\rhoDNstar{4}{2}}[\m]$ have no quotient isomorphic to the trivial representation.
    \item The $\Gal{\overline{\Q}_p}{K_\p}$-module $D_{E,\kappa_\p^{-1}\kappa_\Cyc^{-1}}[\m]$  has no quotient isomorphic to the trivial representation or $\mathbb{F}_p(\omega)$.
    \item The $\Gal{\overline{\Q}_p}{K_\q}$-module $D_{E,\kappa_\p^{-1}\kappa_\Cyc^{-1}}[\m]$  has no quotient isomorphic to the trivial representation or $\mathbb{F}_p(\omega)$.
  \end{enumerate}
\end{proposition}

As an immediate corollary to Proposition \ref{prop:residually-irred}, we have the following equalities:

\begin{align*}
  H^0\left(\Gal{\overline{\Q}_p}{\Q_p},D_{\rhoDN{4}{2}}\right)^\vee = H^2\left(\Gal{\overline{\Q}_p}{\Q_p},D_{\rhoDN{4}{2}}\right)^\vee =0.
\end{align*}
\begin{align*}
  H^0\left(K_\p,D_{E,\kappa_\p^{-1}\kappa_\Cyc^{-1}}\right) = H^2\left(K_\p,D_{E,\kappa_\p^{-1}\kappa_\Cyc^{-1}}\right) = 0. \qquad
  H^0\left(K_\p,D_{E,\kappa_\q^{-1}\kappa_\Cyc^{-1}}\right) = H^2\left(K_\p,D_{E,\kappa_\q^{-1}\kappa_\Cyc^{-1}}\right) =0.
\end{align*}

The local Euler Poincar\'e characteristics (Proposition 4.2 in \cite{MR2290593}) along with Proposition 5.10 and  Remark 5.10.1 in \cite{MR2290593} lets us immediately deduce the following corollary:

\begin{corollary} \label{cor:freeness-four-two}
  The $\Iwalg$-module $H^1\left(\Gal{\overline{\Q}_p}{\Q_p},D_{\rhoDN{4}{2}}\right)^\vee$ is free of rank four. The $\Iwalg$-modules $H^1\left(K_\p,D_{E,\kappa_\p^{-1}\kappa_\Cyc^{-1}}\right)^\vee$ and $H^1\left(K_\p,D_{E,\kappa_\q^{-1}\kappa_\Cyc^{-1}}\right)^\vee$ are free of rank two.
\end{corollary}

Let us recall the following result of Kim.

\begin{proposition}\label{prop:freekim}
  Let $\bullet, \circ \in \{+,-\}$. The following $\Iwalg$-modules are free of rank one:
  $$\bigoplus_{\P \mid \p} \Loc_\bullet\left((\tilde K_\infty)_\P,E[p^\infty]\right)^\vee,  \qquad \bigoplus_{\QQ \mid \q} \Loc_\circ \left((\tilde K_\infty)_\QQ,E[p^\infty]\right)^\vee.$$
\end{proposition}
\begin{proof}
  We will simply show that the $\Iwalg$-module $\bigoplus_{\P \mid \p} \Loc_\bullet\left((\tilde K_\infty)_\P,E[p^\infty]\right)^\vee$ is free of rank one. One can similarly show that the $\Iwalg$-module $\bigoplus_{\QQ \mid \q} \Loc_\circ \left((\tilde K_\infty)_\QQ,E[p^\infty]\right)^\vee$ is free of rank one.  Let us fix a prime $\P _0$ above $\p$ in $\widetilde{K}_\infty$. Let $\widetilde{\Gamma}_{\P_0}$ denote the decomposition group of $\widetilde{\Gamma}$ for the prime $\P_0$ lying above the prime $\p$ in $K$. Note that we have the following isomorphism of $\Iwalg$-modules:
  \begin{align*}
    \bigoplus_{\P \mid \p} \Loc_\bullet\left((\tilde K_\infty)_\P,E[p^\infty]\right)^\vee \cong \Ind^{\widetilde{\Gamma}}_{\widetilde{\Gamma}_{\P_0}} \left(E^\bullet\left((\Kinf)_{\P_0}\right)\otimes \Q_p/\Z_p\right)^\vee
  \end{align*}
  It suffices to show that the $\Z_p[[\widetilde{\Gamma}_{\P_0}]]$-module $\left(E^\bullet\left((\Kinf)_{\P_0}\right)\otimes \Q_p/\Z_p\right)^\vee$ is free of rank one. This follows from a result of Kim \cite[Proposition~2.11]{MR3224266}.
\end{proof}

Let $\bullet, \circ \in \{+,-\}$.  Since we have the following equality of $\Iwalg$-modules:
$$ \Loc_{\bullet\circ}(\Qp,\Drho) = \bigoplus_{\P \mid \p} \Loc_\bullet\left((\tilde K_\infty)_\P,E[p^\infty]\right) \oplus  \bigoplus_{\QQ \mid \q} \Loc_\circ \left((\tilde K_\infty)_\QQ,E[p^\infty]\right),$$
we immediately have the following corollary:

\begin{corollary} \label{cor:kimfree}
  Let $\bullet, \circ \in \{+,-\}$. The $\Iwalg$-module $\Loc_{\bullet\circ}(\Qp,\Drho)^\vee$ is free of rank two.
\end{corollary}

\begin{lemma}\label{lem:torfree}
  Let $\QQ$ (and $\P$ respectively) be a prime of $\widetilde{K}_\infty$ lying above the prime $\q$ (and $\p$ respectively) of $K$. Then, we have
  \begin{align*}
    H^0\left(I_\QQ,E[p^\infty]\right) = H^0\left(I_\P,E[p^\infty]\right) = H^0\left(G_\QQ,E[p^\infty]\right) = H^0\left(G_\P,E[p^\infty]\right) = 0.
  \end{align*}
\end{lemma}
Here, the groups $I_\QQ$, $G_\QQ$ (and $I_\P$, $G_\P$ respectively) denote the inertia and decomposition subgroups for the primes $\QQ$ (and $\P$ respectively) inside $\Gal{\overline{\Q}_p}{K_\q}$ (and $\Gal{\overline{\Q}_p}{K_\p}$ respectively).

\begin{proof}
  Note that if we show $H^0\left(I_\QQ,E[p^\infty]\right)=H^0\left(I_\P,E[p^\infty]\right)=0$, then we would have $H^0\left(G_\QQ,E[p^\infty]\right)= H^0\left(G_\P,E[p^\infty]\right)=0$.  We will simply show that $H^0\left(I_\QQ,E[p^\infty]\right) =0$. One can similarly prove that $H^0\left(I_\P,E[p^\infty]\right) = 0$.\\

  Since the quotient $\frac{I_\q}{I_\QQ}$ is a pro-p group and since $E[p]$ is a discrete torsion $p$-group, we have
  \begin{align*}
    H^0\left(I_\QQ,E[p^\infty]\right) = 0 \iff H^0\left(I_\QQ,E[p]\right)  = 0 \iff H^0\left(I_\q,E[p]\right) = 0.
  \end{align*}
  Therefore, to complete the proof of the lemma, it suffices to show that $H^0\left(I_\q,E[p]\right)=0$. By \cite[Theorem~2.6]{Edix}, we have the following isomorphism of $F_p[I_q]$-modules:
  \begin{align*}
    E[p] \cong \F_p(\psi) \oplus \F_p(\psi'),
  \end{align*}
  where $\psi:I_q \rightarrow \mathbb{F}_{p^2}^\times$ and $\psi':I_q \rightarrow \mathbb{F}_{p^2}^\times$ are two fundamental characters of level two. That is, $\mathrm{Image}(\psi)$ and $\mathrm{Image}(\psi')$ lie inside $\mathbb{F}_{p^2}$ but not inside $\mathbb{F}_p$. As a result,
  \begin{align*}
    H^0\left(I_\q,E[p]\right) \cong H^0\left(I_\q,\F_p(\psi)\right) \oplus H^0\left(I_\q,\F_p(\psi)\right)  =0.
  \end{align*}
  This completes the proof of the lemma. \\

  Alternatively, to deduce the lemma, we can adopt the techniques of the proof appearing in Kobayashi's work \cite[Proposition~8.7]{kobayashi}, which relies on the fact the $E[p]$  is isomorphic to the $p$-torsion points on a formal group of height two over the local ring $\Z_p$.  See also Kim's work \cite[discussion towards the end of page 829]{MR3224266}).
\end{proof}

\begin{corollary}\label{cor:locgr}
  Under the isomorphism given in equation (\ref{first-dec}), we can identify  $\Loc_\Gr (\Q_p,D_{\rhoDN42})$  with  $H^1\left(K_\p,D_{E,\kappa_\p^{-1}\kappa_\Cyc^{-1}}\right)$. Consequently, the $\Iwalg$-module $\Loc_\Gr (\Q_p,D_{\rhoDN42})^\vee$ is free of rank two.
\end{corollary}

\begin{proof}
  Under the isomorphism
  \begin{align*}
    D_{\rhoDN{4}{2}} \cong D_{E,\kappa_\p^{-1}\kappa_\Cyc^{-1}} \oplus D_{E,\kappa_\q^{-1}\kappa_\Cyc^{-1}},
  \end{align*}
  we can identify
  $\Fil D_{\rhoDN{4}{2}}$ with $D_{E,\kappa_\p^{-1}\kappa_\Cyc^{-1}} $.  As a result, under the isomorphism given in equation (\ref{first-dec}), we can identify $\Loc_\Gr (\Q_p,D_{\rhoDN42})$ with
  \begin{align*} \label{eq:iso-coincide}
    \Loc_\Gr (\Q_p,D_{\rhoDN42})\cong H^1\left(K_\p,D_{E,\kappa_\p^{-1}\kappa_\Cyc^{-1}}\right) \oplus \ker \left( H^1\left(K_\q,D_{E,\kappa_\p^{-1}\kappa_\Cyc^{-1}}\right) \rightarrow H^1\left(I_\q,D_{E,\kappa_\p^{-1}\kappa_\Cyc^{-1}}\right) \right).
  \end{align*}
  Here, $I_\q$ denotes the inertia group inside $\Gal{\overline{\Q}_p}{K_\q}$. The inflation-restriction exact sequence gives us the following isomorphism of $\Iwalg$-modules:
  \[
    H^1\left(G_{K_\q}/I_\q,H^0\left(I_\q, D_{E,\kappa_\p^{-1}\kappa_\Cyc^{-1}}\right)\right) \cong \ker \left( H^1\left(K_\q,D_{E,\kappa_\p^{-1}\kappa_\Cyc^{-1}}\right) \rightarrow H^1\left(I_\q,D_{E,\kappa_\p^{-1}\kappa_\Cyc^{-1}}\right) \right).
  \]
  Note that we have the following isomorphism of $\Iwalg$-modules:
  \begin{align*}
    H^0\left(I_\q,D_{E,\kappa_\p^{-1}\kappa_\Cyc^{-1}}\right) \cong \bigoplus \limits_{\QQ \mid \q} H^0\left(I_\QQ, E[p^\infty]\right)
  \end{align*}
  Here, the direct sum is taken over all primes $\QQ$ of  $\widetilde{K}_\infty$ lying above the prime $\q$ of $K$. The groups $I_{\QQ}$ and $G_{\QQ}$ denote the inertia and decomposition subgroups of $\QQ$ inside $\Gal{\overline{\Q}_p}{K_\q}$.
  By Lemma~\ref{lem:torfree}, we have $$H^0\left(I_\QQ, E[p^\infty]\right)=0, \quad \forall \ \QQ \mid \q.$$
  These observations now give us following isomorphism of $\Iwalg$-modules: $$\Loc_\Gr (\Q_p,D_{\rhoDN42}) \cong H^1\left(K_\p,D_{E,\kappa_\p^{-1}\kappa_\Cyc^{-1}}\right).$$
  The last part of the corollary follows from Corollary \ref{cor:freeness-four-two}.
\end{proof}

\begin{proposition}\label{prop:sumquotient} Let $\bullet,\circ\in\{+,-\}$. The Pontraygin dual of the $\Iwalg$-modules
  \begin{align} \label{al:quot}
    \frac{\Loc_{+\circ}(\Qp,\Drho)+\Loc_{-\circ}(\Qp,\Drho)}{\Loc_{\bullet\circ}(\Qp,\Drho)}, \quad \frac{\Loc_{\bullet+}(\Qp,\Drho)+\Loc_{\bullet-}(\Qp,\Drho)}{\Loc_{\bullet\circ}(\Qp,\Drho)}
  \end{align}
  are free of rank one.
\end{proposition}
\begin{proof}
  We will show that the Pontryagin dual of the $\Iwalg$-module $\frac{\Loc_{\bullet+}(\Qp,\Drho)+\Loc_{\bullet-}(\Qp,\Drho)}{\Loc_{\bullet\circ}(\Qp,\Drho)}$ is free of rank one. The proof that the Pontryagin dual of the $\Iwalg$-module $\frac{\Loc_{+\circ}(\Qp,\Drho)+\Loc_{-\circ}(\Qp,\Drho)}{\Loc_{\bullet\circ}(\Qp,\Drho)}$ is free of rank one would follow similarly.

  Let $\QQ$ be any prime of $\tilde K_\infty$ lying above $p$ and write $\varkappa_\QQ$ for the Kummer map as before. Note that the $p$-power torsion points on the elliptic curve coincides with the $p$-power torsion points on the formal group $\hat{E}$.

  By Proposition 4.3 in the work of Coates-Greenberg \cite{coates1996kummer} (see also Section 2 of Rubin's work \cite{rubin85}), we have
  \[
    \varkappa_\QQ\left(E((\tilde K_\infty)_\QQ)\otimes \Qp/\Z_p\right)=H^1\left((\tilde K_\infty)_\QQ,E[p^\infty]\right).
  \]
  A result of Kim \cite[Proposition~2.6]{MR3224266} says that
  \begin{align} \label{eq:plusminushat}
    E^+((\tilde K_\infty)_\QQ)+E^-((\tilde K_\infty)_\QQ)=\hat E((\tilde K_\infty)_\QQ).
  \end{align}
  Recall from \S\ref{S:Kim} that we have injections $E^\pm((\tilde K_\infty)_\QQ)\otimes\Qp/\Z_p\hookrightarrow \hat E((\tilde K_\infty)_\QQ)\otimes\Qp/\Z_p$.  This fact along with equation (\ref{eq:plusminushat}) now implies the following:
  \begin{align}\label{eq:sum}
    \bigoplus \limits_{\QQ \mid \q} \left(E^+\left((\Kinf)_\QQ\right)\otimes \Q_p/\Z_p \right)\  + \ \bigoplus \limits_{\QQ \mid \q} \left(E^-\left((\Kinf)_\QQ\right)\otimes \Q_p/\Z_p \right) & = \bigoplus_{\QQ \mid \q} \left(\hat E((\tilde K_\infty)_\QQ) \otimes \Q_p/\Z_p\right) \\ \notag  & \cong \bigoplus \limits_{\QQ|\q}H^1\left((\tilde K_\infty)_\QQ,E[p^\infty]\right), \\
    & \notag \cong H^1\left(K_\q,D_{E,\kappa_\p^{-1}\kappa_\Cyc^{-1}}\right).
  \end{align}
  Therefore, the quotient appearing in equation (\ref{al:quot}) is isomorphic to $\frac{ H^1\left(K_\q,D_{E,\kappa_\p^{-1}\kappa_\Cyc^{-1}}\right)}{\bigoplus \limits_{\QQ \mid \q} E^\circ\left((\Kinf)_\QQ\right)\otimes \Q_p/\Z_p }$.

  Consider the short exact sequence
  \begin{align*}
    0 \rightarrow \left(\frac{ H^1\left(K_\q,D_{E,\kappa_\p^{-1}\kappa_\Cyc^{-1}}\right)}{\bigoplus \limits_{\QQ \mid \q} E^\circ\left((\Kinf)_\QQ\right)\otimes \Q_p/\Z_p }\right)^\vee \rightarrow H^1\left(K_\q,D_{E,\kappa_\p^{-1}\kappa_\Cyc^{-1}}\right)^\vee \rightarrow \bigoplus \limits_{\QQ \mid \q} \left(E^\circ\left((\Kinf)_\QQ\right)\otimes \Q_p/\Z_p \right)^\vee \rightarrow 0.
  \end{align*}
  By Proposition \ref{prop:freekim}, the $\Iwalg$-module $\bigoplus \limits_{\QQ \mid \q} \left(E^\circ\left((\Kinf)_\QQ\right)\otimes \Q_p/\Z_p \right)^\vee$ is free of rank one. By Corollary \ref{cor:freeness-four-two}, the $\Iwalg$-module $H^1\left(K_\q,D_{E,\kappa_\p^{-1}\kappa_\Cyc^{-1}}\right)^\vee$ is free of rank two. The proposition now follows.
\end{proof}

\begin{corollary}\label{cor:union}
  Let $\{\One,\Two\}$ be any one of the following unordered pairs
  $$
  \{++,+-\},\quad\{++,-+\},\quad\{--,+-\},\quad\{--,-+\},\quad\{++,\Gr\},\quad\{+-,\Gr\},\quad\{-+,\Gr\},\quad\{--,\Gr\}.
  $$
  Then, the $\Iwalg$-module $\left(\Loc_\One(\Q_p,D_{\rhoDN42}) + \Loc_\Two(\Q_p,D_{\rhoDN42}) \right)^\vee$ is free of rank three.
\end{corollary}
\begin{proof}
  We will prove the corollary when $\{\One,\Two\}$ equals $\{++,+-\}$ and when $\{\One,\Two\}$ equals $\{+-,\Gr\}$. The remaining cases of the corollary would follow similarly.

  If $\{\One,\Two\}$ equals $\{++,+-\}$, then
  \begin{align} \label{eq:sum1}
    \Loc_{++}(\Qp,\Drho) + \Loc_{+-}\left(\Qp,\Drho\right) & = \bigoplus \limits_{\P \mid \p} \left(E^+\left((\Kinf)_\P\right)\otimes \Q_p/\Z_p \right) \oplus H^1\left(K_\q,D_{E,\kappa_\p^{-1}\kappa_\Cyc^{-1}}\right).
  \end{align}
  If $\{\One,\Two\}$ equals $\{+-,\Gr\}$, then
  \begin{align}\label{eq:sum2}
    \Loc_{+-}(\Qp,\Drho) + \Loc_\Gr\left(\Qp,\Drho \right) & = H^1\left(K_\p,D_{E,\kappa_\p^{-1}\kappa_\Cyc^{-1}}\right) \oplus \bigoplus \limits_{\QQ \mid \q} \left(E^-\left((\Kinf)_\QQ\right)\otimes \Q_p/\Z_p \right).
  \end{align}

  By Corollary \ref{cor:freeness-four-two} and Proposition \ref{prop:freekim}, the Pontryagin dual of the $\Iwalg$-modules appearing in equations (\ref{eq:sum1}) and (\ref{eq:sum2}) are free of rank three.
\end{proof}

\begin{proposition} \label{prop:verlf}
  Let $\{\One,\Two\}$ be one of the pairs appearing in the statement of Corollary~\ref{cor:union}. Then, the Pontryagin duals of
  \begin{align*}
    \frac{\Loc_{\Two}\left(\Q_p,D_{\rhodn}\right)}{{\Loc_{\One}\left(\Q_p,D_{\rhodn}\right)\bigcap \Loc_{\Two}\left(\Q_p,D_{\rhodn}\right)}}, \quad \frac{\Loc_{\One}\left(\Q_p,D_{\rhodn}\right)}{{\Loc_{\One}\left(\Q_p,D_{\rhodn}\right)\bigcap \Loc_{\Two}\left(\Q_p,D_{\rhodn}\right)}}
  \end{align*}
  are free $\Iwalg$-modules of rank one.
\end{proposition}
\begin{proof}
  We will only consider the cases when $\{\One,\Two\}$ equals $\{\Gr,+-\}$ and $\{++,+-\}$. The remaining cases follow similarly.

  The identifications in Figure \ref{fig:1} lets us conclude that
  $$\Loc_\Gr(\Qp,\Drho)\cap\Loc_{+-}(\Qp,\Drho)=\bigoplus \limits_{\P \mid \p} \left(E^+\left((\Kinf)_\P\right)\otimes \Q_p/\Z_p\right).$$ Consequently,
  $$\frac{\Loc_\Gr(\Qp,\Drho)}{\Loc_\Gr(\Qp,\Drho)\cap\Loc_{+-}(\Qp,\Drho)}\cong \bigoplus \limits_{\P|\p} \frac{H^1\left((\tilde K_\infty)_\P,E[p^\infty]\right)}{ \left(E^+\left((\Kinf)_\P\right)\otimes \Q_p/\Z_p\right) }.$$
  The Pontryagin dual of this $\Iwalg$-module is free of rank one.

  By Proposition~\ref{prop:freekim}, the Pontryagin dual of the quotient
  $$\frac{\Loc_{+-}(\Qp,\Drho)}{\Loc_\Gr(\Qp,\Drho)\cap\Loc_{+-}(\Qp,\Drho)} \cong \bigoplus_{\QQ \mid \q} \Loc_- \left((\tilde K_\infty)_\QQ,E[p^\infty]\right)$$
  is free of rank one over $\Iwalg$.

  The case when $\{\One,\Two\}$ equals  $\{++,+-\}$ follows from Proposition~\ref{prop:sumquotient} and the isomorphisms
  \[
    \frac{\Loc_{++}(\Qp,D_{\rhoDN42})}{\Loc_{++}(\Q_p,D_{\rhoDN42})\cap\Loc_{+-}(\Q_p,D_{\rhoDN42})}\cong \frac{\Loc_{++}(\Q_p,D_{\rhoDN42})+\Loc_{+-}(\Q_p,D_{\rhoDN42})}{\Loc_{+-}(\Q_p,D_{\rhoDN42})},
  \]

  \[\frac{\Loc_{+-}(\Q_p,D_{\rhoDN42})}{\Loc_{++}(\Q_p,D_{\rhoDN42})\cap\Loc_{+-}(\Q_p,D_{\rhoDN42})}\cong \frac{\Loc_{++}(\Qp,D_{\rhoDN42})+\Loc_{+-}(\Qp,D_{\rhoDN42})}{\Loc_{++}(\Qp,D_{\rhoDN42})}.
  \]
  \end{proof}

\subsection{Description of the modules $\ZZZ(\Q,D_{\rhoDN42})$ and $\ZZZ^\Sstar(\Q,D_{\rhoDNstar42})$} \label{S:pseudonulldesc}

We will only describe the  modules $\ZZZ(\Q,D_{\rhoDN42})$ and $\ZZZ^\Sstar(\Q,D_{\rhoDNstar42})$ when $\{\One,\Two\}$ equals $\{\Gr,+-\}$ and $\{++,+-\}$. The remaining cases follow similarly. Throughout the description, we will keep in mind the  identifications provided by the inclusions and isomorphisms in Figure \ref{fig:1}.

Recall that the module $\ZZZ(\Q,D_{\rhoDN42})$ was defined to be the Pontryagin dual of
\begin{align*}
  \ker\bigg(H^1\left(G_\Sigma, D_{\rhoDN42}\right) \rightarrow   \frac{H^1\left(\Q_p,D_{\rhoDN42}\right)}{\Loc_{\One}\left(\Q_p,D_{\rhoDN42}\right) \bigcap \Loc_{\Two}\left(\Q_p,D_{\rhoDN42}\right)} \times \bigoplus_{l \in \Sigma \setminus \{p\}} H^1\left(\Q_l,D_{\rhoDN42}\right) \bigg).
\end{align*}
To describe $\ZZZ(\Q,D_{\rhoDN42})$, it suffices to identify $\Loc_{\One}\left(\Q_p,D_{\rhoDN42}\right) \bigcap \Loc_{\Two}\left(\Q_p,D_{\rhoDN42}\right)$ in Figure \ref{fig:1}.

When $\{\One,\Two\}$ equals $\{\Gr,+-\}$,
\begin{align*}
  \Loc_{\One}\left(\Q_p,D_{\rhoDN42}\right) \bigcap \Loc_{\Two}\left(\Q_p,D_{\rhoDN42}\right) = \bigoplus \limits_{\P \mid \p} \left(E^+\left((\Kinf)_\P\right)\otimes \Q_p/\Z_p\right)\  \bigoplus  \ 0.
\end{align*}

\begin{proposition} \label{int-description}
  When $\{\One,\Two\}$ equals $\{++,+-\}$,
  \begin{align*}
    \Loc_{\One}\left(\Q_p,D_{\rhoDN42}\right) \bigcap \Loc_{\Two}\left(\Q_p,D_{\rhoDN42}\right) = \bigoplus \limits_{\P \mid \p} \left(E^+\left((\Kinf)_\P\right)\otimes \Q_p/\Z_p\right)\  \bigoplus_{\QQ|\q}\left(\hat E\left(K(\p^\infty)_{\Z_p,\QQ}\right)\otimes\Q_p/\Z_p\right).
  \end{align*}
\end{proposition}

Here, $K(\p^\infty)_{\Z_p,\QQ}$ is the unique unramified $\Z_p$-extension of $K_\q$. Note that we have the following inclusions: $$K_\q \subset K(\p^\infty)_{\Z_p,\QQ} \subset (\widetilde{K}_\infty)_\QQ.$$

\begin{proof} Fix a prime $\QQ$ in $\tilde K_\infty$ above $\q$. First note that there is a short exact sequence
  \begin{align} \label{eq:inter}
    0 \rightarrow \hat E\left(K(\p^\infty)_{\Z_p,\QQ}\right) \rightarrow E^+\left((\tilde K_\infty)_\QQ\right) \oplus E^-\left((\tilde K_\infty)_\QQ\right) \rightarrow \hat E\left((\tilde K_\infty)_\QQ\right) \rightarrow 0.
  \end{align}
  The first map is given by the diagonal map (obtained from the natural inclusions) and the second map is defined by $(a,b)\mapsto a-b$. The exactness at the middle term is given by an argument similar to the one in the proof of  \cite[Proposition~8.12(ii)]{kobayashi}. The exactness at the last term is given in equation  \eqref{eq:plusminushat}. Now, consider the following commutative diagram.
  \begin{align*}
    \xymatrix{
    0 \ar[r] & \hat E\left(K(\p^\infty)_{\Z_p,\QQ}\right) \ar[d] \ar[r]        & E^+\left((\tilde K_\infty)_\QQ\right) \oplus E^-\left((\tilde K_\infty)_\QQ\right) \ar[d] \ar[r]                      & \hat E\left((\tilde K_\infty)_\QQ\right) \ar[d] \ar[r]      & 0 \\
    0 \ar[r] & \hat E\left(K(\p^\infty)_{\Z_p,\QQ}\right) \otimes \Q_p  \ar[r] & \left(E^+\left((\tilde K_\infty)_\QQ\right)   \oplus E^-\left((\tilde K_\infty)_\QQ\right)\right) \otimes\Q_p  \ar[r] & \hat E\left((\tilde K_\infty)_\QQ\right)\otimes \Q_p \ar[r] & 0
    }
  \end{align*}
  Using equation (\ref{eq:inter}), one can conclude that both the rows are exact. All vertical maps turn out to be injective. This is because the kernel of each of the vertical maps is a subgroup of $\hat E\left((\tilde K_\infty)_\QQ\right)[p^\infty]$, which equals zero by Proposition~\ref{prop:residually-irred}. The snake lemma then gives the following short exact sequence:
  \[
    0 \rightarrow \hat E\left(K(\p^\infty)_{\Z_p,\QQ}\right) \otimes \Q_p /\Z_p \rightarrow\left(E^+\left((\tilde K_\infty)_\QQ\right)   \oplus E^-\left((\tilde K_\infty)_\QQ\right)\right) \otimes\Q_p /\Z_p  \rightarrow \hat E\left((\tilde K_\infty)_\QQ\right)\otimes \Q_p/\Z_p \rightarrow0,
  \]
  This lets us conclude that we have the following equality in $\hat{E}((K_\infty)_{\QQ}) \otimes \Q_p/\Z_p$:
  \[
    \left(E^+((\tilde K_\infty)_\QQ)\otimes\Q_p/\Z_p\right)\cap \left(E^-((\tilde K_\infty)_\QQ)\otimes\Q_p/\Z_p\right)
    = \hat E\left(K(\p^\infty)_{\Z_p,\QQ}\right)\otimes \Q_p/\Z_p.
  \]
  The proposition follows.
\end{proof}

To describe $\ZZZ^\Sstar(\Q,D_{\rhoDNstar42})$, we will need to consider the following isomorphism of $\Iwalg$-modules induced by the Weil pairing:
\begin{align} \label{eq:weiliso}
T_{\rhoDNstar42} \cong \Ind_{G_K}^{G_\Q} T_p(E) \otimes_{\Z_p} \Z_p[[\widetilde{\Gamma}]](\kappa).
\end{align}
Let $\iota:\Iwalg \rightarrow \Iwalg$ denote the involution defined by sending every element $\gamma$ of $\tilde{\Gamma}$ to $\gamma^{-1}$. If $M$ is  a $\Iwalg$-module, then we let $M^\iota$ denote the $\Iwalg$-module which is equal to $M$ as a set (and as a $\Z_p$-module) and on which $\gamma \in \widetilde{\Gamma}$ acts as $\iota(\gamma)$. The isomorphism in equation (\ref{eq:weiliso}), in turn, lets us deduce that the discrete $\Iwalg$-module $H^1\left(\Q_p, D_{\rhoDNstar42}\right)$ is isomorphic to $H^1\left(\Q_p, D_{\rhoDN42}\right)^\iota$. For the rest of this section, using this isomorphism induced by the Weil pairing, we will identify $H^1\left(\Q_p, D_{\rhoDNstar42}\right)$  with $H^1\left(\Q_p, D_{\rhoDN42}\right)^\iota$. We have the following pairing given by local duality:
\begin{align}\label{eq:localpairing}
\xymatrix{
H^1_\ct\left(\Q_p, T_{\rhoDN42}\right) \times \underbrace{H^1\left(\Q_p, D_{\rhoDNstar42}\right)}_{\cong H^1\left(\Q_p,D_{\rhoDN42}\right)^\iota} \rightarrow \Q_p/\Z_p
}
\end{align}

The $\Iwalg$-module $\ZZZ^\Sstar(\Q,D_{\rhoDNstar42})$ is defined in \eqref{ZDstar} as the Pontryagin dual of
\[
  \ker\left(H^1\left(G_\Sigma,D_\rhoDNstar42\right) \rightarrow \frac{H^1\left(\Q_p,D_\rhoDNstar42\right)}{\Loc_{\One,\Two}\left(\Q_p,D_\rhoDNstar42\right)}\oplus \bigoplus_{l \in \Sigma \setminus \{p\}} H^1\left(\Q_l,D_\rhoDNstar42\right) \right),
\]
where $
\Loc_{\One,\Two}\left(\Q_p,D_\rhoDNstar42\right)  \subset H^1\left(\Q_p,D_\rhoDNstar42\right)$,
is  the orthogonal complement of $$\left(\left(\Loc_{\One}\left(\Q_p,D_\rhoDN42\right)+\Loc_{\Two}\left(\Q_p,D_\rhoDN42\right)\right)^\vee\right)^*$$
under the pairing given in equation (\ref{eq:localpairing}). We will also keep in mind all the identifications in Figure \ref{fig:1}. \\

If $\{\One,\Two\}$ equals $\{+-,\Gr\}$, then
  \begin{align}\label{eqn:complementdeduction1}
  \notag  \Loc_{+-}(\Qp,\Drho) + \Loc_\Gr\left(\Qp,\Drho \right) &= H^1\left(K_\p,D_{E,\kappa_\p^{-1}\kappa_\Cyc^{-1}}\right) \oplus \bigoplus \limits_{\QQ \mid \q} \left(E^-\left((\Kinf)_\QQ\right)\otimes \Q_p/\Z_p \right), \\
\implies \Loc_{\One,\Two}\left(\Q_p,D_\rhoDNstar42\right) & = \bigoplus_{\P \mid \p} 0  \oplus \left(\bigoplus \limits_{\QQ \mid \q} E^-\left((\Kinf)_\QQ\right)\otimes \Q_p/\Z_p \right)^\iota .
\end{align}

If $\{\One,\Two\}$ equals $\{++,+-\}$, then
  \begin{align}\label{eqn:complementdeduction2}
 \notag   \Loc_{++}(\Qp,\Drho) + \Loc_{+-}\left(\Qp,\Drho\right) & = \bigoplus \limits_{\P \mid \p} \left(E^+\left((\Kinf)_\P\right)\otimes \Q_p/\Z_p \right) \oplus H^1\left(K_\q,D_{E,\kappa_\p^{-1}\kappa_\Cyc^{-1}}\right), \\
\implies \Loc_{\One,\Two}\left(\Q_p,D_\rhoDNstar42\right) &= \left(\bigoplus \limits_{\P \mid \p} E^+\left((\Kinf)_\P\right)\otimes \Q_p/\Z_p \right)^\iota   \oplus  \bigoplus_{\QQ \mid \q} 0 .
\end{align}
Equations (\ref{eqn:complementdeduction1}) and (\ref{eqn:complementdeduction2}) follow from Proposition \ref{prop:identifying}. Note that we have the following perfect pairings:
\begin{align}\label{eq:anotherquadraticpairing1}
H^1_\ct\left(K_\p,T_{E,\widetilde{\kappa}^{-1}}\right) \times H^1\left(K_\p, D_{E,\widetilde{\kappa}^{-1}}\right)^{\iota} &\rightarrow \Q_p/\Z_p
\end{align}
\begin{align}\label{eq:anotherquadraticpairing2}
H^1_\ct\left(K_\p, T_{E,\widetilde{\kappa}^{-1}}\right)^{\iota} \times H^1\left(K_\p,D_{E,\widetilde{\kappa}^{-1}}\right) &\rightarrow \Q_p/\Z_p.
\end{align}

\begin{proposition} \label{prop:identifying} Let $\bullet,\circ \in \{+,-\}$.
\begin{itemize}
\item $\left(\bigoplus \limits_{\P \mid \p} E^\bullet\left((\Kinf)_\P\right)\otimes \Q_p/\Z_p \right)^\iota$ equals the orthogonal complement of $\left(\bigoplus \limits_{\P \mid \p} \left(E^\bullet\left((\Kinf)_\P\right)\otimes \Q_p/\Z_p \right)^\vee\right)^*$ under the pairing given in equation (\ref{eq:anotherquadraticpairing1}).
\item $\left(\bigoplus \limits_{\QQ \mid \q} E^\circ\left((\Kinf)_\QQ\right)\otimes \Q_p/\Z_p \right)^\iota$ equals the orthogonal complement of $\left(\bigoplus \limits_{\QQ \mid \q} \left(E^\circ\left((\Kinf)_\QQ\right)\otimes \Q_p/\Z_p \right)^\vee\right)^*$, under the pairing for the field $K_\q$ similar to the one given in equation (\ref{eq:anotherquadraticpairing1}).
\end{itemize}
\end{proposition}
Here, $T_{E,\widetilde{\kappa}^{-1}}$ denotes $\bigg(T_p(E) \otimes_{\Z_p} \Z_p[[\widetilde{\Gamma}]](\widetilde{\kappa}^{-1})  \bigg)$ and  $D_{E,\widetilde{\kappa}^{-1}}$ denotes $T_{E,\widetilde{\kappa}^{-1}} \otimes_{\Z_p[[\widetilde{\Gamma}]]} \Z_p[[\widetilde{\Gamma}]]^\vee$.

\begin{proof}

We will show that $\left(\bigoplus \limits_{\P \mid \p} E^+\left((\Kinf)_\P\right)\otimes \Q_p/\Z_p \right)^\iota$, under the pairing given in equation (\ref{eq:anotherquadraticpairing1}), equals the orthogonal complement of $\left(\bigoplus \limits_{\P \mid \p} \left(E^+\left((\Kinf)_\P\right)\otimes \Q_p/\Z_p \right)^\vee\right)^*$. The rest of the proposition would follow similarly. Let $$H_+ \subset H^1_\ct\left(K_\p,T_{E,\widetilde{\kappa}^{-1}}\right), \qquad H_+^\iota \subset H^1_\ct\left(K_\p,T_{E,\widetilde{\kappa}^{-1}}\right)^\iota,$$ respectively denote the orthogonal complements of $$\left(\bigoplus \limits_{\P \mid \p} E^+\left((\Kinf)_\P\right)\otimes \Q_p/\Z_p\right)^\iota, \qquad \bigoplus \limits_{\P \mid \p} E^+\left((\Kinf)_\P\right)\otimes \Q_p/\Z_p$$  under the pairings given in equation (\ref{eq:anotherquadraticpairing1}) and equation (\ref{eq:anotherquadraticpairing2}) respectively. We will keep the following isomorphisms in mind:
\begin{align*}
\widetilde{\Gamma} \cong \varprojlim \limits_\alpha \widetilde{\Gamma}_\alpha, \qquad \Lambda \cong \varprojlim_{m,\alpha} \frac{\Z_p}{p^m\Z_p}[\widetilde{\Gamma}_\alpha].
\end{align*}
Here, $m$ varies over all the positive integers and $\widetilde{\Gamma}_\alpha$ varies over all finite quotients of $\widetilde{\Gamma}$. By abuse of notation, we let $\iota: \frac{\Z_p}{p^m\Z_p}[\widetilde{\Gamma}_\alpha] \rightarrow \frac{\Z_p}{p^m\Z_p}[\widetilde{\Gamma}_\alpha]$ also denote the $\Z_p$-linear involution obtained by sending every element $\gamma \in  \widetilde{\Gamma}_\alpha$ to $\gamma^{-1}$.

To identify $\left(\bigoplus \limits_{\P \mid \p} \left(E^+\left((\Kinf)_\P\right)\otimes \Q_p/\Z_p \right)^\vee\right)^*$ inside $H^1_\ct\left(K_\p, T_{E,\widetilde{\kappa}^{-1}}\right)$, it will be helpful to expand on the isomorphism $\Hom_\Lambda\bigg(H^1_\ct\left(K_\p, T_{E,\widetilde{\kappa}^{-1}}\right)^{\iota}, \Lambda\bigg) \cong H^1_\ct\left(K_\p, T_{E,\widetilde{\kappa}^{-1}}\right)$, which we obtain by combining equations (\ref{iso-local-tate-duality}) and  (\ref{eq:anotherquadraticpairing2}), using the commutative diagram below\footnote{For a more explicit description of this identification, see works of Perrin-Riou \cite[\S3.6.1]{PR94} along with Loeffler and Zerbes \cite[Definition~2.3]{LZZp2}.}.

\begin{tikzcd}
\Hom_\Lambda\bigg(H^1_\ct\left(K_\p, T_{E,\widetilde{\kappa}^{-1}}\right)^{\iota}, \Lambda\bigg) \arrow[d,"\cong"] \ar[r, "\cong"] & H^1_\ct\left(K_\p, T_{E,\widetilde{\kappa}^{-1}}\right) \ar[d, "\cong"] \\
\varprojlim \limits_{m,\alpha} \Hom_{\frac{\Z_p}{p^m\Z_p}[[\widetilde{\Gamma}_\alpha]]}\bigg(H^1\left(K_\p, T_{E,\widetilde{\kappa}^{-1}} \otimes_{\Lambda} \frac{\Z_p}{p^m\Z_p}[[\widetilde{\Gamma}_\alpha]] \right)^{\iota},\frac{\Z_p}{p^m\Z_p}[[\widetilde{\Gamma}_\alpha]]\bigg) \arrow[d,"\cong", "\substack{\text{Frobenius} \\ \text{reciprocity}}"'] \arrow[r,"\cong"]  & \varprojlim \limits_{m,\alpha} H^1\left(K_\p,T_{E,\widetilde{\kappa}^{-1}} \otimes_{\Lambda} \frac{\Z_p}{p^m\Z_p}[[\widetilde{\Gamma}_\alpha]]\right) \arrow[d,"\cong"] \\
\varprojlim \limits_{m,\alpha} \Hom_{\frac{\Z_p}{p^m\Z_p}}\bigg(H^1\left(K_\p, T_{E,\widetilde{\kappa}^{-1}} \otimes_{\Lambda} \frac{\Z_p}{p^m\Z_p}[[\widetilde{\Gamma}_\alpha]] \right)^{\iota},\frac{\Z_p}{p^m\Z_p}\bigg) \arrow[r, "\cong"] & \varprojlim \limits_{m,\alpha} H^1\left(K_\p,T_{E,\widetilde{\kappa}^{-1}} \otimes_{\Lambda} \frac{\Z_p}{p^m\Z_p}[[\widetilde{\Gamma}_\alpha]]\right)
\end{tikzcd}
 \captionof{figure}{} \label{fig:2}

The isomorphism in the second row of Figure \ref{fig:2}, in turn, relies on the isomorphism $H^1_\ct\left(K_\p, T_{E,\widetilde{\kappa}^{-1}}\right) \otimes_\Lambda \frac{\Z_p}{p^m\Z_p}[[\widetilde{\Gamma}_\alpha]] \cong H^1\left(K_\p, T_{E,\widetilde{\kappa}^{-1}} \otimes_\Lambda \frac{\Z_p}{p^m\Z_p}[[\widetilde{\Gamma}_\alpha]]\right)$, where \ref{locp0}  comes in to play. The isomorphism in the last row of Figure \ref{fig:2} is given by local duality as indicated in the commutative diagram below.

\begin{tikzpicture}[auto]
    \node (H1iota) {$H^1\left(K_\p,T_{E,\widetilde{\kappa}^{-1}} \otimes_{\Lambda} \frac{\Z_p}{p^m\Z_p}[[\widetilde{\Gamma}_\alpha]]\right)
^{\iota} \times $};

    \node (H1) [right of=H1iota, node distance = 6cm]{$H^1\left(K_\p,T_{E,\widetilde{\kappa}^{-1}} \otimes_{\Lambda} \frac{\Z_p}{p^m\Z_p}[[\widetilde{\Gamma}_\alpha]]\right)$};

    \node(zpm) [right of =H1, node distance = 6cm] {$\frac{\Z_p}{p^m\Z_p}$};

    \node(H1dual)[below of = H1iota, node distance = 1.7cm]{$H^1\left(K_\p,\Hom_{\Z_p}\left(T_{E,\widetilde{\kappa}^{-1}} \otimes_{\Lambda} \frac{\Z_p}{p^m\Z_p}[[\widetilde{\Gamma}_\alpha]]\right), \mu_{p^m}
\right) \times $};

    \node(H1again) [right of =H1dual, node distance = 6.8cm] {$H^1\left(K_\p,T_{E,\widetilde{\kappa}^{-1}} \otimes_{\Lambda} \frac{\Z_p}{p^m\Z_p}[[\widetilde{\Gamma}_\alpha]]\right)$};

    \node(H2brauer)[below of = zpm, node distance = 1.7cm]{$H^2\left(K_\p,\mu_{p^m}\right)$};

    \draw[->] (H1iota) to node  {$\cong$} node[pos=.5,left] {$\substack{\text{Weil} \\ \text{pairing}}$}  (H1dual);
     \draw[->] (H1) to node  {} (zpm);
     \draw[->] (zpm) to node  {$\cong$}  (H2brauer);
\draw[->] (H1again) to node  {$\bigcup$} (H2brauer);
  \end{tikzpicture}
 \captionof{figure}{} \label{fig:3}

To prove the proposition, we need to show that under the identification given in Figure \ref{fig:2}, we have
\begin{align}\label{eq:neededequality}
H_+ \stackrel{?}{=} \left(\bigoplus \limits_{\P \mid \p} \left(E^+\left((\Kinf)_\P\right)\otimes \Q_p/\Z_p \right)^\vee\right)^*.
\end{align}

Observe that we have a short exact sequence of $\Lambda$-modules:
\begin{align*}
0 \rightarrow H^\iota_+ \rightarrow H^1_\ct\left(K_\p, T_{E,\widetilde{\kappa}^{-1}} \right)^\iota \rightarrow \bigoplus \limits_{\P \mid \p} \left(E^+\left((\Kinf)_\P\right)\otimes \Q_p/\Z_p \right)^\vee \rightarrow 0.
\end{align*}
This allows us to describe $\left(\bigoplus \limits_{\P \mid \p} \left(E^+\left((\Kinf)_\P\right)\otimes \Q_p/\Z_p \right)^\vee\right)^*$ as the following:
\begin{align} \label{eq:desc_hom}
\left\{\phi \in \Hom_\Lambda\bigg(H^1_\ct\left(K_\p, T_{E,\widetilde{\kappa}^{-1}}\right)^{\iota}, \Lambda\bigg) \text{ such that } \phi(x) =0 \text{ for all $x$ in } H^\iota_+\right\}.
\end{align}

For each $y \in H^1_\ct\left(K_\p, T_{E,\widetilde{\kappa}^{-1}}\right)$, let $\phi_y$ denote the element in $\Hom_\Lambda\left(H^1_\ct\left(K_\p, T_{E,\widetilde{\kappa}^{-1}}\right)^\iota, \Lambda \right)$ corresponding to isomorphism in Figure \ref{fig:2}. Disregarding the involution $\iota$, the pairing in the first row of Figure \ref{fig:3} is skew-symmetric since, as indicated by the pairing in the second row of Figure \ref{fig:3}, it is induced by the cup-product on the first Galois cohomology groups. This lets us conclude that if $y \in H_+$, then $\phi_y(x^\iota)$ equals zero\footnote{Note however that if  $x,y$  are two elements in $H^1_\ct\left(K_\p, T_{E,\widetilde{\kappa}^{-1}}\right)$, then in general $\phi_y(x^\iota)$ need not equal $-\phi_{x}(y^\iota)$. That is, the pairing $H^1_\ct\left(K_\p, T_{E,\widetilde{\kappa}^{-1}}\right)^\iota \times H^1_\ct\left(K_\p, T_{E,\widetilde{\kappa}^{-1}}\right) \rightarrow \Lambda$,  described using the identification in Figure \ref{fig:2}, need not be skew-symmetric (disregarding the involution $\iota$).} for all $x^\iota$ in $H^{\iota}_+$. As a result, using the description in equation (\ref{eq:desc_hom}), we have the following inclusion of $\Lambda$-modules inside $H^1_\ct\left(K_\p,T_{E,\widetilde{\kappa}^{-1}}\right)$:
\begin{align} \label{eq:dualityinclusion}
H_+ \hookrightarrow \left(\bigoplus \limits_{\P \mid \p} \left(E^+\left((\Kinf)_\P\right)\otimes \Q_p/\Z_p \right)^\vee\right)^*
\end{align}
The $\Lambda$-modules $H_+$ and $\left(\bigoplus \limits_{\P \mid \p} \left(E^+\left((\Kinf)_\P\right)\otimes \Q_p/\Z_p \right)^\vee\right)^*$ are free of rank one and are direct summands of the $\Lambda$-module $H^1_\ct\left(K_\p,T_{E,\widetilde{\kappa}^{-1}}\right)$. As a result, the inclusion in equation (\ref{eq:dualityinclusion}) must be an equality. This lets us assert the validity of  equation (\ref{eq:neededequality}). The Proposition follows.\end{proof}

{\begin{remark} \label{synopsis-match}

\mbox{}
The following table summarizes our discussion for the local Selmer conditions at $p$ for $\ZZZ(\Q,D_{\rhoDN42})$  and $\ZZZ^\Sstar(\Q,D_{\rhoDNstar42})$ associated to the pairs $\{+-,\Gr\}$ and $\{++,+-\}$.

\begin{center}
  {\renewcommand{\arraystretch}{1.2}
\captionof{table}{Summary of local Selmer conditions at $p$}\label{table:descZZstar}
\begin{tabular}{rccccc}\toprule
& \multicolumn{2}{c}{$\{+-,\Gr\}$} & \phantom{abc}& \multicolumn{2}{c}{$\{++,+-\}$}  \\
\cmidrule{2-3} \cmidrule{5-6}
&   Loc. cond. at $\p$     & Loc. cond. at    $\q$   && Loc. cond. at  $\p$   & Loc. cond. at    $\q$  \\ \midrule
$\ZZZ(\Q,D_{\rhoDN42})$ &  Plus  & Empty   && Plus & See Prop. \ref{int-description}   \\
$\ZZZ^\Sstar(\Q,D_{\rhoDNstar42})$
& Empty & Minus   &&  Plus & Empty    \\
\bottomrule
\end{tabular}
}
\end{center}

The discussion of the local Selmer conditions at $p$ for the remaining pairs is very similar. For the pairs $\{++,\Gr\}, \{+-,\Gr\}, \{-+,\Gr\}, \{--,\Gr\}$, the description of the module $\ZZZ^\Sstar(\Q,D_{\rhoDNstar42})$ associated to the Tate dual $\rhoDNstar42$ is analogous to the description of $\ZZZ(\Q,D_{\rhoDN42})$ associated to $\rhoDN42$ and matches the description given in equation (\ref{ZD}). For the remaining pairs however, the description of the module $\ZZZ^\Sstar(\Q,D_{\rhoDNstar42})$ for $\rhoDNstar42$  does not match the description as given in equation (\ref{ZD}).
\end{remark}
}

\subsection{Local fudge factors at primes $l \neq p$ when $p \geq 5$} \label{section:fudge}

Throughout Section \ref{section:fudge}, we will assume that $p \geq 5$ and that $l \neq p$ is a prime number. Consider the Galois representation $\rho_{E,\widetilde{\kappa}^{-1}}:G_K \rightarrow \Gl_2\left(\Iwalg\right)$ given by the action of $G_K$ on the free rank two $\Iwalg$-module $T_{E,\widetilde{\kappa}^{-1}}:=\bigg(T_p(E) \otimes_{\Z_p} \Z_p[[\widetilde{\Gamma}]](\widetilde{\kappa}^{-1})  \bigg)$.  The discrete $\Iwalg$-module associated to $\rho_{E,\widetilde{\kappa}^{-1}}$ is  $D_{E,\widetilde{\kappa}^{-1}} := T_{E,\widetilde{\kappa}^{-1}} \otimes_{\Z_p[[\widetilde{\Gamma}]]} \Z_p[[\widetilde{\Gamma}]]^\vee$. We have the following isomorphism of Galois representations over $\Iwalg$:
\begin{align*}
  \rhoDN42 \cong \Ind^\Q_K\left(\rho_{E,\widetilde{\kappa}^{-1}}\right).
\end{align*}
Let $l \neq p$ be a prime number. If the prime number $l$ splits into two primes $\eta_1$ and $\eta_2$ in $K$, then we have the following isomorphism of discrete $\Iwalg$-modules:
\begin{align*}
  H^0\left(\Q_l, D_{\rhoDN42}\right) \cong H^0\left(K_{\eta_1},D_{E,\widetilde{\kappa}^{-1}}\right) \oplus H^0\left(K_{\eta_2},D_{E,\widetilde{\kappa}^{-1}}\right).
\end{align*}

If there exists a unique prime $\eta$ in $K$ lying above $l$, then we have the following isomorphism of $\Iwalg$-modules:
\begin{align*}
  H^0\left(\Q_l, D_{\rhoDN42}\right) \cong H^0\left(K_{\eta},D_{E,\widetilde{\kappa}^{-1}}\right).
\end{align*}
Here, $K_\eta$ denotes the completion of the imaginary quadratic field $K$ at the prime $\eta$. Combining these observations, we have the following equality in $Z^2\left(\Iwalg\right)$:
\begin{align}
  \sum_{l \in \Sigma \setminus \{p\}} c_2\bigg(\left(H^0\left(\Q_l, D_{\rhoDN42}\right)^\vee\right)_{\pn}\bigg) = \sum_{l \in \Sigma \setminus \{p\}} \sum_{\substack{\eta \mid l \\ \eta \text{ in } K}} c_2\bigg(\left(H^0\left(K_\eta, D_{E,\widetilde{\kappa}^{-1}}\right)^\vee\right)_{\pn}\bigg)
\end{align}
The calculation of the invariant $c_2\bigg(\left(H^0\left(K_\eta, D_{E,\widetilde{\kappa}^{-1}}\right)^\vee\right)_{\pn}\bigg)$ depends on the reduction type of the elliptic curve $E$ at the prime $\eta$ in $K$. See Propositions \ref{fudge-good}, \ref{fudge-additive}, \ref{fudge-nonsplit} and \ref{fudge-split} below. \\

Let $K_{\eta}^{ur}$ denote the maximal unramified extension of $K_{\eta}$. Let $I_{\eta,ur} :=\Gal{\overline{\Q}_l}{K_{\eta,\ur}}$ denote the corresponding inertia group. The Galois group $\Gal{K^{ur}_\eta}{K_\eta}$, which is isomorphic to $\hat{\Z}$, is topologically generated by the Frobenius element $\Frob_\eta$. We make the following observations:
\begin{enumerate}[style=sameline, align=left, label=(\thesection\roman*), ref=(\thesection\roman*)]
  \item\label{ob1} The restriction of the character $\widetilde{\kappa}$ to the inertia group $I_{\eta,ur}$ is trivial.
  \item\label{ob2} Since the prime $\eta$ does not split completely in the cyclotomic $\Z_p$-extension $K_{\Cyc}/K$, the image of the character $\widetilde{\kappa}$ lies inside $\Z_p[[\widetilde{\Gamma}]]$ but not inside $\Z_p$.
  \item\label{ob3} The residual representation, associated to the character $\widetilde{\kappa}:G_K \twoheadrightarrow \widetilde{\Gamma} \hookrightarrow \Gl_1(\Iwalg)$, is trivial.
  \item\label{ob4} Let $\chi_p:\Gal{\overline{\Q}_l}{K_\eta} \rightarrow \Z_p^\times$ denote the $p$-adic cyclotomic character given by the action of $\Gal{\overline{\Q}_l}{K_\eta}$ on $\mu_{p^\infty}$. The restriction of the $p$-adic cyclotomic character $\chi_p$ to the inertia group $I_{\eta,ur}$ is trivial.
\end{enumerate}

\subsubsection{$E$ has good reduction at $\eta$}

\begin{proposition} \label{fudge-good}
  Suppose the elliptic curve $E$ has good reduction at the prime $\eta$. Then,
  \begin{align*}
    \pd_{\Iwalg}  H^0\left(\Gal{\overline{\Q}_l}{K_\eta},D_{E,\widetilde{\kappa}^{-1}}\right)^\vee \leq 1.
  \end{align*}
  Consequently, we have
  \begin{align*}
    \bigg(H^0\left(\Gal{\overline{\Q}_l}{K_\eta},D_{E,\widetilde{\kappa}^{-1}}\right)^\vee\bigg)_{\pn}=0.
  \end{align*}
\end{proposition}

\begin{proof}
  Since $\eta$ is a prime not lying above $p$ and since the elliptic curve $E$ has good reduction at the prime $\eta$, we can conclude that the action of the inertia group $I_{\eta,ur}$ on $D_{E,\widetilde{\kappa}^{-1}}$ is trivial. As a result, we have
  \begin{align*}
    H^0\left(\Gal{\overline{\Q}_l}{K_\eta},D_{E,\widetilde{\kappa}^{-1}}\right) = H^0\left(\Gal{K^{ur}_{\eta}}{K_\eta}, D_{E,\widetilde{\kappa}^{-1}}\right).
  \end{align*}
  The eigenvalues of $\Frob_\eta$, say $a$ and $b$, on the $p$-adic Tate module $T_p(E)$ are distinct (see Theorem 4.1 in Coleman-Edixhoven's work \cite{MR1600034}). Hence, the action of $\Frob_\eta$ on the free $\Iwalg$-module $\left(D_{E,\widetilde{\kappa}^{-1}}\right)^\vee$ is semi-simple. Let $O$ denote the ring of integers in a finite extension of $\Q_p$, containing the eigenvalues $a$ and $b$. This lets us deduce the following isomorphisms of $O[[\widetilde{\Gamma}]]$-modules:
  \begin{align*}
    H^0\left(\Gal{K^{ur}_{\eta}}{K_\eta}, D_{E,\widetilde{\kappa}^{-1}}\right)^\vee \otimes_{\Z_p}O & \cong \coker\left(O[[\widetilde{\Gamma}]]^2 \xrightarrow{\left[\begin{array}{cc} a^{-1}\widetilde{\kappa}(\Frob_\eta)-1 & 0 \\ 0 & b^{-1}\widetilde{\kappa}(\Frob_\eta) -1 \end{array}\right]} O[[\widetilde{\Gamma}]]^2\right), \\ &  \cong \left(\frac{O[[\widetilde{\Gamma}]]}{(a^{-1}\widetilde{\kappa}(\Frob_\eta)-1)}\right)  \oplus \left(\frac{O[[\widetilde{\Gamma}]]}{(b^{-1}\widetilde{\kappa}(\Frob_\eta)-1)}\right).
  \end{align*}
  As a result, we have
  \begin{align*}
    \pd_{O[[\widetilde{\Gamma}]]} H^0\left(\Gal{K^{ur}_{\eta}}{K_\eta}, D_{E,\widetilde{\kappa}^{-1}}\right)^\vee \otimes_{\Z_p}O \leq 1.
  \end{align*}
  These observations, along with Corollary 3.2.10 in \cite{weibel1995introduction} and the fact that the extension $\Z_p \rightarrow O$ is faithfully flat, let us deduce that
  \begin{align*}
    \pd_{\Iwalg}  H^0\left(\Gal{\overline{\Q}_l}{K_\eta},D_{E,\widetilde{\kappa}^{-1}}\right)^\vee \leq 1
  \end{align*}
  and the proposition follows.
\end{proof}

\subsubsection{$E$ has additive reduction at $\eta$}

\begin{proposition} \label{fudge-additive}
  Suppose the elliptic curve $E$ has additive reduction at the prime $\eta$. Then, we have
  \begin{align*}
    H^0\left(\Gal{\overline{\Q}_l}{K_\eta},D_{E,\widetilde{\kappa}^{-1}}\right)^\vee=0.
  \end{align*}
\end{proposition}

\begin{proof}
  Let $\m$ denote the maximal ideal of $\Z_p[[\widetilde{\Gamma}]]$. Note that  we have the following isomorphism of $\mathbb{F}_p[\Gal{\overline{\Q}_l}{K_\eta}]$-modules (see Observation \ref{ob3}):
  \begin{align*}
    D_{E,\widetilde{\kappa}^{-1}}[\m] \cong E[p]
  \end{align*}
  We now claim that the trivial character is not a component of $E[p]^{s.s.}$, the semi-simplification of the $\mathbb{F}_p[I_{\eta,ur}]$-representation $E[p]$. As a result, we would have
  \begin{align}\label{inertia-semi=simplification}
    H^0\left(I_{\eta,ur}, \ E[p]^{s.s.}\right) \stackrel{?}{=}0.
  \end{align}
  The proposition would follow from this claim (see Corollary 3.1.1 in Greenberg's work \cite{MR2290593}). We consider two cases.\\

  \noindent
  \underline{Case 1: $E$ has potentially good reduction at $\eta$.}  The action of $I_{\eta,ur}$ factors through a non-trivial finite group $\Delta$, such that the only prime factors\footnote{When $l \neq 3$, then the elliptic curve $E$ attains good reduction over the field $K_\eta(E[3])$. See Proposition 10.3(b) in Chapter 10 of Silverman's book \cite{MR1312368}. Note that $|\Gal{K_\eta(E[3])}{K_\eta}|$ divides $|\Gl_2(\mathbb{F}_3)|$ (which equals $48$). \\ When $l =3$, then $E$ attains good reduction over the fields $K_\eta(E[5])$ and $K_\eta(E[7])$. The ramification degrees of the field extensions $K_\eta(E[5])/K_\eta$ and $K_\eta(E[7])/K_\eta$ are equal and hence must divide $|\Gal{K_\eta(E[5])}{K_\eta}|$ (which divides  $|\Gl_2(\mathbb{F}_5)|=2^5*3*5$) and $|\Gal{K_\eta(E[7])}{K_\eta}|$ (which divides $|\Gl_2(\mathbb{F}_7)|=2^5*3^2*7$). See Problem 7.9(a) in \cite{MR817210}.} dividing $\Delta$ are in the set $\{2,3\}$. Choose an element $\varsigma$ in $\Delta$ that acts non-trivially on $T_p(E)$. The action of $\varsigma$ on $T_p(E)$ is via a $2 \times 2$ matrix, say $B$, with values in $\Z_p$. Since $\det(B)$ must equal $1$ (see Observation \ref{ob4}), the two eigenvalues of $B$ are of the form $b$ and $b^{-1}$.  Both eigenvalues must be non-trivial roots of unity in $\overline{\Q}_p$, of order prime to $p$ (since $p \geq 5$). As a result, we have $b \not\equiv 1 \mod p$ and $b^{-1} \not\equiv 1 \mod p$. Let $\overline{B}$ denote the $2 \times 2$ matrix with values in $\mathbb{F}_p$ which gives us the action of $\varsigma$ on $E[p]$. Our observations let us conclude that the eigenvalues of $\overline{B}$, which are $\overline{b}$ and $\overline{b}^{-1}$, are both not equal to one. This lets us establish the validity of equation (\ref{inertia-semi=simplification}) in this case.

  \noindent
  \underline{Case 2: $E$ has potentially split multiplicative reduction at $\eta$.} In this case, we have the following short exact sequence of $\Z_p$-modules, that is $\Gal{\overline{\Q}_l}{K_\eta}$-equivariant:
  \begin{align*}
    0 \rightarrow \Z_p(\chi_p\chi) \rightarrow T_p(E) \rightarrow \Z_p(\chi) \rightarrow 0.
  \end{align*}
  Here, $\chi$ is a quadratic ramified character of $\Gal{\overline{\Q}_l}{K_\eta}$. As a result, we have the following isomorphism of $\mathbb{F}_p[I_{\eta,ur}]$-modules:
  \begin{align*}
    E[p]^{s.s.} \cong \mathbb{F}_p(\chi) \oplus \mathbb{F}_p(\chi).
  \end{align*}
  Since the restriction of $\chi$ to the inertia group $I_{\eta,ur}$ is a non-trivial quadratic character, equation (\ref{inertia-semi=simplification}) follows in this case too.
\end{proof}

\subsubsection{$E$ has non-split multplicative reduction at $\eta$}

\begin{proposition} \label{fudge-nonsplit}
  Suppose the elliptic curve $E$ has non-split multiplicative reduction at the prime $\eta$. Then, we have
  \begin{align*}
    \pd_{\Iwalg}  H^0\left(\Gal{\overline{\Q}_l}{K_\eta},D_{E,\widetilde{\kappa}^{-1}}\right)^\vee \leq 1.
  \end{align*}
  Consequently, we have
  \begin{align*}
    \bigg(H^0\left(\Gal{\overline{\Q}_l}{K_\eta},D_{E,\widetilde{\kappa}^{-1}}\right)^\vee\bigg)_{\pn}=0.
  \end{align*}
\end{proposition}

\begin{proof}
  Since $E$ has non-split multiplicative reduction at $\eta$, we have the following short exact sequence of $\Z_p$-modules that is $\Gal{\overline{\Q}_l}{K_\eta}$-equivariant:
  \begin{align*}
    0 \rightarrow \Z_p(\chi_p\chi) \rightarrow T_p(E) \rightarrow \Z_p(\chi) \rightarrow 0.
  \end{align*}
  Here, $\chi$ is an unramified quadratic character. Taking the tensor product $\otimes_{\Z_p} \Z_p[[\widetilde{\Gamma}]]^\vee(\widetilde{\kappa}^{-1})$, we obtain the following short exact sequence of $\Iwalg$-modules, that is $\Gal{\overline{\Q}_l}{K_\eta}$-equivariant:
  \begin{align} \label{ses-nonsplit-tate}
    0 \rightarrow \Z_p[[\widetilde{\Gamma}]]^\vee(\chi_p\chi\widetilde{\kappa}^{-1}) \rightarrow D_{E,\widetilde{\kappa}^{-1}} \rightarrow \Z_p[[\widetilde{\Gamma}]]^\vee(\chi\widetilde{\kappa}^{-1}) \rightarrow 0.
  \end{align}
  The residual representation associated to the character $\chi\widetilde{\kappa}^{-1}$ coincides with the non-trivial quadratic character $\chi$. As a result, $H^0\left(\Gal{\overline{\Q}_l}{K_\eta},\Z_p[[\widetilde{\Gamma}]]^\vee(\chi\widetilde{\kappa}^{-1})\right)$ equals $0$. Also note that the restriction of the character $\chi_p\chi\widetilde{\kappa}^{-1}$  to the inertia group $I_{\eta,ur}$ is trivial. Combining these observations and considering the $\Gal{\overline{\Q}_l}{K_\eta}$-invariants of the modules given in the short exact sequence (\ref{ses-nonsplit-tate}) let us deduce the following iosmorphism of $\Iwalg$-modules:
  \begin{align*}
    H^0\left(\Gal{\overline{\Q}_l}{K_\eta}, \ D_{E,\widetilde{\kappa}^{-1}}\right) & \cong H^0\left(\Gal{\overline{\Q}_l}{K_\eta},  \ \Z_p[[\widetilde{\Gamma}]]^\vee(\chi_p\chi\widetilde{\kappa}^{-1}) \right), \\ & \cong H^0\left(\Gal{K_{\eta,ur}}{K_\eta},  \ \Z_p[[\widetilde{\Gamma}]]^\vee(\chi_p\chi\widetilde{\kappa}^{-1}) \right).
  \end{align*}
  As a result, we have
  \begin{align*}
    H^0\left(\Gal{\overline{\Q}_l}{K_\eta},D_{E,\widetilde{\kappa}^{-1}}\right)^\vee \cong \frac{\Z_p[[\widetilde{\Gamma}]]}{\left(\chi_p\chi\widetilde{\kappa}^{-1}(\Frob_\eta)-1\right)}.
  \end{align*}
  The proposition follows.
\end{proof}

\subsubsection{$E$ has split multplicative reduction at $\eta$}

\begin{proposition}\label{fudge-split}
  Suppose the elliptic curve $E$ has split multiplicative reduction at the prime $\eta$.
  \begin{align*}
    H^0\left(\Gal{\overline{\Q}_l}{K_\eta},D_{E,\widetilde{\kappa}^{-1}}\right)^\vee \cong \frac{\Z_p[[\widetilde{\Gamma}]]}{\left(\chi_p\widetilde{\kappa}^{-1}(\Frob_\eta)-1\right)} \oplus \frac{\Iwalg}{\left(\mathrm{ord}_{\pi_\eta}(q_{E,\eta}), \ \widetilde{\kappa}^{-1}(\Frob_\eta)-1 \right)}.
  \end{align*}
  Consequently,
  \begin{align*}
    \bigg(H^0\left(\Gal{\overline{\Q}_l}{K_\eta},D_{E,\widetilde{\kappa}^{-1}}\right)^\vee\bigg)_{\pn} \cong \frac{\Iwalg}{\left(\mathrm{ord}_{\pi_\eta}(q_{E,\eta}), \ \widetilde{\kappa}^{-1}(\Frob_\eta)-1 \right)}.
  \end{align*}
\end{proposition}

\begin{remark} Before proving Proposition \ref{fudge-split}, we make the following observations:
  \begin{itemize}
    \item If  $K_\eta/ \Q_l$ is an unramified quadratic extension, then $\chi_p(\Frob_\eta) = l^2$. Otherwise, $\chi_p(\Frob_\eta)~=~l$.
    \item $p$ does not divide the element $\widetilde{\kappa}^{-1}(\Frob_\eta)-1$.
    \item Proposition \ref{fudge-split} lets us conclude that if the elliptic curve $E$ has split multplicative reduction at $\eta$, then the invariant $c_2\left(\bigg(H^0\left(\Gal{\overline{\Q}_l}{K_\eta},D_{E,\widetilde{\kappa}^{-1}}\right)^\vee\bigg)_{\pn}\right)$ is non-trivial if and only if $p$ divides $\mathrm{ord}_{\pi_\eta}(q_{E,\eta})$.
  \end{itemize}
\end{remark}

\begin{proof}[Proof of Proposition \ref{fudge-split}]
  Since $E$ has split multiplicative reduction at $\eta$, there exists an element $q_{E,\eta} \in K_\eta^\times$ (usually called the Tate parameter) that gives us the following isomorphism $E(\overline{\Q}_l) \cong \frac{\overline{\Q}_l^\times}{q_{E,\eta}}$. We have the following short exact sequence of $\Z_p$-modules that is $\Gal{\overline{\Q}_l}{K_\eta}$-equivariant:
  \begin{align} \label{eq:ses-split}
    0 \rightarrow \Z_p(\chi_p) \rightarrow T_p(E) \rightarrow \Z_p \rightarrow 0.
  \end{align}
  The action of $\Gal{\overline{\Q}_l}{K_\eta}$ on $T_p(E)$ factors through $\Gal{K^{ur}_{\eta}(q_{E,\eta}^{\frac{1}{p^\infty}})}{K_{\eta}}$, which fits into the following short exact sequence:
  \begin{align*}
    1 \rightarrow   \underbrace{\Gal{K^{ur}_{\eta}(q_{E,\eta}^{\frac{1}{p^\infty}})}{K^{ur}_{\eta}}}_{\cong \Z_p} \rightarrow \Gal{K^{ur}_{\eta}(q_{E,\eta}^{\frac{1}{p^\infty}})}{K_{\eta}}\rightarrow \underbrace{\Gal{K^{ur}_{\eta}}{K_\eta}}_{\cong \hat{\Z}} \rightarrow 1.
  \end{align*}
  Let $t_p$ denote a topological generator of $\Gal{K^{ur}_{\eta}(q_{E,\eta}^{\frac{1}{p^\infty}})}{K^{ur}_{\eta}}$. The action of $t_p$ on $T_p(E)$ can be described by a $2 \times 2$ matrix $\left[\begin{array}{cc} 1 & a_{t_p} \\ 0 & 1\end{array} \right]$ with values in $\Z_p$, where $a_{t_p}$ is an element in $\Z_p$ satisfying
  \begin{align*}
    a_{t_p} = \mathrm{ord}_{\pi_\eta}(q_{E,\eta})u, \text{ for some element $u \in \Z_p^\times$}.
  \end{align*}
  One can choose a lift $\widetilde{\Frob}_\eta$ of $\Frob_\eta$ in $\Gal{K^{ur}_{\eta}(q_{E,\eta}^{\frac{1}{p^\infty}})}{K_{\eta}}$ so that the action of $\widetilde{\Frob}_\eta$ on $T_p(E)$ is given by the  $2 \times 2$ matrix $\left[\begin{array}{cc} \chi_p(\Frob_\eta)  & 0 \\ 0 & 1 \end{array} \right]$.

The actions of $t_p$ and $\widetilde{\Frob}_\eta$ on $D_{E,\widetilde{\kappa}^{-1}}$ are given by the following $2 \times 2$ matrices with values in $\Iwalg$:
  \begin{align*}
    t_p \rightarrow \left[\begin{array}{cc} 1 & a_{t_p} \\ 0 & 1\end{array} \right], \qquad \widetilde{\Frob}_\eta \rightarrow \left[\begin{array}{cc} \chi_p\widetilde{\kappa}^{-1}(\Frob_\eta)  & 0 \\ 0 & \widetilde{\kappa}^{-1}(\Frob_\eta) \end{array} \right].
  \end{align*}
  Combining all the observations stated above, we have the following isomorphism of $\Iwalg$-modules:
  \begin{align*}
    H^0\left(\Gal{\overline{\Q}_l}{K_\eta},D_{E,\widetilde{\kappa}^{-1}}\right)^\vee \cong \frac{\Z_p[[\widetilde{\Gamma}]]}{\left(\chi_p\widetilde{\kappa}^{-1}(\Frob_\eta)-1\right)} \oplus \frac{\Iwalg}{\left(\mathrm{ord}_{\pi_\eta}(q_{E,\eta}), \ \widetilde{\kappa}^{-1}(\Frob_\eta)-1 \right)}.
  \end{align*}
  This completes the proof of the Proposition.
\end{proof}

\section{Numerical evidence towards the validity of \ref{gcd}} \label{section-examples}

In this section, we wish to provide some evidence towards the existence of elliptic curves $E$ with supersingular reduction at an odd prime $p$ such that the following two conditions hold:
\begin{enumerate}[style=sameline, align=left, label=(\thesection\roman*), ref=(\thesection\roman*)]
  \item\label{cond:gcd} \ref{gcd} holds for the pair $\{\tDN42{++}, \tDN42{+-}\}$. That is, the elements $\tDN42{++}, \tDN42{+-}$ have no common irreducible factor in the UFD $\Z_p[[\widetilde{\Gamma}]]$.
  \item\label{cond:nontrivial} In $Z^2\left(\Z_p[[\widetilde{\Gamma}]]\right)$, we have $c_2 \left(\frac{\Z_p[[\widetilde{\Gamma}]]}{(\tDN42{++}, \tDN42{+-})}\right) \neq 0$.
\end{enumerate}
Let $E_K$ denote the elliptic curve corresponding to the quadratic twist of $E$ by the quadratic character $\epsilon_K: \Gal{\overline{\Q}}{\Q} \twoheadrightarrow \Gal{K}{\Q} \rightarrow \{\pm 1\}$. Since $p$ splits in the imaginary quadratic field $K$, the elliptic curve $E_K$ also has good supersingular reduction at $p$ with $a_p(E_K)=0$. {In all the examples we consider, we also indicate how to unconditionally verify Conjecture~\ref{coates-sujatha-conj}.

\subsection{Kobayashi's $\pm$ Selmer groups and Pollack's $p$-adic $L$-function} \label{S:kobayashi}

We will recall the description of the $p$-adic $L$-functions $\theta^{\pm}_E$ and the Selmer groups $\Sel^{\pm}(\Q,D_{\rhoDN21})$  associated to the elliptic curve $E$. The $p$-adic $L$-functions $\theta^{\pm}_{E_K}$ and the Selmer groups $\Sel^{\pm}(\Q,D_{\rhoDN21(\epsilon_K)})$ associated to the elliptic curve $E_K$ are defined similarly.

Let $\rhoDN21:G_\Sigma\rightarrow \GL_2(\Z_p[[\Gamma_{\Cyc}]])$ denote the Galois representation given by the action of $G_\Sigma$ on the following free $\Z_p[[\Gamma_{\Cyc}]]$-module of rank $2$:
\[
  T_{\rhoDN21}:=T_p(E)\otimes_{\Z_p} \Z_p[[\Gamma_{\Cyc}]](\kappa_{\Cyc}^{-1}).
\]
We will also need to consider the following discrete $\Z_p[[\Gamma_{\Cyc}]]$-module:
\[
  D_{\rhoDN{2}{1}}:=T_{\rhoDN21}\otimes_{\Z_p[[\Gamma_\Cyc]]}\Hom_{\mathrm{cont}}\left(\Z_p[[\Gamma_\Cyc]],\Q_p/\Z_p\right).
\]

We will recall Kobayashi's construction of $\pm$ Selmer groups over the cyclotomic $\Z_p$-extension of $\Q$ in \cite{kobayashi}. We have the following local conditions at $p$:
\begin{align*}
  E^+(\Q_p(\mu_{p^n})) & =\left\{P\in \hat E(\Q_p(\mu_{p^n})):\Tr_{n/m+1}P\in \hat E(\Q_p(\mu_{p^m})), \text{ for   odd } m<n\right\};  \\
  E^-(\Q_p(\mu_{p^n})) & =\left\{P\in \hat E(\Q_p(\mu_{p^n})):\Tr_{n/m+1}P\in \hat E(\Q_p(\mu_{p^m})), \text{ for   even } m<n\right\}.
\end{align*}
Here, $\Tr_{n/m+1}: \hat E(\Q_p(\mu_{p^n}))\rightarrow  \hat E(\Q_p(\mu_{p^{m+1}}))$ denotes the trace map. We set
\[
  E^\pm(\Q_{p,\Cyc})=\left(\cup_{n\ge1} E^\pm(\Q_p(\mu_{p^n}))\right)^{\Gal{\Q_p (\mu_{p^\infty})}{\Q_{p,\Cyc}}},
\]
and $\Loc_\pm(\Q_p,D_{\rhoDN{2}{1}})\subset H^1(\Q_p,D_{\rhoDN{2}{1}})$ is defined to be the image of $E^\pm(\Q_{p,\Cyc})\otimes \Q_p/\Z_p$ under the Kummer map. The $\pm$ Selmer groups are then given by
\[
  \Sel^\pm(\Q,D_{{\rhoDN{2}{1}}})=\ker\left(H^1\left(G_\Sigma, D_{\rhoDN{2}{1}}\right)\rightarrow \frac{H^1(\Q_p,D_{\rhoDN{2}{1}})}{\Loc_\pm(\Q_p,D_{\rhoDN{2}{1}})} \oplus  \bigoplus_{l \in \Sigma \setminus \{p\}} H^1\left(\Q_l,D_{\rhoDN{2}{1}}\right)\right).
\]
The Pontryagin dual of the $\pm$ Selmer groups are torsion modules over the ring $\Z_p[[\Gamma_{\Cyc}]]$. See \cite[Theorem~1.2]{kobayashi}.\\

\begin{remark}
  Note that our choice of signs $+$ and $-$ is opposite to the the one used by Kobayashi \cite{kobayashi}. The Selmer group that we have denoted $\Sel^+(\Q,D_{{\rhoDN{2}{1}}})$ (and $\Sel^-(\Q,D_{{\rhoDN{2}{1}}})$ respectively) corresponds to Kobayashi's Selmer group in \cite{kobayashi} with the minus sign (and plus sign respectively).
\end{remark}

An important point to note is that since $p$ splits in $K$, the $\Gal{\overline{\Q}_p}{\Q_p}$-modules $\Res^\Q_K \left(D_{\rhoDN{2}{1}}\right)$, $D_{\rhoDN{2}{1}}$ and $D_{\rhoDN{2}{1}(\epsilon_K)}$ are isomorphic. This allows us to make the following identifications:
\begin{align} \label{eq:identifylocal}
H^1\left(\Q_p,\Res^\Q_K \left(D_{\rhoDN{2}{1}}\right)\right) \cong H^1\left(\Q_p,D_{\rhoDN{2}{1}}\right) \cong H^1\left(\Q_p,D_{\rhoDN{2}{1}(\epsilon_K)}\right).
\end{align}

Let $\bullet, \circ \in \{+,-\}$. Since the local Selmer conditions at $p$ only depend on the isomorphism class of the elliptic curve over $\Q_p$, under the identification given in equation (\ref{eq:identifylocal}), we have the following identifications:
\begin{align*}
\Loc_\bullet\left(\Q_p, \Res^\Q_K \left(D_{\rhoDN{2}{1}}\right)\right) \cong \Loc_\bullet(\Q_p,D_{\rhoDN{2}{1}}) \cong \Loc_\bullet(\Q_p,D_{\rhoDN{2}{1}(\epsilon_K)}), \\ \Loc_\circ\left(\Q_p, \Res^\Q_K \left(D_{\rhoDN{2}{1}}\right)\right) \cong \Loc_\circ(\Q_p,D_{\rhoDN{2}{1}}) \cong \Loc_\circ(\Q_p,D_{\rhoDN{2}{1}(\epsilon_K)}).
\end{align*}

By \cite[Theorem 6.2]{kobayashi}, the $\Z_p[[\Gamma_\Cyc]]$-modules $\left(\frac{H^1\left(\Q_p,D_{\rhoDN{2}{1}} \right)}{\Loc_{\bullet}\left(\Q_p,D_{\rhoDN{2}{1}}\right)}\right)^\vee$ and $\left(\frac{H^1\left(\Q_p,D_{\rhoDN{2}{1}(\epsilon_K)}\right)}{\Loc_{\circ}\left(\Q_p,D_{\rhoDN{2}{1}(\epsilon_K)}\right)}\right)^\vee$ are free of rank one. Let us choose generators $b^\bullet$ and $b^\circ$ respectively for these free modules inside $H^1\left(\Q_p,D_{\rhoDN{2}{1}}\right)^\vee$ (or equivalently inside $H^1\left(\Q_p,D_{\rhoDN{2}{1}(\epsilon_K)}\right)^\vee$). We also have the following short exact sequences of $\Z_p[[\Gamma_\Cyc]]$-modules:
 \begin{align} \label{eq:ses-1variable}
 & 0 \rightarrow \left(\frac{H^1\left(\Q_p,D_{\rhoDN{2}{1}} \right)}{\Loc_{\bullet}\left(\Q_p,D_{\rhoDN{2}{1}}\right)}\right)^\vee \xrightarrow {\mathrm{rest}_E} H^1\left(G_\Sigma,D_{\rhoDN{2}{1}} \right)^\vee \rightarrow \Sel^\bullet(\Q,D_{\rhoDN{2}{1}})^\vee \rightarrow 0, \\ \notag
& 0 \rightarrow \left(\frac{H^1\left(\Q_p,D_{\rhoDN{2}{1}(\epsilon_K)} \right)}{\Loc_{\circ}\left(\Q_p,D_{\rhoDN{2}{1}(\epsilon_K)}\right)}\right)^\vee \xrightarrow {\mathrm{rest}_{E_K}} H^1\left(G_\Sigma,D_{\rhoDN{2}{1}(\epsilon_K)} \right)^\vee \rightarrow \Sel^\circ(\Q,D_{\rhoDN{2}{1}(\epsilon_K)})^\vee \rightarrow 0.
\end{align}
As mentioned above, the Pontryagin duals of the Selmer groups appearing in equation (\ref{eq:ses-1variable}) are  torsion (\cite[Theorem 1.2]{kobayashi}). Consequently, the $\Z_p[[\Gamma_\Cyc]]$-modules $H^1\left(G_\Sigma,D_{\rhoDN{2}{1}} \right)^\vee$ and $H^1\left(G_\Sigma,D_{\rhoDN{2}{1}(\epsilon_K)} \right)^\vee$ must have rank one. We can thus identify their reflexive hulls with the ring $\Z_p[[\Gamma_\Cyc]]$.  Moreoever, the elements ${\mathrm{rest}}_E(b^\bullet)$ and ${\mathrm{rest}}_{E_K}(b^\circ)$ must be  torsion-free over $\Z_p[[\Gamma_\Cyc]]$. This in turn allows us to identify ${\mathrm{rest}}_E(b^\bullet)$ and ${\mathrm{rest}}_{E_K}(b^\circ)$ with elements of the ring $\Z_p[[\Gamma_\Cyc]]$. Let us denote these elements by $\tilde{\mathrm{rest}}_E(b^\bullet)$ and $\tilde{\mathrm{rest}}_{E_K}(b^\circ)$ respectively. One may similarly define $\tilde{\mathrm{rest}}_E(b^\circ)$ and $\tilde{\mathrm{rest}}_{E_K}(b^\bullet)$. Let us also fix two elements $c_E$ and $c_{E_K}$ in  $\Z_p[[\Gamma_\Cyc]]$ such that $\Div(c_E) = \Div\left(H^1_{\Sigma_0}\left(G_\Sigma,D_{\rhoDN{2}{1}} \right)^\vee_{\mathrm{tor}}\right)$ and $\Div(c_{E_K}) = \Div\left(H^1_{\Sigma_0}\left(G_\Sigma,D_{\rhoDN{2}{1}(\epsilon_K)} \right)^\vee_{\mathrm{tor}}\right)$. Let us define the following elements in $\Z_p[[\Gamma_\Cyc]]$:
\begin{align} \label{eq:intermdiv}
& \vartheta_E^\bullet :=  \tilde{\mathrm{rest}}_E(b^\bullet)  \times c_E, \quad && \vartheta_E^\circ :=  \tilde{\mathrm{rest}}_E(b^\circ)  \times c_E, \\ \notag &  \vartheta_{E_K}^\bullet := \tilde{\mathrm{rest}}_{E_K}(b^\bullet) \times c_{E_K}, \quad && \vartheta_{E_K}^\circ := \tilde{\mathrm{rest}}_{E_K}(b^\circ) \times c_{E_K}.
\end{align}

Combining these observations and using the structure theorem for modules over PIDs, we have the following equality of divisors in $\Z_p[[\Gamma_\Cyc]]$:

\begin{align*}
& \Div\left(\vartheta_E^\bullet\right) =  \Div\left(\Sel^\bullet(\Q,D_{\rhoDN{2}{1}})^\vee\right), \quad && \Div\left(\vartheta_E^\circ\right) =  \Div\left(\Sel^\circ(\Q,D_{\rhoDN{2}{1}})^\vee\right), \\ & \Div\left(\vartheta_{E_K}^\bullet\right) =  \Div\left(\Sel^\bullet(\Q,D_{\rhoDN{2}{1}(\epsilon_K)})^\vee\right), \quad &&
\Div\left(\vartheta_{E_K}^\circ\right) =  \Div\left(\Sel^\circ(\Q,D_{\rhoDN{2}{1}(\epsilon_K)})^\vee\right).
\end{align*}

\begin{remark}
Choosing a different generator for $\left(\frac{H^1\left(\Q_p,D_{\rhoDN{2}{1}} \right)}{\Loc_{\bullet}\left(\Q_p,D_{\rhoDN{2}{1}}\right)}\right)^\vee$ changes $\tilde{\mathrm{rest}}_{E}(b^\bullet)$ and $\tilde{\mathrm{rest}}_{E_K}(b^\bullet)$ by the same unit in the ring $\Z_p[[\Gamma_\Cyc]]$. One obtains a similar conclusion by replacing $\bullet$ with $\circ$. By choosing a different identification of the reflexive hull of $H^1\left(G_\Sigma,D_{\rhoDN{2}{1}} \right)^\vee$ with the ring $\Z_p[[\Gamma_\Cyc]]$, one modifies $\tilde{\mathrm{rest}}_{E}(b^\bullet)$ and $\tilde{\mathrm{rest}}_{E}(b^\circ)$ by the same unit in $\Z_p[[\Gamma_\Cyc]]$. One obtains a similar conclusion by replacing $E$ with $E_K$. We may also modify $c_E$ and $c_{E_K}$ by units in $\Z_p[[\Gamma_\Cyc]]$. However, for all these different choices, the ideal $\left(\vartheta_E^\bullet \vartheta^\circ_{E_K} + \vartheta_E^\circ \vartheta^\bullet_{E_K}\right)$ in $\Z_p[[\Gamma_\Cyc]]$ remains the same.
\end{remark}

Kobayashi's Selmer groups are related to Pollack's $\pm$ one-variable $p$-adic $L$-functions defined in \cite{pollack}. Let $L_\lambda$ denote the classical $p$-adic $L$-function of Amice-V\'elu \cite{amicevelu} and Vi\v sik \cite{visik}, where $\lambda$ is of one the two roots of the Hecke polynomial $X^2+p$. We may consider $L_\lambda$ as a power series in $\Qp(\sqrt{-p})[[\gamma-1]]$ (where $\gamma$ is a fixed topological generator of $\Gamma_\Cyc$) and a $\Qp(\sqrt{-p})$-valued distribution on $\Gamma_\Cyc$ interchangably via Amice's transform. There exist $\theta_E^\pm \in\Z_p[[\Gamma_{\Cyc}]] $ such that
\[
  L_\lambda=\log_p^+\theta_E^++\lambda \log_p^-\theta_E^-,
\]
where $\log_p^\pm$ is the power series defined in Pollack's work \cite{pollack} using cyclotomic polynomials in $\gamma$. See \cite{dionlei} for a description of the corresponding distributions under Amice transform.\\

The $p$-adic $L$-functions $\theta_E^\pm$ satisfy the following interpolation properties: Let $\psi$ be a non-trivial Dirichlet character of conductor $p^n$ on $\Gamma_\Cyc$, with $n$ odd. Then,
\begin{equation}\label{eq:interkob1}
  \theta_E^-(\psi)=(-1)^{\frac{n+1}{2}}\frac{p^{n}}{\tau_\Q(\psi^{-1}){\displaystyle \prod_{\substack{1 \leq k < n \\ k \text{ odd}}}}\Phi_{p^{k}}(\zeta)}\times\frac{L(E,\psi^{-1},1)}{\Omega_{E}^-},
\end{equation}
Similarly, if $\psi$ is a non-trivial Dirichlet character of conductor $p^n$ on $\Gamma_\Cyc$, with $n$ even, then
\begin{equation}\label{eq:interkob2}
  \theta_E^+(\psi)=(-1)^{\frac{n}{2}}\frac{p^{n}}{\tau_\Q(\psi^{-1}){\displaystyle \prod_{\substack{1 \leq k < n \\ k \text{ even }}}}\Phi_{p^{k}}(\zeta)}\times\frac{L(E,\psi^{-1},1)}{\Omega_{E}^+}.
\end{equation}
Here, $\tau_\Q(\psi^{-1})$ denotes the Gauss sum of $\psi^{-1}$,  $\Phi_{p^k}$ denotes the $p^k$-th cyclotomic polynomial, $\zeta$ denotes $\psi(\gamma)$  which is a primitive $p^{n-1}$-root of unity, the complex numbers $\Omega_E^\pm$ denote the real (and imaginary) N\'eron periods associated to the elliptic curve $E$.

Note that $\theta_E^\pm$ are uniquely determined by \eqref{eq:interkob1} and \eqref{eq:interkob2} respectively. As in the two-variable case (c.f. Remark~\ref{rk:choices}), choosing a different topological generator of $\Gamma_\Cyc$ changes the $p$-adic $L$-function by a factor $u$, for some unit $u$ in the ring $\Z_p[[\Gamma_\Cyc]]$.

We will consider the following Iwasawa main conjecture formulated by Kobayashi \cite{kobayashi}:

\begin{conjecture}\label{conj:kobayashi} Let $\bullet,\circ\in\{\pm\}$. We have the following equalities  in $Z^1\left(\Z_p[[\Gamma_\Cyc]])\right)$:
  \begin{align*}
    \Div\left(\vartheta_E^\bullet\right)              = \Div\left(\theta^{\bullet}_E\right),     \qquad \Div(\vartheta_{E_K}^\circ)              = \Div\left(\theta^{\circ}_E\right).
  \end{align*}
\end{conjecture}

\subsection{Outline of strategy}

We briefly outline our strategy to prove the existence of elliptic curves verifying \ref{cond:gcd} and \ref{cond:nontrivial}. Consider the one-variable $p$-adic $L$-functions $\theta^{\pm}_E$, $\theta_{E_K}^{\pm}$  associated to the elliptic curves $E$ and $E_K$ respectively and $\vartheta^{\pm}_E$, $\vartheta_{E_K}^{\pm}$ as defined in \eqref{eq:intermdiv}.  For any $\bullet,\circ\in\{\pm\}$, let $\vartheta^{\bullet\circ}_{\cmmnt{\pmb}{4,2}}$ be an element in $\Iwalg$ such that $\Div(\vartheta^{\bullet\circ}_{\cmmnt{\pmb}{4,2}})$ equals $\Div(\Sel^{\bullet\circ}\left(\Q,D_{\rhoDN{4}{2}} \right)^\vee)$ in $Z^1\left(\Iwalg\right)$.

 Under the  natural projection map (which one can view as the ``cyclotomic specialization'')
\begin{align*}
  \pi_{\Cyc}: \Z_p[[\widetilde{\Gamma}]] \rightarrow \Z_p[[\Gamma_\Cyc]],
\end{align*}
we prove the following equalities of ideals in $\Z_p[[\Gamma_\Cyc]]$:
\begin{align} \label{equality-alg-cyclosp1}
 \left(\pi_{\Cyc}(\vartheta^{++}_{\cmmnt{\pmb}{4,2}})\right) = \left(\vartheta_E^+ \vartheta^+_{E_K}\right), \quad  \left(\pi_{\Cyc}(\vartheta^{+-}_{\cmmnt{\pmb}{4,2}})\right) = \left(\vartheta_E^+ \vartheta^-_{E_K} + \vartheta_E^- \vartheta^+_{E_K}\right).
\end{align}
Equation (\ref{equality-alg-cyclosp1}) is proved in Section \ref{ssec:cyclotomicspecialization}. If we assume the validity of Conjectures \ref{kim-main-conjecture} and \ref{conj:kobayashi},  there exists a unit $u$ in $\Z_p[[\Gamma_\Cyc]]$ such that we have the following equalities of ideals in $\Z_p[[\Gamma_\Cyc]]$:
\begin{align}\label{equality-cyclosp1}
  \left(\pi_{\Cyc}(\tDN42{++})\right) = \left(\theta_E^+ \theta^+_{E_K}\right),\qquad  \left(\pi_{\Cyc}(\tDN42{+-})\right) = \left(\theta_E^+ \theta^-_{E_K} + u \theta_E^- \theta^+_{E_K}\right).
\end{align}
}

{
The table given in Section \ref{S:pollacktable} has examples of elliptic curves $E$, primes $p$ and imaginary quadratic fields $K$ satisying the following conditions:
\begin{enumerate}[style=sameline, align=left, label=(\thesection\alph*), ref=(\thesection\alph*)]
  \item\label{cond:a} $\theta_{E}^+$ is a unit in $\Z_p[[\Gamma_\Cyc]]$.
  \item\label{cond:b} $\theta_{E_K}^+$ and $\theta_{E_K}^-$ have no common irreducible factor in the UFD $\Z_p[[\Gamma_\Cyc]]$.
  \item\label{cond:c} $\theta_{E_K}^+$ and $\theta_{E_K}^-$ are not units in $\Z_p[[\Gamma_\Cyc]]$.
\end{enumerate}
Suppose, as indicated by condition \ref{cond:a}, that $\theta_{E}^+$ is a unit in $\Z_p[[\Gamma_\Cyc]]$. The ratio of the two values obtained by specializing the $p$-adic $L$-functions $\theta_{E}^+$ and $\theta_{E}^-$ at the trivial character is a $p$-adic unit (in fact, it is equal to $\frac{p-1}{2}$). See \cite[Equation 3.6]{kobayashi}.  As a result, $\theta_{E}^-$ is also a unit in $\Z_p[[\Gamma_\Cyc]]$. In this case, $\theta_E^+ \theta^+_{E_K}$ and $\theta_E^+ \theta^-_{E_K} + u\theta_E^- \theta^+_{E_K}$ have no common irreducible factor in $\Z_p[[\Gamma_\Cyc]]$ if and only if $\theta^+_{E_K}$ and $\theta^-_{E_K}$ have no common irreducible factor in $\Z_p[[\Gamma_\Cyc]]$.

Note that \ref{gcd} holds for the pair $\{\tDN42{++}, \tDN42{+-}\}$ if $\pi_{\Cyc}(\tDN42{++})$ and $\pi_\Cyc(\tDN42{+-})$ have no common irreducible factor in $\Z_p[[\Gamma_\Cyc]]$. Assume the validity of Conjectures \ref{kim-main-conjecture} and \ref{conj:kobayashi}. Conditions \ref{cond:a} and \ref{cond:b}, along with equation (\ref{equality-cyclosp1}), then let us deduce that \ref{gcd} holds for the pair $\{\tDN42{++}, \tDN42{+-}\}$. Condition \ref{cond:c} lets us deduce that $c_2 \left(\frac{\Z_p[[\widetilde{\Gamma}]]}{(\tDN42{++}, \tDN42{+-})}\right) \neq 0$. \\

For all the curves given in the table in Section \ref{S:pollacktable} , it will be possible to unconditionally verify Conjecture~\ref{coates-sujatha-conj}. When the elliptic curve has complex multiplication (CM), Pollack and Rubin \cite{pollackrubin} have proved Conjecture \ref{conj:kobayashi}. If $E$ and $E_K$ are  elliptic curves without CM  such that the $p$-adic Galois representation $\rho:G_\Q \rightarrow \Gl_2(\Z_p)$, given by the action of $G_\Q$ on the corresponding $p$-adic Tate modules is surjective\footnote{If the Galois representations given by the action of $G_\Q$ on the $p$-adic Tate modules are not known to be surjective, then Kobayashi establishes the inequality (\ref{est-kobayashiinequality}) in $Z^1\left(\Z_p[[\Gamma_\Cyc]]\left[\frac{1}{p}\right]\right)$.  When the elliptic curve does not have complex multplication, further progress under certain hypotheses towards Conjecture \ref{conj:kobayashi} has been made by Wan \cite{wan}.}
, Kobayashi (\cite[Theorem 4.1]{kobayashi}) has shown that we have the following inequality of divisors in $Z^1\left(\Z_p[[\Gamma_\Cyc]]\right)$:
\begin{align} \label{est-kobayashiinequality}
\Div\left(\vartheta^{\pm}_E\right) \leq  \Div\left(\theta^{\pm}_E\right), \quad  \Div\left(\vartheta^{\pm}_{E_K}\right) \leq  \Div\left(\theta^{\pm}_{E_K}\right).
\end{align}

We use Sage \cite{sagemath} to numerically establish that the Galois representations given by the action of $G_\Q$ on the $p$-adic Tate modules is surjective; hence in these examples the inequalities in equation (\ref{est-kobayashiinequality}) hold unconditionally. We can argue as we did earlier. Using conditions \ref{cond:a} and \ref{cond:b} along with equation (\ref{equality-alg-cyclosp1}), we can conclude that $\pi_{\Cyc}(\vartheta^{++}_{\cmmnt{\pmb}{4,2}})$ and $\pi_{\Cyc}(\vartheta^{+-}_{\cmmnt{\pmb}{4,2}})$ have no common irreducible factor in $\Z_p[[\Gamma_\Cyc]]$. As a result, $\vartheta^{++}_{\cmmnt{\pmb}{4,2}}$ and $\vartheta^{+-}_{\cmmnt{\pmb}{4,2}}$ must have no common irreducible factor in $\Z_p[[\widetilde{\Gamma}]]$ and we can thus  conclude that the $\Z_p[[\widetilde{\Gamma}]]$-module $\Sha^1\left(\Q,D_{\rhoDN{4}{2}}\right)^\vee$ is pseudo-null (see equations (\ref{eq:surjectionspn}) and (\ref{supp-intersect})).
}

\subsection{Projection to the cyclotomic line} \label{ssec:cyclotomicspecialization}

Consider the following ring homomorphism, induced by the natural surjection $\widetilde{\Gamma} \rightarrow \Gamma_\Cyc$:
\begin{align*}
  \pi_{\Cyc} : \Z_p[[\widetilde{\Gamma}]] \rightarrow \Z_p[[\Gamma_\Cyc]].
\end{align*}

\begin{proposition}\label{prop:cyclosp}
  Let $\bullet,\circ\in\{\pm\}$. Then, we have the following equality of ideals in $\Z_p[[\Gamma_\Cyc]]$:
{  \begin{align}\label{equationspecialization}
    \left(\pi_{\Cyc}(\vartheta^{\bullet\circ}_{\cmmnt{\pmb}{4,2}})\right) = \left(\vartheta_E^\bullet \vartheta^\circ_{E_K} + \vartheta_E^\circ \vartheta^\bullet_{E_K} \right).
  \end{align}
  As a result, if we assume the validity of Conjectures \ref{kim-main-conjecture} and \ref{conj:kobayashi}, there exists a unit $u$ in $\Z_p[[\Gamma_\Cyc]]$ such that we have the following equalities of ideals in $\Z_p[[\Gamma_\Cyc]]$:
  \begin{align*}
     \left(\pi_{\Cyc}(\theta^{++}_{\cmmnt{\pmb}{4,2}})\right) = \left(\theta_E^+ \theta^+_{E_K}\right), \ \ \left(\pi_{\Cyc}(\theta^{--}_{\cmmnt{\pmb}{4,2}})\right) = \left(\theta_E^- \theta^-_{E_K}\right), \ \ \   \left(\pi_{\Cyc}(\theta^{+-}_{\cmmnt{\pmb}{4,2}})\right) = \left(\pi_{\Cyc}(\theta^{-+}_{\cmmnt{\pmb}{4,2}})\right) = \left(\theta_E^+ \theta^-_{E_K} + u\theta_E^- \theta^+_{E_K} \right).
  \end{align*}

}
\end{proposition}

\begin{proof}
  The equality of divisors\footnote{We adopt a convention that sets the divisor of a torsion-free module and $\Div(0)$ to equal ``infinity''.} $\Div\left(\vartheta^{\bullet\circ}_{\cmmnt{\pmb}{4,2}}\right) = \Div\left(\Sel^{\bullet\circ}\left(\Q,D_{\rhoDN{4}{2}} \right)^\vee \right)$ in $Z^1\left(\Z_p[[\widetilde{\Gamma}]]\right)$ lets us deduce the following equality of divisors in $Z^1\left(\Z_p[[\Gamma_\Cyc]]\right)$:
  \begin{align} \label{divspecialization}
    \Div\left(\pi_{\Cyc}\left(\vartheta^{\bullet\circ}_{\cmmnt{\pmb}{4,2}} \right)\right) = \Div\left(\Sel^{\bullet\circ}\left(\Q,D_{\rhoDN{4}{2}} \right)^\vee \otimes_{\Z_p[[\widetilde{\Gamma}]]} \Z_p[[\Gamma_\Cyc]]\right).
  \end{align}
  Equation (\ref{divspecialization}) follows by applying Proposition 5.2 in \cite{palvannan2016algebraic} to the specialization map $\pi_{\Cyc} : \Z_p[[\widetilde{\Gamma}]] \rightarrow \Z_p[[\Gamma_\Cyc]]$ of regular local rings. The only hypotheses from Proposition 5.2 in \cite{palvannan2016algebraic} that needs to be verified is that the $\Z_p[[\widetilde{\Gamma}]]$-module $\Sel^{\bullet\circ}\left(\Q,D_{\rhoDN{4}{2}} \right)^\vee $ has no non-trivial pseudo-null submodules. This follows from Proposition 4.1.1 in Greenberg's work \cite{greenberg2014pseudonull}.

{
Let $\Sigma_0$ denote $\Sigma \setminus \{p\}$. Following \cite{MR2290593}, we define
\begin{align}
H^1_{\Sigma_0}(G_\Sigma,D_{\rhoDN{4}{2}}) :=\ker\left(H^1(G_\Sigma,D_{\rhoDN{4}{2}})\rightarrow\bigoplus_{\nu\in \Sigma_0}H^1(\Q_\nu,D_{\rhoDN{4}{2}})\right).
\end{align}
 One can define a discrete $\Z_p[[\Gamma_\Cyc]]$-module $D_{\pi_\Cyc \circ \rhoDN{4}{2}}$ associated to the Galois representation $\pi_\Cyc \circ \rhoDN{4}{2}$, similar to the definition of the discrete $\Iwalg$-module $D_{\rhoDN{4}{2}}$ associated to $\rhoDN{4}{2}$. We may similarly define $H^1_{\Sigma_0}(G_\Sigma,D_{\pi_\Cyc \circ \rhoDN{4}{2}}) $, $H^1_{\Sigma_0}(G_\Sigma,D_{\rhoDN{2}{1}})$ and $H^1_{\Sigma_0}(G_\Sigma,D_{\rhoDN{2}{1}(\epsilon_K)})$.

Using Proposition 3.4 in \cite{MR2290593} along with \ref{locp0}, we have the following isomorphisms of $\Z_p[[\Gamma_\Cyc]]$-modules:
  \begin{align}\label{global-control}
    H^1\left(G_\Sigma,D_{\pi_\Cyc \circ \rhoDN{4}{2} }\right) \cong H^1\left(G_\Sigma,D_{\rhoDN{4}{2} }\right)[\ker(\pi_\Cyc)].
  \end{align}
  \begin{align}\label{local-control}
    H^1\left(\Q_p,D_{\pi_\Cyc \circ \rhoDN{4}{2} }\right) \cong H^1\left(\Q_p,D_{\rhoDN{4}{2} }\right)[\ker(\pi_\Cyc)].
  \end{align}

For $\ell \nmid p$, the cyclotomic $\Z_p$-extension $\Q_{l,\Cyc}$ is the unique $\Z_p$-extension of $\Q_l$. Consequently, the decomposition group of any prime of $K$ over $\ell$ in $\Gal{\tilde K_\infty }{K_\Cyc}$ is trivial. This lets us obtain the following isomorphism of $\Z_p[[\Gamma_\Cyc]]$-modules:
  \begin{align} \label{l-iso}
    H^1(\Q_\ell,D_{\pi_\Cyc \circ \rhoDN42})\cong H^1(\Q_\ell,\Drho) [\ker(\pi_\Cyc)].
  \end{align}
Equations (\ref{global-control}) and (\ref{l-iso}) lets us deduce the following isomorphism of $\Z_p[[\Gamma_\Cyc]]$-modules
\begin{align} \label{almost-global-control}
H^1_{\Sigma_0}\left(G_\Sigma,D_{\pi_\Cyc \circ \rhoDN{4}{2} }\right) \cong H^1_{\Sigma_0}\left(G_\Sigma,D_{\rhoDN{4}{2} }\right)[\ker(\pi_\Cyc)].
\end{align}

  The ring map $\pi_\Cyc$ induces the following decomposition of Galois representations:
  \begin{align}\label{eqGaloisrep}
    \pi_\Cyc \circ \rhoDN{4}{2} \cong \rhoDN{2}{1} \oplus \rhoDN{2}{1}(\epsilon_K).
  \end{align}

As a manifestation of the isomorphism in (\ref{eqGaloisrep}), we have the following isomorphism of $\Z_p[[\Gamma_\Cyc]]$-modules:
  \begin{align}\label{eq:globalisomanifestation}
    H^1\left(G_\Sigma,D_{\pi_\Cyc \circ \rhoDN{4}{2} }\right)    & \cong H^1\left(G_\Sigma,D_{\rhoDN{2}{1}} \right) \oplus H^1\left(G_\Sigma,D_{\rhoDN{2}{1}(\epsilon_K)} \right),                                      \\
    \notag H^1\left(\Q_p,D_{\pi_\Cyc \circ \rhoDN{4}{2} }\right) & \cong H^1\left(\Q_p,D_{\rhoDN{2}{1}} \right) \oplus H^1\left(\Q_p,D_{\rhoDN{2}{1}(\epsilon_K)} \right),                                               \\
    \notag
    H^1\left(\Q_l,D_{\pi_\Cyc \circ \rhoDN{4}{2} }\right)        & \cong H^1\left(\Q_l,D_{\rhoDN{2}{1}} \right) \oplus H^1\left(\Q_l,D_{\rhoDN{2}{1}(\epsilon_K)} \right), \ \  \forall \ l \in \Sigma \setminus \{p\}.
  \end{align}

  These observations, along with equation (\ref{almost-global-control}), lets us deduce the following isomorphism of $\Z_p[[\Gamma_\Cyc]]$-modules:
  \begin{align}
H^1_{\Sigma_0}\left(G_\Sigma,D_{ \rhoDN{4}{2} }\right)[\ker(\pi_\Cyc)]    & \cong H^1_{\Sigma_0}\left(G_\Sigma,D_{\rhoDN{2}{1}} \right) \oplus H^1_{\Sigma_0}\left(G_\Sigma,D_{\rhoDN{2}{1}(\epsilon_K)} \right).
\end{align}

The discrete module $D_{\pi_\Cyc \circ \rhoDN{4}{2}}$ can be identified with $\Res^\Q_K \left(D_{\rhoDN{2}{1}}\right) \oplus \tilde{\delta} \otimes \Res^\Q_K \left(D_{\rhoDN{2}{1}}\right)$. Here,  $\tilde{\delta}$ denotes a  complex conjugation, which is an element of order $2$ in $G_\Sigma$. We can also describe the Galois action on $\Res^\Q_K \left(D_{\rhoDN{2}{1}}\right) \oplus \tilde{\delta} \otimes \Res^\Q_K \left(D_{\rhoDN{2}{1}}\right)$:
\begin{align*}
\tilde{\delta} \cdot (x, \tilde{\delta} \otimes y) = (y, \tilde{\delta} \otimes x), \qquad h \cdot (x, \tilde{\delta} \otimes y) = \left(hx, \tilde{\delta} \otimes \left(\tilde{\delta} h \tilde{\delta} y\right)\right), \ \forall \ h \in \Gal{\Q_\Sigma}{K}.
\end{align*}
We can identify $D_{\rhoDN{2}{1}(\epsilon_K)}$ with $D_{\rhoDN{2}{1}}(\epsilon_K)$. It will be helpful to make the isomorphism in equation (\ref{eqGaloisrep}) explicit.
 \begin{align} \label{eq:identify}
\Res^\Q_K \left(D_{\rhoDN{2}{1}}\right) \oplus \tilde{\delta} \otimes \Res^\Q_K \left(D_{\rhoDN{2}{1}}\right) & \cong D_{\rhoDN{2}{1}} \oplus D_{\rhoDN{2}{1}(\epsilon_K)}, \\ \notag
(x,0) & \rightarrow (x,x), \\ \notag
(0,\tilde{\delta} \otimes y) & \rightarrow (y,-y).
\end{align}

Throughout the proof, we will keep the identifications in equation (\ref{eq:identifylocal}) in mind. We now proceed to establish control for the local condition at $p$. Following section \ref{S:plusminus}, the discrete module $H^1\left(\Q_p,D_{\rhoDN{4}{2} }\right)$ can be identified with $\bigoplus_{\P|\p}H^1\left((\tilde K_\infty)_\P,E[p^\infty]\right) \oplus \bigoplus_{\QQ|\q}H^1\left((\tilde K_\infty)_\QQ,E[p^\infty]\right)$, whereas  $ H^1\left(\Q_p,D_{\pi_\Cyc \circ \rhoDN{4}{2} }\right)$ can be identified with $H^1\left(\Q_p, \Res^\Q_K \left(D_{\rhoDN{2}{1}}\right)\right)\oplus \tilde{\delta}\otimes H^1\left(\Q_p, \Res^\Q_K \left(D_{\rhoDN{2}{1}}\right)\right) $. We have natural injections of $\Z_p[[\Gamma_\Cyc]]$-modules:
  \begin{align} \label{eq:naturalmorphisms}
    \Loc_\bullet\left(\Qp,\Res^\Q_K \left(D_{\rhoDN{2}{1}}\right)\right) & \hookrightarrow \left(\bigoplus_{\P \mid \p} \Loc_\bullet\left((\tilde K_\infty)_\P,E[p^\infty]\right)\right) [\ker(\pi_\Cyc)], \\ \notag \tilde{\delta} \otimes \Loc_\circ\left(\Qp, \Res^\Q_K \left(D_{\rhoDN{2}{1}}\right)\right) & \hookrightarrow  \left(\bigoplus_{\QQ \mid \q} \Loc_\circ \left((\tilde K_\infty)_\QQ,E[p^\infty]\right)\right) [\ker(\pi_\Cyc)].
  \end{align}

  On the one hand, Proposition 2.11  in Kim's work \cite{MR3224266} lets us deduce that the Pontryagin duals of the $\Z_p[[\Gamma_\Cyc]]$-modules
  \begin{align}\label{localpcontrolI}
    \left(\bigoplus_{\P \mid \p} \Loc_\bullet\left((\tilde K_\infty)_\P,E[p^\infty]\right)\right) [\ker(\pi_\Cyc)], \quad \left(\bigoplus_{\QQ \mid \q} \Loc_\circ \left((\tilde K_\infty)_\QQ,E[p^\infty]\right)\right) [\ker(\pi_\Cyc)],
  \end{align}
  are free of rank one.
  On the other hand, Theorems 2.7 and 2.8 in Kim's work \cite{MR3224266} lets us deduce that the Pontryagin dual of the $\Z_p[[\Gamma_\Cyc]]$-modules
  \begin{align*}
    \Loc_\bullet\left(\Qp,\Res^\Q_K \left(D_{\rhoDN{2}{1}}\right)\right), \qquad   \tilde{\delta}\otimes\Loc_\circ\left(\Qp,\Res^\Q_K \left(D_{\rhoDN{2}{1}}\right)\right),
  \end{align*}
  are also free of rank one. Consequently, the maps in equation (\ref{eq:naturalmorphisms}) must be isomorphisms.

  Recall that we have a short exact sequence $$0 \rightarrow \left(\frac{H^1\left(\Q_p,D_{\rhoDN42}\right)}{\Loc_{\bullet\circ}\left(\Q_p,D_{\rhoDN42}\right)}\right)^\vee \rightarrow H^1\left(\Q_p,D_{\rhoDN42}\right)^\vee \rightarrow \Loc_{\bullet\circ}\left(\Q_p,D_{\rhoDN42}\right)^\vee \rightarrow 0$$ of free $\Iwalg$-modules. Taking the tensor product $\otimes_{\Iwalg} \frac{\Iwalg}{\ker(\pi_\Cyc)}$ and then considering the Pontryagin dual, we obtain the following isomorphism of discrete $\Z_p[[\Gamma_\Cyc]]$-modules:
  \begin{align} \label{pControlIII}
    \frac{H^1\left(\Q_p,D_{\rhoDN42}\right)}{\Loc_{\bullet\circ}\left(\Q_p,D_{\rhoDN42}\right)}[\ker(\pi_\Cyc)] \cong \frac{H^1\left(\Q_p,D_{\rhoDN42}\right)[\ker(\pi_\Cyc)]}{\Loc_{\bullet\circ}\left(\Q_p,D_{\rhoDN42}\right)[\ker(\pi_\Cyc)]}.
  \end{align}

}

{

We have the following commutative diagram

  \begin{tikzpicture}[auto]
    \node (H1locquotP) {$\left(\frac{H^1\left(\Q_p, \Res^\Q_K \left(D_{\rhoDN{2}{1}}\right)\right)}{\Loc_\bullet\left(\Q_p, \Res^\Q_K \left(D_{\rhoDN{2}{1}}\right)\right)}\right)^\vee$};

  \node (H1locquotQ) [right of =H1locquotP, node distance = 4.8cm]{$\tilde{\delta} \otimes \left(\frac{H^1\left(\Q_p, \Res^\Q_K \left(D_{\rhoDN{2}{1}}\right)\right)}{\Loc_\circ\left(\Q_p, \Res^\Q_K \left(D_{\rhoDN{2}{1}}\right)\right)}\right)^\vee$};

\node(oplus1left)[right of = H1locquotP, node distance=2.2cm]{$\oplus$};

    \node(H1p) [below of =H1locquotP, node distance = 2cm] {$H^1\left(\Q_p, \Res^\Q_K \left(D_{\rhoDN{2}{1}}\right)\right)^\vee$};

  \node(H1q) [below of =  H1locquotQ, node distance = 2cm] {$\tilde{\delta} \otimes H^1\left(\Q_p, \Res^\Q_K \left(D_{\rhoDN{2}{1}}\right)\right)^\vee$};

\node(oplus2left)[right of = H1p, node distance=2.2cm]{$\oplus$};

 \node(H1localE) [right of =H1q, node distance = 4.8cm] {$H^1\left(\Q_p,D_{\rhoDN21}\right)^\vee$};

 \node(H1localEK) [right of =H1localE, node distance = 3.4cm] {$H^1\left(\Q_p,D_{\rhoDN21(\epsilon_K)}\right)^\vee$};

 \node(oplus3left)[right of = H1localE, node distance=1.5cm]{$\oplus$};

  \node(H1global) [below of =oplus2left, node distance = 2cm] {$\bigg(H^1_{\Sigma_0}\left(G_\Sigma,D_{\pi_\Cyc \circ \rhoDN{4}{2} }\right)[\ker(\pi_\Cyc)]\bigg)^\vee$};

   \node(H1globalE) [below of =H1localE, node distance = 2cm] {$ H^1_{\Sigma_0}\left(G_\Sigma,D_{\rhoDN{2}{1}} \right)^\vee$};

   \node(oplus4left)[right of = H1globalE, node distance=1.5cm]{$\oplus$};

    \node(H1globalEK) [below of =H1localEK, node distance = 2cm] {$ H^1_{\Sigma_0}\left(G_\Sigma,D_{\rhoDN{2}{1}(\epsilon_K)} \right)^\vee$};

    \draw[dashed] (H1locquotP) to node  {} (H1p);
    \draw[dashed] (H1locquotQ) to node  {} (H1q);
     \draw[dashed] (H1localE)   to node[left] {$\mathrm{rest}_E$}  (H1globalE);
    \draw[dashed] (H1localEK) to node  {$\mathrm{rest}_{E_K}$} (H1globalEK);

    \draw[->] (oplus2left) to node  {$\mathrm{rest}$} (H1global);
     \draw[->] (oplus3left) to node  {$\mathrm{rest}_0$} (oplus4left);
     \draw[->] (oplus1left) to node  {$j$} (oplus2left);
     \draw[->] (H1global) to node  {$\cong$} (H1globalE);
     \draw[->] (H1q) to node  {$\cong$} (H1localE);
      \draw[->] (H1q) to node  {$\cong$} (H1localE);
     \draw[->] (H1locquotQ) to node  {$j_0$} (H1localEK);

           \end{tikzpicture}

Note that $\coker(\mathrm{rest} \circ j)$ is isomorphic to $\coker(\mathrm{rest}_0 \circ j_0)$. Analysing the maps on the left side, we have the following isomorphism of $\Z_p[[\Gamma_\Cyc]]$-modules:
\begin{align} \label{eq:cokernelsel}
\coker(\mathrm{rest} \circ j) \cong \Sel^{\bullet\circ}\left(\Q,D_{\rhoDN{4}{2}} \right)^\vee \otimes_{\Z_p[[\widetilde{\Gamma}]]} \Z_p[[\Gamma_\Cyc]].
\end{align}
Using equation (\ref{divspecialization}), we can deduce the following equality of divisors:
\begin{align} \label{eq:leftcokernel}
\Div\left(\coker(\mathrm{rest} \circ j) \right) =     \Div\left(\pi_{\Cyc}\left(\vartheta^{\bullet\circ}_{\cmmnt{\pmb}{4,2}} \right)\right).
\end{align}

Using the description given in equation (\ref{eq:identify}), we have
\begin{align*}
j_0(b^\bullet,0) = (b^\bullet, b^\bullet), \qquad j_0(0,\tilde{\delta}\otimes b^\circ) = (b^\circ, -b^\circ).
\end{align*}

As a $\Z_p[[\Gamma_\Cyc]]$-module, $\mathrm{Image}\left(\mathrm{rest}_0 \circ j_0\right)$ is generated by $(\mathrm{rest}_{E}(b^\bullet), \mathrm{rest}_{E_K}(b^\bullet))$ and $(\mathrm{rest}_E(b^\circ) , \mathrm{rest}_{E_K}(b^\circ))$. Suppose $\coker\left(\mathrm{rest}_0 \circ j_0\right)$ is a torsion $\Z_p[[\Gamma_\Cyc]]$-module. Note that $\mathrm{Image}\left(\mathrm{rest}_0 \circ j_0\right)$ must be a torsion-free $\Z_p[[\Gamma_\Cyc]]$-module. To see this, it is enough to observe that the domain of $\mathrm{rest}_0 \circ j_0$ is a free $\Z_p[[\Gamma_\Cyc]]$-module of rank two, whereas the codomain of the map is a $\Z_p[[\Gamma_\Cyc]]$-module of rank two. By localizing at every height one prime ideal of the ring $\Z_p[[\Gamma_\Cyc]]$ and using the structure theorem for PIDs, we can conclude that the divisor $\Div\left(\coker(\mathrm{rest}_0 \circ j_0)\right)$ is equal to
\begin{align*}
 \Div\left(\det\left[\begin{array}{cc} \tilde{\mathrm{rest}}_E(b^\bullet) & \tilde{\mathrm{rest}}_{E_K}(b^\bullet) \\  \tilde{\mathrm{rest}}_{E}(b^\circ) & -\tilde{\mathrm{rest}}_{E_K}(b^\circ) \end{array} \right]\right)   +  \Div\left(H^1_{\Sigma_0}\left(G_\Sigma,D_{\rhoDN{2}{1}}\right)^\vee_{\mathrm{tor}}\right) + \Div\left(H^1_{\Sigma_0}\left(G_\Sigma,D_{\rhoDN{2}{1}(\epsilon_K)} \right)^\vee_{\mathrm{tor}}\right).
\end{align*}
Combining the above observation with equation (\ref{eq:intermdiv}), we have
\begin{align} \label{eq:rightcokernel}
\Div\left(\coker(\mathrm{rest}_0 \circ j_0)\right) =  \Div\left(\vartheta_E^\bullet \vartheta^\circ_{E_K} + \vartheta_E^\circ \vartheta^\bullet_{E_K}\right).
\end{align}

Combining equations (\ref{eq:leftcokernel}) and (\ref{eq:rightcokernel}), we have the following equality of divisors in $\Z_p[[\Gamma_\Cyc]]$:
  \begin{align} \label{eq:finalequalitypropdiv}
   \Div\left(\pi_{\Cyc}(\vartheta^{\bullet\circ}_{\cmmnt{\pmb}{4,2}})\right) = \Div\left(\vartheta_E^\bullet \vartheta^\circ_{E_K} + \vartheta_E^\circ \vartheta^\bullet_{E_K}\right).
  \end{align}
If the $\Z_p[[\Gamma_\Cyc]]$-module $\coker\left(\mathrm{rest}_0 \circ j_0\right)$ is not torsion, then $\det\left[\begin{array}{cc} \tilde{\mathrm{rest}}_E(b^\bullet) & \tilde{\mathrm{rest}}_{E_K}(b^\bullet) \\  \tilde{\mathrm{rest}}_{E}(b^\circ) & -\tilde{\mathrm{rest}}_{E_K}(b^\circ) \end{array} \right]$ must equal zero. By equation (\ref{eq:cokernelsel}), $\Sel^{\bullet\circ}\left(\Q,D_{\rhoDN{4}{2}} \right)^\vee \otimes_{\Z_p[[\widetilde{\Gamma}]]} \Z_p[[\Gamma_\Cyc]]$ is also not torsion. Equality in equation (\ref{eq:finalequalitypropdiv}) still holds, where both terms now equal ``infinity''.  The proposition follows.
}
\end{proof}

\begin{remark} \label{rem:cyclotomicspecialization}
{ In the equal sign case, since $p$ is odd, Proposition \ref{prop:cyclosp} lets us deduce the equality of ideals in $\Z_p[[\Gamma_\Cyc]]$ given by $\left(\pi_{\Cyc}(\vartheta^{++}_{\cmmnt{\pmb}{4,2}})\right) = \left(\vartheta_E^+ \vartheta^+_{E_K}  \right)$ and $\left(\pi_{\Cyc}(\vartheta^{--}_{\cmmnt{\pmb}{4,2}})\right) = \left(\vartheta_E^- \vartheta^-_{E_K}  \right)$. By \cite[Theorem 1.2]{kobayashi}, the elements $\vartheta_E^\pm$ and $\vartheta^\pm_{E_K}$ are non-zero. Consequently, the $\Z_p[[\widetilde{\Gamma}]]$-modules $\Sel^{++}\left(\Q,D_{\rhoDN{4}{2}} \right)^\vee$ and $\Sel^{--}\left(\Q,D_{\rhoDN{4}{2}} \right)^\vee$ are torsion.} It is possible to deduce the following equality of ideals in $\Z_p[[\Gamma_\Cyc]]$:
  \begin{align*}
    \left(\pi_{\Cyc}(\tDN42{++})\right) = \left(\theta_E^+\theta^+_{E_K}\right), \qquad \left(\pi_{\Cyc}(\tDN42{--})\right) = \left(\theta_E^-\theta^-_{E_K}\right),
  \end{align*}
  without assuming the validity of Conjectures \ref{kim-main-conjecture} and  \ref{conj:kobayashi}. Consider a finite Galois character $\psi: \Gal{K_\Cyc}{K} \rightarrow \overline{\Q}_p^\times$ with conductor $(p^n)$, for some even positive integer $n$. Abusing notation, we also let $\psi: \Z_p[[\Gal{K_\Cyc}{K}]]\rightarrow \overline{\Q}_p$ denote the corresponding ring homomorphism. Recall that we chose $\gamma_\q$ to be the image of $\gamma_\p$ under the ring automorphism $\Z_p[[\widetilde{\Gamma}]] \rightarrow \Z_p[[\widetilde{\Gamma}]]$ induced by complex conjugation. We have $\psi(\gamma_\p)=\psi(\gamma_\q)$ whenever $\psi$ factors through $\Gal{K_\Cyc}{K}$. Using \eqref{eq:loeffler--}, we obtain the following equalities:
  \begin{align*}
    \psi(\pi_\Cyc(\tDN42{++})) & =\frac{p^{2n}}{\tau_K(\psi^{-1})}\times \frac{L(E/K,\psi^{-1},1)}{\Omega_{f_E}^+\Omega_{f_E}^-\left({\displaystyle {\displaystyle \prod_{\substack{1 \leq k < n \\ k \text{ even}}}}} \Phi_{p^k}(\zeta)\right)^2} \\ &=\frac{p^{2n}}{\tau_K(\psi^{-1})}\times \frac{L(E,\psi^{-1},1)L(E_K,\psi^{-1},1)}{\Omega_{f_E}^+\Omega_{f_E}^-\left({\displaystyle\prod_{\substack{1 \leq k < n \\ k \text{ even} }}}\Phi_{p^k}(\zeta)\right)^2}.
  \end{align*}

  We may identify $\Gal{K_\Cyc}{K}$ with $\Gamma_\Cyc$ and choose $\gamma_\Cyc:=(\gamma_\p\gamma_\q)^{1/2}$ to be the topological generator of $\Gamma_\Cyc$ (it is possible to take a square root, since $p$ is odd). Then, $\psi(\gamma_\p)=\psi(\gamma_\q)=\psi(\gamma_\Cyc)$ for all characters $\psi$ on $\Gal{K_\Cyc}{K}$. Observe that $\tau_K(\psi^{-1}) = \tau_\Q(\psi^{-1})^2$. On comparing the interpolation formula above with equation \eqref{eq:interkob2}, we deduce that the following equality in $\overline{\Q}_p$:
  \[
    \psi(\pi_\Cyc(\tDN42{++}))=\frac{\Omega_E^+}{\Omega_{f_{E}}^+}  \frac{\Omega_{E_K}^+}{\Omega_{f_{E_K}}^+}   \frac{\Omega_{f_{E_K}}^+}{\Omega_{f_E}^-}\times\psi(\theta_E^+\theta_{E_K}^{+}),
  \]
  Since this holds for infinitely many characters $\psi$, we have the following equality in $\Z_p[[\Gamma_\Cyc]]$:
  \[
    \pi_\Cyc(\tDN42{++})=\frac{\Omega_E^+}{\Omega_{f_{E}}^+}  \frac{\Omega_{E_K}^+}{\Omega_{f_{E_K}}^+}   \frac{\Omega_{f_{E_K}}^+}{\Omega_{f_E}^-}\times\theta_E^+\theta_{E_K}^{+}.
  \]
  The ratios $\frac{\Omega_E^+}{\Omega_{f_{E}}^+}$ and  $\frac{\Omega_{E_K}^+}{\Omega_{f_{E_K}}^+}$ are units in the localization $\Z_{(p)}$. See Remark 5.5 in Rob Pollack's work \cite{pollack}. The ratio $\frac{\Omega_{f_{E_K}}^+}{\Omega_{f_E}^-}$ is also a unit in the localization $\Z_{(p)}$.  See Skinner-Zhang's work \cite[Lemma~9.6]{SkinnerZhang}). Our claim now follows.

  By considering cyclotomic characters whose conductors are odd powers of $p$, we may deduce in the same way the following factorization in $\Z_p[[\Gamma_\Cyc]]$:
  \[
    \pi_\Cyc(\tDN42{--})=\frac{\Omega_E^-}{\Omega_{f_{E}}^-}  \frac{\Omega_{E_K}^-}{\Omega_{f_{E_K}}^-}   \frac{\Omega_{f_{E_K}}^-}{\Omega_{f_E}^+}\times\theta_E^-\theta_{E_K}^{-}.
  \]

  {\begin{remark} In the mixed sign case, Proposition \ref{prop:cyclosp} lets us deduce the equality of ideals in $\Z_p[[\Gamma_\Cyc]]$ given by $\left(\pi_{\Cyc}(\vartheta^{+-}_{\cmmnt{\pmb}{4,2}})\right) = \left(\pi_{\Cyc}(\vartheta^{-+}_{\cmmnt{\pmb}{4,2}})\right) = \left(\vartheta_E^+ \vartheta^-_{E_K} +\vartheta_E^- \vartheta^+_{E_K}  \right)$. After the completion of our project, we learnt that in the preprint \cite[Proposition 3.5]{CCSS}, the authors have explained how to deduce a factorization of $p$-adic $L$-functions in the mixed sign case using rank-two Beilinson-Flach elements. In general, we are not able to rule out the possibility that $\vartheta_E^+\vartheta_{E_K}^{-} + \vartheta_E^-\vartheta_{E_K}^{+}$ could be zero. If it does equal zero, Proposition \ref{prop:cyclosp} asserts that $\ker(\pi_\Cyc)$  belongs to the support of the $\Z_p[[\widetilde{\Gamma}]]$-modules $\Sel^{+-}\left(\Q,D_{\rhoDN{4}{2}} \right)^\vee$ and $\Sel^{-+}\left(\Q,D_{\rhoDN{4}{2}} \right)^\vee$. However in the examples that we consider in Section \ref{S:pollacktable}, we can use the arguments given at the beginning of Section \ref{section-examples} to conclude that $\vartheta_E^+\vartheta_{E_K}^{-} + \vartheta_E^-\vartheta_{E_K}^{+}$ is in fact non-zero. For these examples, we can conclude that the $\Z_p[[\widetilde{\Gamma}]]$-modules $\Sel^{+-}\left(\Q,D_{\rhoDN{4}{2}} \right)^\vee$ and $\Sel^{-+}\left(\Q,D_{\rhoDN{4}{2}} \right)^\vee$ are torsion.
  \end{remark}}
\end{remark}

\subsection{Examples} \label{S:pollacktable}

The following data is computed in Sage \cite{sagemath} using Rob Pollack's algorithms. Here, $\lambda\left(\theta_{E_K}^+\right)$ and $\lambda\left(\theta_{E_K}^-\right)$ denote the Lambda-invariants of $\theta_{E_K}^+$ and $\theta_{E_K}^-$ respectively. The $\mu$-invariants for these $p$-adic $L$-functions turn out to equal zero.

\begin{center}
  {\renewcommand{\arraystretch}{1.2}
      \captionof{table}{Examples}\label{table:evidence}
    \begin{tabular}{ | c | l |c| c | c |c|c| }
      \hline
      $E$ & $K$               & $p$ & $\lambda\left(\theta_{E_K}^+\right)$ & Roots for $\theta_{E_K}^+$& $\lambda\left(\theta_{E_K}^-\right)$ & Roots for $\theta_{E_K}^-$ \\
      \hline
      32A & $\Q(\sqrt{-43})$  & 3   & 8                                    & $(2:\frac{1}{2})$, $(6:\frac{1}{6})$                     & 2                                    & $(2:1)$                      \\
          & $\Q(\sqrt{-107})$  & 3   & 2                                    & $(2:1)$                    & 6                                    & $(2:\frac{1}{2})$, $(4:\frac{1}{4})$                     \\
          & $\Q(\sqrt{-283})$  & 3   & 6                                    & $(2:\frac{1}{2})$, $(4:\frac{1}{4})$                     & 2                                    & (2:1)                     \\
          &  & & & & &   \\
      \hline
     40A      & $\Q(\sqrt{-331})$ & 3   & 6                                    & $(2:\frac{1}{2})$, $(4:\frac{1}{4})$                     & 2                                    & $(2:1)$                     \\
  &  & & & & &   \\
\hline
     56A      & $\Q(\sqrt{-139})$  & 3   & 6                                    & $(2:\frac{1}{2})$, $(4:\frac{1}{4})$                     & 2                                    & $(2:1)$                     \\
     & $\Q(\sqrt{-487})$  & 3   & 6                                    & $(2:\frac{1}{2})$, $(4:\frac{1}{4})$                     & 2                                    & $(2:1)$                     \\
  &  & & & & &   \\      \hline
    \end{tabular}
  }
\end{center}

Assume the validity of Conjectures \ref{kim-main-conjecture} and \ref{conj:kobayashi}. To ensure that Condition \ref{cond:a} is satisfied for all the elliptic curves in Table \ref{table:evidence}, that is, $\theta_E^+$ is a unit in the ring $\Z_p[[\Gamma_\Cyc]]$, one can glean from Rob Pollack's tables on his website  \url{http://math.bu.edu/people/rpollack/Data/data.html} that $\lambda(\theta_E^+) = \mu(\theta_E^+)=0$.

To ensure that Condition \ref{cond:b}  is satisfied for all the elliptic curves in Table \ref{table:evidence}, that is, $\theta_{E_K}^+$ and $\theta_{E_K}^-$ have no common irreducible factor in the UFD $\Z_p[[\Gamma_\Cyc]]$, it is sufficient to observe that the valuation of the roots of $\theta_{E_K}^+$ and $\theta_{E_K}^-$ are different. The entries $(r:s)$, under the columns for the roots of the $p$-adic $L$-functions, denote $r$ roots with $p$-adic valuation $s$.

To ensure that Condition \ref{cond:c} is satisfied for all the elliptic curves in Table \ref{table:evidence},  that is, $\theta_{E_K}^\pm$ are not units in the ring $\Z_p[[\Gamma_\Cyc]]$, one can glean the following inequality of Iwasawa invariants of the $p$-adic $L$-functions $\theta_{E_K}^+$ and $\theta_{E_K}^-$ in the ring $\Z_p[[\Gamma_\Cyc]]$ from Rob Pollack's tables:
\begin{align*}
  \lambda(\theta_{E_K}^+) \neq 0, \qquad \lambda(\theta_{E_K}^-) \neq 0.
\end{align*}

The validity of conditions \ref{cond:a} and \ref{cond:b} is sufficient to assert \ref{gcd} (one doesn't need Condition \ref{cond:c}). For more evidence towards \ref{gcd}, we refer the interested reader to Pollack's tables on his website \url{http://math.bu.edu/people/rpollack/Data/data.html}, and to use the heuristics provided by Problem 3.2 in the work of Kurihara-Pollack \cite{kuriharapollack}. See also the examples in Section 3.3 of their work.\\

Without assuming the validity of the two variable Iwasawa main conjectures (Conjecture \ref{kim-main-conjecture}), one can still use Pollack's algorithms and results of Kobayashi to obtain unconditional results towards the pseudo-nullity conjecture of Coates-Sujatha in this setup. \\

Note that we have natural surjections of $\Iwalg$-modules:
\begin{align} \label{eq:surjectionspn}
    \Sel^{++}(\Q,D_{\rhoDN{4}{2}})^\vee \twoheadrightarrow \Sha^1\left(\Q,D_{\rhoDN{4}{2}}\right)^\vee, \qquad \Sel^{+-}(\Q,D_{\rhoDN{4}{2}})^\vee \twoheadrightarrow \Sha^1\left(\Q,D_{\rhoDN{4}{2}}\right)^\vee.
\end{align}

Consider the elliptic curve $E=32A$ and the corresponding quadratic twists in Table \ref{table:evidence}. These elliptic curves have CM\footnote{We thank Somnath Jha for alerting us to this fact.} by the imaginary quadratic field $\Q(\sqrt{-1})$. The one-variable Iwasawa main conjecture (Conjecture \ref{conj:kobayashi}) is known due to work of Pollack-Rubin \cite{pollackrubin} when the elliptic curve (with supersingular reduction at $p$) has CM. We have the following equality of divisors in $Z^1\left(\Z_p[[\Gamma_\Cyc]]\right)$:
\begin{align*}
            \Div\left(\Sel^{+}\left(\Q,D_{\rhoDN21} \right)^\vee\right) = \Div(\theta_E^+) = 0, \qquad  \Div\left(\Sel^{+}\left(\Q,D_{\rhoDN21(\epsilon_K)} \right)^\vee \right) &= \Div(\theta_{E_K}^+),   \\ \Div\left(\Sel^{-}\left(\Q,D_{\rhoDN21(\epsilon_K)} \right)^\vee \right) & = \Div(\theta_{E_K}^-) .
          \end{align*}

Using Proposition \ref{prop:cyclosp} along with the data in Table \ref{table:evidence}, we have
  \begin{align*}
    \mathrm{Supp}_{\Ht=1}\left(\Sel^{++}(\Q,D_{\rhoDN{4}{2}})^\vee\right) & \bigcap \mathrm{Supp}_{\Ht=1}\left(\Sel^{+-}(\Q,D_{\rhoDN{4}{2}})^\vee\right) =\varnothing,
  \end{align*}
for the elliptic curve $E=32A$, $p=3$ and the corresponding imaginary quadratic fields in Table \ref{table:evidence}. Equation (\ref{eq:surjectionspn}) now lets us conclude that the $\Iwalg$-module $\Sha^1\left(\Q,D_{\rhoDN{4}{2}}\right)^\vee$ is pseudo-null in these examples. \\

We can also unconditionally assert that the $\Iwalg$-module $\Sha^1\left(\Q,D_{\rhoDN{4}{2}}\right)^\vee$ is pseudo-null for the elliptic curves $E=40A$ and $E=56A$, $p=3$ and the corresponding imaginary quadratic fields in Table \ref{table:evidence}. Let $E$ equal either $40A$ or $56A$. We will first need to show that the 3-adic Galois representations  $\rho:G_\Q \rightarrow \Gl_2(\Z_3)$, given by the action of $G_\Q$ on the $3$-adic Tate modules are surjective.  The computations and arguments turn out to be similar for both the elliptic curves $40A$ and $56A$.

To show that $\rho$ is surjective it suffices to show that the $\rho_{9}:G_\Q \rightarrow \Gl_2(\mathbb{F}_9)$, given by the action of $G_\Q$ on the $9$-division points, is surjective. This observation combines the fact that the determinant $\det(\rho):G_\Q \rightarrow \Z_3^\times$ is surjective (since it coincides with the $3$-adic cyclotomic character) and Exercise 1(b), Chapter IV, \S 3 from Serre's book \cite{MR0263823} (so that $\mathrm{Image}(\rho_9) \supset \mathrm{SL}_2(\mathbb{F}_9)$).

We will first show that $\Gal{\Q(E[9])}{\Q} \stackrel{?}{=} \Gl_2(\mathbb{F}_9)$. One can perform computations on Sage to conclude that $\Gal{\Q(E[3])}{\Q} \cong \Gl_2(\mathbb{F}_3)$. Note that
\begin{align*}
|\Gl_2(\mathbb{F}_9)| = 2^4*3^5, \qquad |\Gl_2(\mathbb{F}_3)| = 2^4*3.
\end{align*}
To show that $\Gal{\Q(E[9])}{\Q} \stackrel{?}{=} \Gl_2(\mathbb{F}_9)$, it then suffices to show that $3^5$ divides $[\Q(E[9]):\Q]$.

Let $f(t)$ in $\Q[t]$ denote the polynomial corresponding to the $9$-division points of $E$. One can use Sage and deduce that there exists an irreducible polynomial $f_{36}(t)$ in $\Q[t]$ of degree $36$ dividing $f(t)$. Let $\alpha$ denote a root of $f_{36}(t)$. So, $[\Q(\alpha):\Q]=36$. Once again, one can use Sage to deduce that there exists an irreducible polynomial $g_{27}(t)$ in the polynomial ring $\Q(\alpha)[t]$ dividing $f_{36}(t)$. So, $[\Q(\alpha,\beta):\Q(\alpha)]=27$.
Since $\Q(\alpha,\beta) \subset \Q(E[9])$, we have $2^2*3^5$ divides $[\Q(E[9]):\Q]$. These observations let us conclude that $\Gal{\Q(E[9])}{\Q} \cong \Gl_2(\mathbb{F}_9)$.

$\Q(\mu_3)$ is the unique quadratic subfield of $\Q(E[9])$. This is because $\mathrm{SL}_2(\mathbb{F}_3)$ is the unique subgroup of $\Gl_2(\mathbb{F}_3)$ with index $2$. As a result, we have the following natural isomorphisms:
\begin{align*}
\Gal{\Q(E[9])}{\Q} \cong \Gal{K(E[9])}{K} \cong \Gal{\Q(E_K[9])}{\Q} \cong  \Gl_2(\mathbb{F}_9).
\end{align*}
These observations let us conclude that $\rho$ is surjective for both $E$ and $E_K$, when $E$ equals either $40A$ and $56A$, $p$ equals $3$ and $K$ is one of the corresponding imaginary quadratic fields in Table \ref{table:evidence} . \\

When the $p$-adic Galois representation  $\rho:G_\Q \rightarrow \Gl_2(\Z_p)$, given by the action of $G_\Q$ on the $p$-adic Tate modules is surjective,  Kobayashi (Theorem 4.1 in \cite{kobayashi}) has shown that we have the following inequality of divisors in $Z^1\left(\Z_p[[\Gamma_\Cyc]]\right)$:
\begin{align} \label{kobayashiinequality}
 \Div\left(\vartheta^{\pm}_E\right) \leq  \Div\left(\theta^{\pm}_E\right), \quad  \Div\left(\vartheta^{\pm}_{E_K}\right) \leq  \Div\left(\theta^{\pm}_{E_K}\right).
\end{align}
Using equation (\ref{kobayashiinequality}), Proposition \ref{prop:cyclosp} along with the data in Table \ref{table:evidence}, our observations let us conclude that
\begin{align} \label{supp-intersect}
    \mathrm{Supp}_{\Ht=1}\left(\Sel^{++}(\Q,D_{\rhoDN{4}{2}})^\vee\right) & \bigcap \mathrm{Supp}_{\Ht=1}\left(\Sel^{+-}(\Q,D_{\rhoDN{4}{2}})^\vee\right) =\varnothing,
  \end{align}
for the elliptic curves $E=40A$ and $E=56A$, $p=3$ and the corresponding imaginary quadratic fields in Table \ref{table:evidence}. Equation (\ref{eq:surjectionspn}) now lets us conclude that the $\Iwalg$-module $\Sha^1\left(\Q,D_{\rhoDN{4}{2}}\right)^\vee$ is pseudo-null in these examples too.

\section*{Acknowledgements}
We would like to thank Ted Chinburg, Ralph Greenberg and Mahesh Kakde for answering various questions related to the work \cite{bleher2015higher}. We are also grateful towards Rob Pollack for sharing the Sage code implementing his algorithms and answering our questions related to it. We would also like to thank Francesc Castella, Ching-Li Chai, Meng Fai Lim, David Loeffler, Marc Masdeu, Florian Sprung and Chris Williams for various helpful comments during the preparation of this article. Part of this work was carried out during the second named author's visit to Universit\'e Laval in 2017 and the first named author's visit to University of Pennsylvania in 2018. We would like to thank both universities for their hospitalities. {Finally, we thank the anonymous referees for their valuable comments and suggestions on an earlier version of the manuscript, which {corrected many inaccuracies and} led to many improvements in the paper.}

\bibliographystyle{abbrv}
\bibliography{biblio}
\end{document}